\documentclass[10pt]{article}
\usepackage[left=1in,top=1in,right=1in,bottom=1in,head=.1in,nofoot]{geometry}

\usepackage{setspace,url,bm,amsmath} 

\usepackage{titlesec} 
\titlelabel{\thetitle.\quad} 
\titleformat*{\section}{\bf\Large\center}

\usepackage{graphicx} 
\usepackage{bbm}
\usepackage{latexsym}
\usepackage{caption}
\usepackage[margin=20pt]{subcaption}
\usepackage{hyperref}
\usepackage{booktabs}
\usepackage{diagbox}
\usepackage{marginnote}

\usepackage[table]{xcolor}

\newcommand{\GG}[1]{}

\usepackage{amsthm}
\usepackage{amssymb}
\usepackage{amsmath}
\usepackage{color}

\usepackage{comment}
\theoremstyle{definition}

\newtheorem*{theorem*}{Theorem}
\newtheorem{theorem}{Theorem}
\newtheorem*{rmk*}{Remark}

\newtheorem{lemma}{Lemma}

\newtheorem*{corollary*}{Corollary}

\def\pr{\mathbb{P}}
\def\CV{\text{CV}}
\def\IC{\text{IC}}
\def\UE{\text{UE}}
\def\opt{\text{opt}}
\def\PE{\text{PE}}

\usepackage{natbib} 
\bibpunct{(}{)}{;}{a}{}{,} 

\usepackage{etoolbox} 
\apptocmd{\sloppy}{\hbadness 10000\relax}{}{} 

\usepackage{color}
\usepackage{listings}

\DeclareMathOperator*{\argmin}{arg\,min}

\def\ind{\begin{picture}(9,8)
         \put(0,0){\line(1,0){9}}
         \put(3,0){\line(0,1){8}}
         \put(6,0){\line(0,1){8}}
         \end{picture}
        }

\def\Var{\mathbb{V}}

\def\I{\mathbbm{1}}

\newcommand{\V}{\mathbb{V}}
\newcommand{\E}{\mathbb{E}}
\newcommand{\trainingset}{\mathcal{T}}
\newcommand{\loss}{\mathcal{L}}

\newcommand{\asympprob}{\overset{\pr}{\asymp}}

\newcommand{\F}{{\cal F}}

\numberwithin{equation}{section}

\newcommand{\Smiley}{\odot}

\RequirePackage[normalem]{ulem}

\usepackage[sc]{mathpazo}

\usepackage[textsize=scriptsize, textwidth = 2cm, shadow]{todonotes}

\def\newchange{\color{black}}

\providecommand{\NOOP}[1]{}

\begin{document}

\allowdisplaybreaks
\singlespacing

\title{\bf \large  A Multi-resolution Theory for Approximating Infinite-$p$-Zero-$n$:
Transitional Inference, Individualized Predictions, and a World Without Bias-Variance Trade-off}
\author{Xinran Li and Xiao-Li Meng
\footnote{Xinran Li is Assistant Professor, Department of Statistics, University of Illinois, Champaign, IL 61820 (e-mail: \href{mailto:xinranli@illinois.edu}{xinranli@illinois.edu}). 
Xiao-Li Meng is Whipple V.N. Jones Professor of Statistics, Harvard University, Cambridge, MA 02138 (e-mail: \href{mailto:meng@stat.harvard.edu}{meng@stat.harvard.edu}).}
}
\date{}
\maketitle
\begin{abstract}
	{\newchange \textit{Transitional inference} is an empiricism concept, rooted and practiced in clinical medicine since ancient Greece. Knowledge and experiences gained from treating one entity (e.g., a disease or a group of patients) are applied to treat a related but distinctively different one (e.g., a similar disease or a new patient). This notion of ``transition to the similar'' renders individualized treatments an operational meaning, yet its theoretical foundation defies the familiar inductive inference framework. The uniqueness of entities  is the result of potentially an infinite number of attributes (hence $p=\infty$), which entails zero direct training sample size (i.e., $n=0$) because genuine guinea pigs do not exist.} However, the literature on wavelets and on sieve methods for non-parametric estimation suggests a principled \textit{approximation} theory for transitional inference  via a multi-resolution (MR) perspective, where we use the resolution level to index the degree of approximation to ultimate individuality \citep{meng2014trio}. MR inference seeks a primary resolution indexing an \textit{indirect} training sample, which provides enough matched attributes to increase the relevance of the results to the target individuals and yet still accumulate sufficient indirect sample sizes for robust estimation.  Theoretically, MR inference relies on an infinite-term ANOVA-type decomposition, providing an alternative way to model sparsity via the decay rate of the resolution bias as a function of the primary resolution level. Unexpectedly, this decomposition reveals a world without variance when the outcome is a  deterministic function of potentially infinitely many predictors. In this deterministic  world, the optimal resolution prefers over-fitting in the traditional sense when the resolution bias decays sufficiently rapidly. 
	Furthermore, there can be many ``descents'' in the prediction error curve, when the contributions of predictors are inhomogeneous and the ordering of their importance does not align with the order of their inclusion in prediction. These findings may hint at a deterministic approximation theory for understanding the apparently over-fitting resistant phenomenon of some over-saturated models in machine learning. 
\end{abstract}


\section{Motivation and Resolution}
\subsection{Individualized predictions and transitional inference}\label{sec:unique}
Predicting an individual's outcome, such as for personalized medicine, is an alluring proposition. Who would not want to know how a treatment would work for \textit{me} before 
such treatment even begins?
But in order to test the effectiveness of a treatment, we will need some guinea pigs. But who can approximate \textit{me}?  Someone with my genetic profiles, age, diet, exercise habit, and medical history?  But how detailed should the medical history be? What about family medical history? And how extended should my ``family" be? 

The arrival of Big Data permits us to look into such questions at deeper levels than before, but it does not make our job easier in any fundamental way. Finding a proxy population to approximate an individual is inherently an ill-defined problem from a mathematical perspective, since each of us is defined by an essentially infinite number of attributes, denoted by $p=\infty$. The implied uniqueness of ``me" then renders 
$n=0$, that is, there will never be any genuine guinea pig for me.  {\newchange Epistemologically, this need of ``transition to the similar" has been pondered by philosophers from Galen to Hume \citep[e.g., see ][]{hankinson1987causes,hankinson1995growth}. For example, Galen, a physician and philosopher in the Roman Empire, wrote (see \citet{hankinson1987causes}): 
\begin{quote}
``In cases in which there is no history, or in which there is none of sufficient similarity, there is not much hope. And the same thing is true in the case of transference of one remedy from one ailment to another similar to it: one has a greater or smaller basis for expectation of success in proportion to the increase or decrease in similarity of the ailment, whether or not history is involved. And the same goes for the transference from one part of the body to another part: expectation of success varies in direct proportion to the similarity." \end{quote}

Galen's framing is essentially a statistical one, with a nice blend of Bayesian (the reliance on history) and frequentist (the emphasis on proportions regardless of history), albeit long before any of these qualifying terms was invented. Perhaps it is a surprise then that, to the best of our knowledge, there is no statistical theory for this kind of \textit{transitional inference} \citep{hankinson1995growth}.  We surmise that this absence is largely due to the fact that transitional inference goes outside of our traditional inductive framework since it is not about inferring a population from samples of individuals, but rather about predicting individuals' outcomes by learning from a proxy population.  The notion of \textit{similarity}, central to transitional inference, is also a challenging one to metricize in general.}

However, the concept of multi-resolution (MR) analysis in engineering and applied mathematics, such as wavelets \cite[see][]{meyer1993wavelets,daubechies1992ten}, turns out to be rather useful for establishing such a theoretical framework. For wavelets, variations in data are first decomposed according to their resolution levels. For image data, the resolution level is the pixel resolution as we ordinarily define, and the concept of multi-resolution can be easily visualized by the common practice of zooming in and out when taking pictures.  Zooming too much or too little both would result in losing seeing the big picture, figuratively and literally. Our central task is then to identify a suitable \textit{primary resolution} to separate signals (i.e., lower-resolution wavelet coefficients) from noise (i.e., higher-resolution wavelet coefficients); see \citet{donoho1995wavelet} and especially \citet{johnstone2011gaussian} for a survey. 
The choice of primary resolution thus determines the unit of our inference, that is, the degree of individualization. 
The search for the primary resolution is generally a quest for an age-old bias-variance trade-off: estimating more precisely a less relevant individual assessment versus estimating less precisely a more relevant one.

Because the MR framework permits the resolution level to be potentially infinite, it can also be viewed as the predictive counterpart of the estimation method of sieves for dealing with infinite-dimension models. In order to reveal as early as possible what this framework can offer, we follow a reviewer's suggestion to defer a literature review and comparison with the standard large-$p$-small-$n$ framework to the end of our article.

\subsection{A fundamental resolution decomposition}\label{sec:identity}
To set up our MR framework, we consider an outcome variable $Y$ sharing the same probability space $(\Omega, \F, P)$ as an information filtration $\{\F_r, r=0, 1, ..., \}$, where $\F_{r-1}\subset \F_{r}$,  and $r$  indexes  our resolution level.  Here $\F_0$ corresponds to a population of interest (e.g., those who are infected by a 
certain virus) from where target individuals come, and $\F_{\infty}=\cup_{r=0}^\infty \F_r$ permits us to define (unique) individuality. For example, $\F_r$ is the $\sigma$-field generated by covariates $\{X_0, X_1, \ldots, X_r\}$, and hence determining  the primary resolution is the same as determining how many covariates should be used for predicting $Y$ for a given information filtration (see Section~\ref{sec:ordering} for the issue of ordering the covariates).  
Let $\E(\cdot)$ and $\V(\cdot)$ denote mean and variance, respectively.
Denote $\mu_r=\E[Y|\F_r]$ and  $\sigma_r^2=\V[Y|\F_r]$ for all $r$'s, including $r=0$ and $r=\infty$ (and assume these are well defined). Then by repeatedly applying the iterative law $\V[Y|\F_r]=\E[\V( Y|\F_{s})|\F_r]+\V[\E( Y|\F_{s})|\F_{r}]$,
where $s>r$,  we have the usual ANOVA decomposition \citep{meng2014trio}, 
\begin{equation}\label{eq:keyi}
\sigma^2_r=\E[\sigma^2_{\infty}|\F_r]+\sum_{i=r}^{\infty}\E[(\mu_{i+1}-\mu_i)^2|\F_r], \quad {\rm for\ any }\  r\ge 0.
\end{equation}

Decomposition \eqref{eq:keyi} reminds us that the usual dichotomy between  \textit{variance}, as a measure of random variations,  and \textit{bias}, as a measure of systematic differences, is an artificial one, except possibly at the infinite  resolution level. That is, the variance at any particular resolution level is merely the accumulation of all the (squared consecutive) systematic differences,  i.e., biases, at higher resolution levels, plus $\sigma^2_\infty$, the \textit{intrinsic variance}.  Conceptually $\sigma^2_\infty$ cannot be ascertained from any empirical data, because we can never be sure whether the residual variance from whatever model we fit is due to $\sigma^2_\infty$ or to a limitation of our always finite amount of data. It therefore seems inconsequential to set $\sigma_\infty^2=0$ 
since we can never prove it false. This proposition should be particularly acceptable to those who believe that the world is ultimately deterministic once all its operating mechanisms are measured and understood 
\citep[e.g., see][]{peat2002certainty}. 

However, as we shall reveal in this article, whether or not to set $\sigma_\infty^2$ to zero has profound implications on the bias-variance trade-off phenomenon. To the best of our knowledge, the statistical literature has not investigated this phenomenon for chaotic dynamic systems \citep[e.g.,][]{devaney2018introduction}, since when $\sigma_\infty^2=0$, the setup here enters the realm of deterministic but potentially chaotic systems. The corresponding findings therefore may be counter-intuitive (initially) to statisticians, but they might provide a bridge to the growing literature in machine learning that casts doubts on the applicability of bias-variance trade-off, especially the literature surrounding the phenomenon of ``double descent" \citep[e.g.,][]{belkin2019reconciling, belkin2019two, hastie2019surprises,nakkiran2019deep}, which we shall explain and extend to ``multiple descents" later in this article.

Regardless of how we treat $\sigma^2_\infty$, declaring that a resolution level $R$ is our primary resolution implies that all the information conveyed by variations at resolution levels higher than $R$ can be effectively ignored when predicting $Y$. The MR formulation therefore permits us to quantify the degree of individualization, and to be explicit about the two contributing factors of our overall prediction error:  (I) the resolution bias due to choosing a finite $R$; and (II) the estimation error at the given resolution $R$. 
The MR framework therefore integrates the model selection step (I) with the model estimation step (II), and hence it does not need to treat the issue of selection post-hoc \citep[e.g.,][]{berk2013,lee2016exact,tibshirani2016exact}. Furthermore, since the filtration $\{\F_r, r=0, 1, \ldots \}$ forms  a cumulative ``information basis", the choice of optimal $R_n$ for a given data set with size $n$  is in the same spirit as finding a \textit{sparse representation} in wavelets, for which there is a large literature \cite[see][]{poggio1998sparse,donoho2003optimally}, though here perhaps it is more appropriate to term it as \textit{parsimonious representation}.

\subsection{Time-honored intuitions, and timely new insights?}
Our findings confirm some time-honored intuitions and build new ones. Specifically, in Section \ref{sec:multi_resol_pred} we first decompose  
 the total prediction error into three components:
 the ultimate risk, the resolution bias and the estimation error. 
We then provide an overview and highlight on how the optimal resolution depends on the decay rates of the resolution bias and the corresponding estimation error under a particular ordering of covariates, respectively, in the stochastic world (i.e.,  $\sigma^2_\infty>0$) and deterministic world (i.e., $\sigma^2_\infty=0$). Section~\ref{sec:multi_resol_pred} concludes with some theoretical insights on the issue of ordering the covariates.  

Sections~\ref{sec:theory_linear} and \ref{sec:theory_tree} then establish our general results with an infinite number of continuous and categorical predictors, and illustrates them with linear regression and tree regression, respectively. In particular, in Sections~\ref{sec:linear_zero_tau} and \ref{sec:determin_categorial}, we  report some intriguing findings when $\sigma_\infty^2=0$ respectively for these two regression models. In this world without variance, the optimal resolution may rightly prefer the direction of over-fitting in the traditional sense; indeed the optimal resolution level can even approach infinity. But this preference does not violate the time-honored bias-variance trade-off principle because, without variance, the optimal trade-off may have to put all its eggs in the basket of bias. 

We also find that the predictive error curve can exhibit double descents or even arbitrarily many descents without ever entering the over-parameterized realm.  These findings might provide a new angle to investigate very flexible and saturated models, such as deep learning networks, to understand their seemingly magical ability to resist over-fitting. That is, with a huge amount of data, it is conceivable that an exceedingly rich and flexible deterministic model class can learn to practically exhaust all patterns detectable with reasonable chances in reality (which can be far fewer than in theory). In such cases, we would not need  $\sigma_\infty^2>0$ to absorb the imperfection of the model, effectively rendering it a deterministic system,
a system that prefers ``over-fitting" in the traditional sense. This is also explored empirically in Section \ref{sec:finite}, 
where we summarize
a simulation study with linear models that investigates the practicality of
the MR approach that employs cross validation and other methods for selecting the primary resolution in practice.
The details of the study, as well as all the technical proofs in our article, are deferred to the Appendices. Section~\ref{sec:main_practical} completes our exploration by making connections to relevant literature and discussing further work.

\section{A Multi-Resolution Framework}\label{sec:multi_resol_pred}

\subsection{Prediction with potentially infinitely many predictors}

To start, let 
$\odot$ be a member of a target population, which can be as small as a single individual,
and $Y(\odot)$ be a univariate response from $\odot$, which can be discrete (e.g., a treatment success indicator) or continuous (e.g., the change of the cholesterol level due to a treatment). Typically the investigators have some prior knowledge about which set of the individual's attributes play more critical roles in determining $Y$. But,  philosophically and practically,  no one can be certain about what constitutes the complete set of relevant predictors. Statistically we can model such a situation by requiring the distribution of $Y(\odot)$ to depend on potentially infinitely many attributes of $\odot$, denoted by $\vec{\bm{X}}_{\infty}(\odot)=\{X_0(\odot), X_1(\odot), X_2(\odot), \ldots \}$.  In reality we can never observe infinitely many covariates, but the arrival of the digital age has created many situations where we have far more predictors than the sample size. Our job is to seek a small subset of the predictors of the outcome with accuracy that makes our prediction useful. 

We use $f_{\odot}$ to denote the joint probability mass/density function of the response and covariates for the target individual $\odot$. 
To learn about  $f_{\odot}$, especially the dependence of $Y(\odot)$ on $\vec{\bm{X}}_{\infty}(\odot)$, 
we need to collect a training set 
$\trainingset_n=\{(y_i,\vec{\bm{x}}_{i \infty}): i=1,2,\ldots,n\}$,
which are (assumed to be) independent and identically distributed (i.i.d.) samples from a training (proxy) population.  Clearly the phrase ``training" implies that we need some assumptions to link $\trainingset_n$ to the target population. The ideal assumption of course is that 
$f_{\odot}$ equals the joint probability mass/density function $f$
of $(Y, \vec{\bm{X}}_\infty)$ for the training population. Whereas all attempts should be made to mimic the target population when we form the training population, it is wise to  permit our framework 
sufficient flexibility 
to admit cases where $f$ may differ from $f_{\odot}$ but in an approximately known way. Mathematically, this flexibility can be handled by introducing a weight function  
\begin{equation}\label{eq:weight}
w_{\odot}(Y, \vec{\bm{X}}_{\infty})= \frac{f_{\odot}(Y, \vec{\bm{X}}_{\infty})}{f(Y, \vec{\bm{X}}_{\infty})}=\frac{f_{\odot}(Y| \vec{\bm{X}}_{\infty})}{f(Y|\vec{\bm{X}}_{\infty})}\frac{f_{\odot}(\vec{\bm{X}}_{\infty})}{f(\vec{\bm{X}}_{\infty})}. 
\end{equation}
Normally it is almost inevitable to assume $f_{\odot}(Y| \vec{\bm{X}}_{\infty}) \approx f(Y|\vec{\bm{X}}_{\infty})$, that is, the (stochastic) relationships between the outcome and the predictors for the target population and the training population must be approximately the same, because otherwise our selection of the training sample is a very poor one. Consequently, (\ref{eq:weight}) implies 
$w_{\odot}(Y, \vec{\bm{X}}_{\infty})\approx f_{\odot}(\vec{\bm{X}}_{\infty})/f(\vec{\bm{X}}_{\infty})$, which is easier to estimate since it merely involves adjusting the marginal distribution of the $\vec{\bm{X}}_\infty$, known as a ``covariate shift" in the  literature \citep[see, e.g.,][]{learningdiff2007, sugiyama2012machine}. 
{\newchange However, when $\odot$ is indeed a single individual or beyond the support of the training population, the weight $w_{\odot}(Y, \vec{\bm{X}}_{\infty})$ is not defined without lowering the resolution level for evaluation; see \citet{meng2020}. We leave the choice of weights for a future study, as our focus in this article is on the choice of optimal resolutions with given weight functions.}

To avoid confusion, 
we use $\E_{\odot}$ and $\E$  to denote the expectations over the target and the training populations respectively.  To evaluate the prediction performance of a prediction function $\hat{y}(\vec{\bm{X}}_\infty)$, we can adopt a loss function $\loss(y, \hat{y})$, which is problem-dependent.  Clearly, 
we can minimize the expected loss $\E_{\odot}[\loss(Y, \hat{y}(\vec{\bm{X}}_\infty ))]$ via minimizing
$\E[\loss_{\odot}(Y, \hat{y}(\vec{\bm{X}}_\infty ))]$,
where $\loss_{\odot}(Y, \hat{y}(\vec{\bm{X}}_\infty )) \equiv \loss(Y, \hat{y}(\vec{\bm{X}}_\infty )) w_{\odot}(Y, \vec{\bm{X}}_\infty)$; 
the subscript $\odot$ indicates its dependence on the utility of prediction and the target population of interest. 
With this setup, we proceed as follows.
At each resolution $r$, 
we restrict our prediction to a family of functions $\{g(\vec{\bm{x}}_{r}; \bm{\theta}_r)\}$, where $\vec{\bm{x}}_{r}=(x_0, \ldots, x_r)$. 
For notational simplicity, we suppress the explicit dependence of $g(\cdot)$ on $r$, 
but rather use the inputs $\vec{\bm{x}}_{r}$ and $\bm{\theta}_r$ to emphasize such dependence implicitly.
Note that $\bm{\theta}_r$ denotes a generic parameter whose dimension can vary with $r$. For example, $\dim(\bm{\theta}_r) = 
\binom{r+2}{2}$
if $g(\vec{\bm{x}}_{r};\bm{\theta}_r)$ is a linear function of covariates up to resolution $r$ and of all their 
quadratic terms and pairwise interactions.
Generally, we will choose $g(\cdot)$ such that the family of prediction functions becomes richer as resolution increases. 
That is, for any $r< r'$, 
any prediction function $g(\vec{\bm{x}}_{r}; \bm{\theta}_r)$ at resolution $r$, viewed as a function of $\vec{\bm{x}}_{r'}$, belongs to the family of prediction functions at resolution $r'$. At each resolution $r$,
the optimal prediction is then $g(\vec{\bm{x}}_{r}; \bm{\theta}_r^*)$, with 
$
\bm{\theta}^*_r \equiv \argmin_{\bm{\theta}_r} 
\E[\loss_{\odot}( Y, g(\vec{\bm{X}}_{r}; \bm{\theta}_r))]. 
$
A usual estimator for $\bm{\theta}^*_r$ 
is obtained by minimizing the empirical risk:
$
\hat{\bm{\theta}}_r \equiv
\arg\min_{\bm{\theta}_r}
\sum_{i=1}^{n}
\loss_{\odot}(y_i, g(\vec{\bm{x}}_{ir};\bm{\theta}_r)).
$
Hence, once we choose the primary resolution $R$, we predict $Y$ by $g(\vec{\bm{x}}_{R}; \hat{\bm{\theta}}_R)$ for an individual with covariate $\bm{x}_\infty$, and estimate 
the prediction error 
$
\E[\loss_{\odot}( Y, g(\vec{\bm{X}}_{R}; {\bm{\hat\theta}}_R) )]
$ by  the empirical risk $n^{-1}\sum_{i=1}^{n}
\loss_{\odot}(y_i, g(\vec{\bm{x}}_{iR};\bm{\hat\theta}_R))
$, or by cross-validation. 

\subsection{A trio decomposition of the prediction error}\label{sec:decomposition}

To better understand the prediction error at a resolution $R$, we decompose
$\E[\loss_{\odot}( Y, g(\vec{\bm{X}}_{R}; \hat{\bm{\theta}}_R) )]$ into three parts: 
the ultimate risk, the resolution bias at resolution $R$, and the estimation error at resolution $R$. 
{\newchange The \textit{ultimate risk} is
$\tau^2 \equiv \E[\loss_{\odot}(Y, g(\vec{\bm{X}}_\infty; \bm{\theta}_{\infty}^*))],$
which depends on the families of functions used for prediction. Specifically, it has two sources, one due to model misspecification and the other due to the \textit{intrinsic variation} at the infinite resolution,
i.e., $f(Y|\vec{\bm{X}}_{\infty})$. That is, the intrinsic variance $\sigma^2_{\infty}=\Var(Y|\vec{\bm{X}}_{\infty})$ can be positive (or even infinity) in a stochastic world.} 
The \textit{resolution bias} at resolution $R$ 
then is
\begin{align*}
A(R) & = 
\sum_{r=R+1}^\infty
\left\{
\E[\loss_{\odot}(Y, g(\vec{\bm{X}}_{r-1};\bm{\theta}_{r-1}^*))]
-
\E[\loss_{\odot}(Y, g(\vec{\bm{X}}_{r};\bm{\theta}_{r}^*))]
\right\}.
\end{align*}
When the family of prediction functions becomes richer as resolution increases, 
$A(R)$ is non-increasing in $R$ and approaches zero as $R\rightarrow\infty$, i.e., 
$\lim_{R\rightarrow \infty} A(R) =  0.$ 
Finally, 
the \textit{estimation error} at resolution $R$, 
\begin{align*}
\varepsilon(R,\trainingset_n) = 
\E[\loss_{\odot}(Y, g(\vec{\bm{X}}_{R};\hat{\bm{\theta}}_R))]
- 
\E[\loss_{\odot}(Y, g(\vec{\bm{X}}_{R};\bm{\theta}_R^*))],
\end{align*}
is non-negative by the optimality of $\bm{\theta}_R^*.$
From the above,  the prediction error at resolution $R$ using training set $\trainingset_n$ can be decomposed as
\begin{align}\label{eq:decomposition}
\E[\loss_{\odot}( Y, g(\vec{\bm{X}}_{R}; \hat{\bm{\theta}}_R))] & = 
\tau^2
+ A(R) + \varepsilon(R,\trainingset_n).
\end{align}
As we shall show shortly, theoretically, we can gain good insight by considering the averaged version of this decomposition, that is, 
\begin{align}\label{eq:decomposition_average_training}
\E_n\left[\E[\loss_{\odot}( Y, g(\vec{\bm{X}}_{R}; \hat{\bm{\theta}}_R))]\right] & = 
\tau^2
+ A(R) +  \varepsilon(R, n),  
\end{align}
where, with slight abuse of notation, $ \varepsilon(R, n)=\E_n [\varepsilon(R,\trainingset_n )]$, and $\E_n$ denotes the expectation over all training sets of size $n$. 

{\newchange It is worthy noting that \eqref{eq:decomposition} is an extension of the ANOVA decomposition \eqref{eq:keyi} in expectation,
with \eqref{eq:keyi} being a special case with $\loss_{\Smiley}(y,\hat{y}) = (y-\hat{y})^2$ and $g(\bm{X}_{\vec{r}}; \bm{\theta}_r^*) = \E(Y \mid \vec{\bm{X}}_r)$ for $r\ge 1$, i.e., the prediction functions are correctly specified.
Under this special case, the ultimate risk $\tau^2$ reduces to $\E(\sigma^2_{\infty})$. We remark that in general
$\tau^2 \ge \E(\sigma^2_{\infty})$, 
with equality holds when we correctly specified the prediction functions.
Because $\sigma^2_{\infty} \ge 0$, a zero $\tau^2$ then must imply $\sigma^2_{\infty}=0$ (almost surely), i.e., a deterministic world without variance. Here, 
as in \eqref{eq:keyi}, $\sigma^2_r=\Var[Y \mid \bm{X}_{\vec{r}}]$ and $\mu_r=\E[Y \mid \bm{X}_{\vec{r}}] = g(\bm{X}_{\vec{r}}; \bm{\theta}_r^*)$, which is estimated by $\hat{\mu}_r = g(\bm{X}_{\vec{r}}; \hat{\bm{\theta}}_r)$. The resolution bias at resolution $R$ reduces to $\sum_{r=R}^{\infty}[\E(\sigma^2_r) - \E(\sigma^2_{r+1})] = \sum_{r=R}^{\infty}\E (\mu_{r+1}-\mu_r)^2$, and 
the estimation error to $\E(\hat{\mu}_R - \mu_R)^2$. 
Consequently, \eqref{eq:decomposition} reduces to
\begin{align}
& \E(\sigma_R^2) + \E(\hat{\mu}_R - \mu_R)^2  =\E (\sigma^2_{\infty}) +\sum_{r=R}^{\infty}\E (\mu_{r+1}-\mu_r)^2 +
\E(\hat{\mu}_R - \mu_R)^2, \label{eq:quad} 
\end{align}
which is equivalent to \eqref{eq:keyi} by further averaging over $\F_r$ (i.e., the conditioning in \eqref{eq:keyi}).}

Because  in  \eqref{eq:decomposition} and \eqref{eq:decomposition_average_training} 
the ultimate risk 
is not affected by the resolution (under the assumption that the function form is 
the same at the infinite resolution),
for any training set $\trainingset_n$, 
the optimal primary resolution that minimizes the prediction error in \eqref{eq:decomposition} is 
\begin{align*}
R_{\trainingset_n,\text{opt}}  & = \arg\min_{R} \E[\loss_{\odot}( Y, g(\vec{\bm{X}}_{R}; \hat{\bm{\theta}}_R) )]
=
\arg\min_{R}\left[
A(R) + \varepsilon(R,\trainingset_n)
\right].
\end{align*}
Similarly,
the optimal primary resolution that minimizes the prediction error in \eqref{eq:decomposition_average_training} is
\begin{align*}
R_{n,\text{opt}}  & = \arg\min_{R} \E_n\left[ \E[\loss_{\odot}( Y, g(\vec{\bm{X}}_{R}; \hat{\bm{\theta}}_R) )]\right]
=
\arg\min_{R}\left[
A(R) + \varepsilon(R, n)
\right].
\end{align*}
Studying $R_{\trainingset_n,\text{opt}}$ or $R_{n,\text{opt}}$ for a particular training set $\trainingset_n$ or a particular size $n$ is generally 
difficult. We therefore resort to the usual asymptotic strategy. 
That is, as $n$ goes to infinity, we seek a sequence $\{R_n\}_{n=1}^\infty$ such that $A(R_n)+ \varepsilon(R_n,\trainingset_n)$ or 
$A(R_n) + \varepsilon(R_n, n)$
converges to zero 
(in probability) 
as fast as possible. We will adopt the notation $a_n \asymp b_n$  if two sequences $\{a_n\}$ and $\{b_n\}$ satisfy $a_n = O(b_n)$ and $b_n = O(a_n)$, and similarly, $\tilde{a}_n \overset{\pr}{\asymp} \tilde{b}_n$ if random sequences
$\{\tilde{a}_n\}$ and $\{\tilde{b}_n\}$ satisfy $\tilde{a}_n = O_{\pr}(\tilde{b}_n)$ and $\tilde{b}_n = O_{\pr}(\tilde{a}_n)$, using the usual definition of $O_{\pr}$.
We also use the notation $a_n \gtrsim b_n$ for $b_n=O(a_n)$.

\subsection{Optimal resolution and learning rate in the stochastic world} \label{sec:rate_optimal_resolution}

Intuitively, there must be a trade-off in determining the optimal $R_n$. To control the resolution bias $A(R_n)$,  we desire large $R_n$ because of the monotonically decreasing nature of $A(R)$. For $A(R_n)$, we will consider four scenarios, representing four different levels of sparsity. 
However, to control the estimation error, we want small $R_n$ to reduce the number of model parameters to be estimated. When the intrinsic variance $\sigma_\infty^2>0$, under some regularity conditions (e.g., our estimation methods are efficient), we have the usual
$\varepsilon(R_n, n) \asymp \dim (\bm{\theta}_{R_n})/n$ asymptotics. Hence we need $\dim (\bm{\theta}_{R_n})=o(n)$ to ensure $\varepsilon(R_n,n)$ converges to zero as $n \rightarrow\infty$.

\begin{table}
	\centering
	\caption{Rate-optimal $R_n$ and minimal error $L_n\equiv A(R_n) + \varepsilon(R_n,n)$ in a stochastic world. All $c_n$'s are of $O(1)$ but satisfy different constraints as specified in Theorem~\ref{th:cont} (Section~\ref{sec:gen_result_linear}) and Theorem~\ref{th:disc} (Section~\ref{sec:general_regression_tree}).}
	\label{tab:optimal_rate}
	\resizebox{\columnwidth}{!}{%
	\begin{tabular}{|l|c|c|c|c|}
		\hline
		\diagbox{$\varepsilon(r,n)$}{$A(r)$} & $\substack{{\rm Hard\ Thresholding}\\ 1_{\{r < r_0\}}}$ 
		&  $\substack{{\rm Exponential\ Decay}\\ e^{-\xi r} \ (\xi>0) }$ & $\substack{{\rm Polynomial\ Decay} \\ r^{-\xi} \ (\xi>0) }$ & $\substack{{\rm Logarithmic\ Decay}\\
		 \log^{-\xi}(r) \ (\xi>0) }$
		\\
		\hline
	Polynomial in $r$ & $R_n \asymp c_n \ge r_0$ & $c_n\log{n}$ & $c_n n^{1/(\xi+\alpha)}$ & $\frac{c_nn^{1/\alpha}}{\log^{\xi/\alpha}(n)}$\\ $r^{\alpha}/{n}$ 
	$(\alpha > 0)$ 	& $L_n\asymp 1/n$ & $\log^{\alpha}(n)/n$ & $n^{-\xi/(\xi+\alpha)}$ &  $[\log(n)]^{-\xi}$ \\
		\hline
Exponential in $r$ & $R_n \asymp c_n \ge r_0$	& $\frac{\log{n}+\log c_n}{\xi+\log\alpha}$ & $c_n\log(n)$ & $c_n\log(n)^{}$\\ $\alpha^r/n$ 
	$(\alpha>1)$ & $L_n\asymp 1/n$	& $ n^{-\xi/(\xi+\log\alpha)}$ & {$[\log(n)]^{-\xi}$} & $[\log\log(n)]^{-\xi}$\\
		\hline
	\end{tabular}%
	}
\end{table}

Table~\ref{tab:optimal_rate} provides a high-level preview of the general asymptotic results under the above setting, with four (common) choices of the decay rate for $A(r)$. What do these asymptotic results tell us? First, the hard-thresholding cases correspond to the classical parametric setting, with a fixed number ($r_0$) of predictors. Hence, as long as our resolution level $R_n$ exceeds $r_0$ (arbitrarily often), we will reach the classical $n^{-1}$ error rate, excluding the ultimate risk (which includes the intrinsic variation). 

Second, the rate-optimal resolution $R_n$---and hence the minimal prediction error---depends critically on both the decay rate $A(r)$ and estimation error $\varepsilon(r, n)$. When $\varepsilon(r, n)$ grows polynomially with the resolution level (e.g., the continuous covariates cases), we can still practically achieve the $n^{-1}$ rate when $A(r)$ decays exponentially, because the price we pay is merely a $\log^\alpha(n)$ term. However, if $\varepsilon(r, n)$ grows exponentially (e.g., with discrete covariates), then although $R_n$ is still practically of $\log(n)$ type,  the parametric error rate $n^{-1}$ is no longer achievable even if $A(r)$ decays exponentially. Instead, we can achieve only a non-parametric like error rate in the form of $n^{-\xi/(\xi+\log\alpha)}$, which reduces to $n^{-1}$ only if the decay rate parameter $\xi$ for $A(r)$ goes to infinity. 

Third, when $A(r)$ decays polynomially, $R_n$ takes on different rate forms depending on how the estimation error varies with the resolution level $r$,  that is, (A) polynomial in $n$ for polynomial estimation error versus (B) $\log(n)$ for exponential estimation error. More importantly, the difference in the corresponding minimal prediction errors tells us that in case (A), the individualized prediction and learning rate is slow but still practical. However, case (B) belongs to the situation where the individualized learning rate is too slow to be useful.  The same is true once the decay rate is logarithmic because then the prediction error rate is no better than that of case (B); see the last column of Table~\ref{tab:optimal_rate}. Therefore, among the eight scenarios in Table~\ref{tab:optimal_rate}, only the first five (counting first top to bottom then
left to right) of these permit practical individualized learning.

Here we give a side note on the asymptotic expression in Table \ref{tab:optimal_rate}. 
First, a more rigorous expression for the polynomial estimation error is $\varepsilon(r,n) \asymp \max\{r^\alpha,1\}/n$. We simply use $r^\alpha/n$ not only for descriptive convenience, but also since $r\ge 1$ is required for achieving rate optimal prediction when $A(0)>0$. Second, the decay rates for resolution biases, e.g., $r^{-\xi}$ and $\log^{-\xi}(r)$, may be well-defined only for $r$ larger than a certain value. 
Whenever such a quantity is not prescribed, we can view it as a finite positive constant. 
Again, this complication has little relevance for our asymptotic theory for the rate-optimal resolution, which must go to infinity as $n\rightarrow \infty$ when $A(r)>0$ for any finite $r$.

\subsection{Optimal resolution and learning rate in the deterministic world} \label{sec:rate_optimal_resolution_zero}
\begin{table}
\centering
\caption{Rate-optimal $R_n$ and minimal error $L_n\equiv \PE_n$ in a deterministic world. All $c_n$'s are of $O(1)$ but satisfy different constraints as specified in Theorem~\ref{th:contzero} (Section~\ref{sec:linear_zero_tau}), Theorems~\ref{thm:binary_zero_tau} and \ref{thm:lower_bound_binary_tau2_zero} (Section~\ref{sec:determin_categorial}). Note:  like in Table~\ref{tab:optimal_rate}, $\xi>0$. In some cases, the forms of rate-optimal $R_n$ are only sufficient but not necessary for achieving the optimal rate.
}\label{tab:optimal_rate_zero}
\resizebox{\columnwidth}{!}{%
\begin{tabular}{|l|c|c|c|c|}
\hline
		\diagbox{Model}{$A(r)$} & $\substack{{\rm Hard\ Thresholding}\\ 1_{\{r < r_0\}}}$ 
		&  $\substack{{\rm Exponential\ Decay}\\ e^{-\xi r} }$ & $\substack{{\rm Polynomial\ Decay} \\ r^{-\xi} }$ & $\substack{{\rm Logarithmic\ Decay}\\
		 \log^{-\xi}(r) } $
		\\
		\hline
{Linear} & $n-3 \ge R_n \ge  r_0$  & {$R_n=n - c_n$} & {$c_n n$} & {$c_n n^k$, $k\in (0,1]$}
\\
{regression} & {$L_n=0$} & {$L_n\asymp n e^{-\xi n}$} & {$n^{-\xi}$} & {$[\log(n)]^{-\xi}$} \\
\hline Regression\ tree& $R_n \ge r_0$ & {$
\begin{cases}
\gtrsim c_n \log (n), & \xi > \log(M) \\
= c_n \log (n), & \xi = \log(M) \\
= c_n\log (n), &  \xi <\log(M)
\end{cases}
$} & {\small$c_n\log^{}(n)$} & {\small$c_n\log(n)^{}$}
\\ with \ predictors  &{} &{} &{} &
\\ 
$\substack{X_i's\ \text{are\ i.i.d.} \\ {\rm Uniform}\{1, \ldots, M\}}$
& $L_n\asymp (1-M^{-r_0})^n$
&
$\begin{cases}
\lesssim n^{-1}, & \xi > \log(M)\\
\lesssim n^{-1}\log(n), &\xi = \log(M)\\
n^{-\xi/\log(M)}, & \xi < \log(M)
\end{cases}
$
 & $[\log(n)]^{-\xi}$& $[\log\log(n)]^{-\xi}$\\
\hline
\end{tabular}
}
\end{table}

The case with $\sigma^2_\infty=0$ or more precisely zero ultimate risk, however, behaves rather differently, and will be studied in Sections \ref{sec:linear_zero_tau} and \ref{sec:determin_categorial} for two popular models. We restrict to specific models because we have not been able to obtain general results parallel to those in Table~\ref{tab:optimal_rate}. But even with these specific results, we already see asymptotic behaviors, as revealed in Table~\ref{tab:optimal_rate_zero}, that are quite different from those in Table~\ref{tab:optimal_rate}. The trivial ones are for hard thresholding, where for the linear model, as long as 
sample sizes are large enough
to solve the linear system, we will have zero error.  Similarly, for the regression tree case, where the only possible error is when no exact match of the target case exists with respect to the $r_0$ important predictors in the training sample of size $n$.
The probability of this occurring is exactly  $(1-M^{-r_0})^n$ under our model assumption that all predictors $X_j$'s are independently and identically distributed as  uniform on $\{1, 2, \ldots, M\}$, a mathematically convenient assumption that permits us to obtain analytical results. 

The more interesting cases are when $A(r)$ decays exponentially, which permits the optimal $R_n$ to be infinity; {\newchange for example, for the regression tree model, when the resolution bias decays exponentially with $\xi>\log(M)$, choosing $R_n = \infty$ can lead to prediction error no worse than order $n^{-1}$.} 
That is, we are not worried about over-fitting because the benefit from exact matching outweighs the imprecision in solving, say, the linear system. This phenomenon does not occur when we restrict ourselves to statistical models with a finite number of predictors, which would force us to adopt an error term to capture the unexplained residual variations in the outcome variable. With an infinite number of predictors, there is at least a theoretical possibility that  collectively they can explain all the variations in the outcome variable. There is no free lunch, however, as this full-explanatory power requires that the predictive model is specified correctly. Nevertheless, the discovery of this phenomena by permitting models with an infinite number of predictors should remind us of the value of exploring this line of thinking, as it might lead to alternative insights into why certain highly saturated black box models (e.g., deep learning networks) can have a seemingly over-fitting resistant nature. We shall explore this line of thinking in Section~\ref{sec:double_descent}, where we show how easily we can go beyond the intriguing ``double descents" phenomenon 
\citep[e.g.,][]{belkin2019reconciling,hastie2019surprises,nakkiran2019deep} in the deterministic world with infinitely many predictors, without even having to actually enter the realm of over-fitting.

\subsection{The impact of ordering}\label{sec:ordering}

So far we have assumed that the order of the covariates 
is pre-determined. In reality, the investigators may have some ``low resolution'' knowledge of the importance of \textit{groups} of the covariates (e.g., age and gender are typically among the predictors to be included in predicting health outcome). However, they often do not possess the refined knowledge to specify the exact order of the covariates in terms of their predictive power (if they did, the problem would be much easier). Mathematically, when the resolution levels change, we can change all the covariates included in the model. But to utilize our partial knowledge, however imprecise, we wish to investigate the dependence of  prediction error on the order of the covariates, and in particular the degree of mis-ordering that can fundamentally alter the prediction error rate.  That is, how much misspecification of the order can we 
tolerate before it really matters?
Assume that the family of prediction functions becomes richer as resolution increases, and they are invariant to the ordering of the covariates, i.e., 
for any $r$ and any permutation $\pi$ of $\{0, 1, 2, \ldots, r\}$, the families of functions $\{g(\vec{\bm{x}}_r; \bm{\theta}_r)\}$ and $\{g(\vec{\bm{x}}_{\pi(r)}; \bm{\theta}_r)\}$ are the same, where $\vec{\bm{x}}_{\pi(r)} \equiv (x_{\pi(0)}, x_{\pi(1)}, \ldots, x_{\pi(r)})$.
Consequently,
the ultimate risk $
\tau^2=\E[\loss_{\odot}(Y, g(\vec{\bm{X}}_\infty; \bm{\theta}_{\infty}^*))]$
is invariant to the ordering of covariates. This is most clearly seen under squared loss and correctly specified conditional mean function, where $\tau^2=\E(\sigma_\infty^2)$, as discussed prior to arriving at \eqref{eq:quad}.
Below we will focus on the resolution bias and estimation error. 

We begin by considering a specific ordering of the covariates, $\{X_0, X_1, X_2,\ldots\}$, identified with its resolution bias $A(\cdot)$,  estimation error $\varepsilon(\cdot, n)$, and rate-optimal resolution  $R_n$. 
Let $A'$, $\varepsilon'$ and $R_n'$ be their counterparts under a new ordering 
$\{X'_{0}, X'_{1}, \ldots\}$. 
Generally, the estimation errors $\varepsilon(r_n, n)$ and $\varepsilon'(r_n, n)$ under both orderings (i.e., $r_n = R_n$ or $R'_n$) are of the same order after some proper scaling of ``unit noise'',
because they involve estimation for the same number of parameters. 
In the following discussion, we assume $\varepsilon(r_n, n)/[A(r_n) + \tau^2] \asymp \varepsilon'(r_n, n)/[A'(r_n) + \tau^2]$, which reduces to $\varepsilon(r_n, n) \asymp \varepsilon'(r_n, n)$ when $\tau^2 > 0$. 
As shown later, this assumption is motivated by the linear regression and tree regression models. 
Then, a sufficient condition for the new order to achieve the optimal rate under the original ordering is that $A’(R_n) = O(A(R_n))$. 
This condition should be intuitive because all it requires is that the new ordering does not delay the inclusion of  covariates which are considered important by the original ordering. 

Suppose now that every covariate matters, in the sense that the resolution bias at any finite resolution is positive, regardless of the ordering of covariates. 
From Section \ref{sec:rate_optimal_resolution}, 
for any ordering of covariates, 
its optimal primary resolution must go to infinity as $n\rightarrow \infty$; that is, we exclude the hard-thresholding case (which is too ideal for the kind of individualized learning we address in this article).
To measure the difference between $A(\cdot)$ and $A'(\cdot)$, we introduce $M_r(A, A')$ to denote the minimum non-negative integer such that the first $r-M_r(A, A')+1$ covariates in ordering $A(\cdot)$ is ranked 
among
the first $r+1$ positions in ordering $A'(\cdot)$, i.e., variables 
$\{X_0, \ldots, X_{r - M_r(A, A')}\}$ are included in $\{X'_0, \ldots, X'_r\}$. 
Note that $M_r(A, A') \le r$ because we can assume $X'_0= X_0$ since they both denote the constant term. It is asymmetric in $A$ and $A'$, 
and the farther $M_r(A, A')$ is away from zero, the more different $A$ and $A'$ will be. 
That is, $M_r(A, A')$ is the number of mistakes we make in choosing the first $r+1$ covariates with respect to the original ordering $A(\cdot)$. The following theorem tells us how many mistakes are acceptable, asymptotically.

\begin{theorem}\label{th:ordering}
Assume that 
(a) the family of prediction functions becomes richer as resolution increases, and is invariant to the permutation of the covariates at each resolution; (b)
the estimation error rate is invariant to the ordering: 
$\varepsilon(r_n, n)/[A(r_n)+\tau^2] \asymp \varepsilon'(r_n, n)/[A'(r_n)+\tau^2]$. 
Then a sufficient condition for $A'(R_n) = O(A(R_n))$  under each decay scenario (as underlined and where $\xi>0$) is given below. 
\begin{itemize}
	\item[(i)] \underline{Exponential Decay: $A(r) \asymp e^{-\xi r}$}:\quad 
	$\limsup_{r \rightarrow \infty}$
	$M_r(A, A') \leq Constant$.
	
	\item[(ii)] \underline{Polynomial Decay:  $A(r) \asymp r^{-\xi}$ }:\quad  $\limsup_{r \rightarrow \infty}$
	$M_r(A, A')/r <1$. 
	
	\item[(iii)] \underline{Logarithmic Decay:  $A(r) \asymp \log^{-\xi}(r)$}: \quad $M_r(A, A') = r - r^{1/a_r}$ with $a_r = O(1)$. 
\end{itemize}
\end{theorem}

The qualitative message of Theorem~\ref{th:ordering} is rather intuitive. The fewer of the important predictors that exist, the surer we need to include them in our prediction model.  Although we still need to obtain the necessary conditions, the quantitative messages here can be taken as theoretical guidelines. With exponential decay, the number of forgivable mistakes is very limited, and it cannot be permitted to grow with the resolution level. Under polynomial decay, which still includes the practically
learn-able case when the estimation error is also polynomial in resolution $r$, we can permit the number of mistakes to increase linearly with $r$ (but of course less rapidly than the growth rate of $r$).  

This learn-able case is perhaps the practically most important scenario, since polynomial decay and polynomial estimation error are the kind of cases that we hope to encounter in practice. Exponential decay is likely too much for which to hope in many practical situations, and logarithmic decay is hopeless in terms of individualized learning, as seen in Table~\ref{tab:optimal_rate} and Table~\ref{tab:optimal_rate_zero}. The result in Theorem \ref{th:ordering} with logarithmic decay indicates that we can be almost entirely wrong in our ordering but still maintain the optimal rate. This seemingly too-good-to-be-true result indeed is a negative one,
because it is made possible by the fact that there is really not much information in the predictors,  so whatever orders one uses will not improve the situation.

\section{Prediction with Infinitely Many Continuous Predictors }\label{sec:theory_linear}

\subsection{Normal linear models with infinitely many continuous covariates}\label{sec:linear_model}

Consider the simple linear regression model with infinitely many covariates, which we assume to hold for both the target and training populations:
\begin{align}\label{eq:linear_model}
    Y & = \bm{\beta}_{\infty}^\top \vec{\bm{X}}_{\infty} + \eta \equiv \sum_{r=0}^{\infty} \beta_r X_r + \eta, \quad \eta \sim \mathcal{N}(0, \sigma^2_\eta), \quad 
    \eta \ind \vec{\bm{X}}_{\infty},
    \nonumber
    \\  
    & {\rm where } 
     \ X_0 = 1, \ \ 
    \{X_1, X_2, \ldots\} \text{ are jointly normally distributed}.  
\end{align}
Clearly for $\V(Y) < \infty$, always the case in practice, there will be restrictions on $\beta_r$'s. 
Here we choose the loss function to be $\loss_{\odot}(y, \hat{y}) = \loss (y, \hat{y}) = (y - \hat{y})^2$, 
and the prediction function at resolution $r$ to be linear in the first $r+1$ covariates, i.e., $g(\vec{\bm{x}}_{r}, \bm{\theta}_r)=\bm{\theta}_r^\top \vec{\bm{x}}_{r}.$

Under this setting, 
the optimal prediction function is $g(\vec{\bm{x}}_{r}, \bm{\theta}_r^*) = \E(Y \mid \vec{\bm{X}}_{r} = \vec{\bm{x}}_{r})$. 
The estimator $\hat{\bm{\theta}}_r$ for the true $\bm{\theta}_r^*$ using empirical risk minimization is the least-squares estimator based on the first $r+1$ covariates in the training set $\trainingset_n$. Thus, our prediction for a unit with covariates $\bm{x}_{\infty}$ using primary resolution $r$ is $g(\vec{\bm{x}}_{r}, \hat{\bm{\theta}}_r)=\hat{\bm{\theta}}_r^\top \vec{\bm{x}}_{r}$. 
Now we investigate the prediction error at a specific resolution $r$ and in particular its decomposition as in Section \ref{sec:decomposition}. 
First, because we consider square loss and specify the prediction function perfectly, the ultimate risk $\tau^2=
\sigma^2_\infty\equiv \Var(Y\mid\vec{\bm{X}}_{\infty})=\sigma^2_\eta$.
Note here because of the additivity of the error term $\eta$ in \eqref{eq:linear_model}, $\sigma^2_\infty$ is a constant. In general, $\tau^2$ and $\sigma_\infty^2$ are different. 
In the following we
will use $\tau^2=0$ to indicate the world without variance.

Second, define $\delta_k^2 \equiv \Var(Y\mid \vec{\bm{X}}_{k-1}) - \Var(Y \mid \vec{\bm{X}}_{k})$ as the variance of the response  explained by the $k$th covariates in excess to that by the previous ones. Then $A(r) = \sum_{k=r+1}^\infty \delta_k^2$. 
Third, the estimation error is 
$
\varepsilon(r, \trainingset_n) = (\hat{\bm{\theta}}_r-\bm{\theta}^*_r)^\top
\E(
\vec{\bm{X}}_{r} \vec{\bm{X}}_{r}^\top
)
(\hat{\bm{\theta}}_r-\bm{\theta}^*_r),
$
and its expectation over all training sets of size $n$ is (see the Appendices)
\begin{align}\label{eq:epsilon_r_n_linear}
    \varepsilon(r,n) = \E_n \left[\varepsilon(r, \trainingset_n)\right] = \frac{A(r)+\tau^2}{n-r-2}\left(
    \frac{n-2}{n}+r
    \right).
\end{align}
Consequently, the average prediction error in \eqref{eq:decomposition_average_training} at resolution $r$ is
\begin{align}\label{eq:pred_loss_linear}
\E_n\left\{\E[Y - g( \vec{\bm{X}}_{r}, \hat{\bm{\theta}}_r)]^2 \right\}
& = \tau^2 + \sum_{k=r+1}^\infty \delta_k^2 + 
\E_n\left[
(\hat{\bm{\theta}}_r-\bm{\theta}^*_r)^\top
\E(
\vec{\bm{X}}_{r} \vec{\bm{X}}_{r}^\top
)
(\hat{\bm{\theta}}_r-\bm{\theta}^*_r)
\right]
\nonumber
\\
& = \left[
	\tau^2 + A(r)
	\right]
	\cdot
	\frac{(n+1)(n-2)}{n(n-r-2)}.
\end{align}

The prediction error under linear models is also reported in \citet{hastie2019surprises}, where the authors studied ridgeless regression in the growing-$p$-\&-$n$ setting, with $p/n$ 
assumed to converge to a limit $\gamma$.
Like most articles in the large-$p$-small-$n$ literature, they assumed the residual variance, in our notation $A(p)+\tau^2$,  is free of 
$p$. Under such an assumption, we see from (\ref{eq:pred_loss_linear}) (after replacing $r$ by $p$), that for any value of $\tau^2>0$, the prediction error always explodes when $\gamma = p/n$ approaches 1, yielding the turning point for the ``double descent" phenomenon that we will discuss in Section~\ref{sec:double_descent}. 

However, under our MR framework, it is clear that as the number of predictors increases, the variance unexplained, that is, the residual variance will decrease in general.  Hence it makes little statistical sense to assume $A(r)$ will stay as a constant as $r$ changes -- if this were the case, what would be the point of including more predictors?  By explicitly considering the behavior of the unexplained variance as number of predictors increases, the prediction error can have very different characteristics  under different scenarios. In particular, it is quite clear from \eqref{eq:pred_loss_linear} that when $\tau^2=0$, the prediction error may not explode when $r/n$ approaches one, because $A(r)$ is approaching zero as well, creating a limit of the form $0/0$, whose value will depend on the rate at which $A(r)$ approaches zero.  We will investigate this issue shortly in Section \ref{sec:linear_zero_tau} when $\tau^2=0$, where we reveal
the phenomenon for the optimal resolution $R$ to be as close to $n$ as possible, traditionally considered impossible because it is in the region of (nearly) over-fitting.

\subsection{General results motivated and illustrated by linear regression}\label{sec:gen_result_linear}

Under the linear model \eqref{eq:linear_model}, when the intrinsic variance is positive, i.e., $\tau^2 > 0$, 
we can show that 
for any sequence of resolution levels $\{r_n\}$, a necessary condition for  $\varepsilon(r_n,n)=o(1)$ is 
$\lim_{n\rightarrow \infty} r_n/n=0$.
Moreover, under this condition, $\varepsilon(r_n,n)\asymp r_n/n$; see  the Appendices for a proof. More generally,
we expect that  $\varepsilon(r_n, n) \asymp \dim(\bm{\theta}_r)/n$ holds for continuous predictors under regularity conditions. 

In general cases with continuous covariates,
typically  $\dim (\bm{\theta}_r)\asymp r^{\alpha}$ for some $\alpha>0$. 
The following theorem considers an assumption involving $\varepsilon(r, n) \asymp \dim(\bm{\theta}_r)/n \asymp r^\alpha/n$. That is, the linear model motivates us to consider this assumption of polynomial estimation error rate in resolution, but the result below is not restricted to the linear model. All proofs are given in the Appendices.

\begin{theorem}\label{th:cont}
Let $R_n$ be a rate-optimal resolution, and 
$L_n=A(R_n) + \varepsilon(R_n, n)$ be the corresponding minimal prediction error (after removing the ultimate risk). Then we have the following asymptotic results under each condition on the decay rate of $A(r)$ (as underlined), but all assume \textit{polynomial estimation error}, that is, $\varepsilon(r, n)\asymp r^{\alpha}/n$, where $\alpha>0$. (As in Theorem~\ref{th:ordering}, all $\xi>0$.)

\begin{itemize}
	\item[(i)] \underline{Hard Thresholding:  $A(r)=0$ for $r\geq r_0$, and $A(r)>0$ for $r< r_0$.}  Then $R_n \asymp 1$  with the constraint that $\liminf_{n\rightarrow \infty} R_n\geq r_0$; and  
	$L_n \asymp n^{-1}$. 
	\smallskip
	\item[(ii)] \underline{Exponential Decay: $A(r) \asymp e^{-\xi r}$.} Then $R_n = a_n \log(n)$ with $a_n$ satisfying 
	$a_n \asymp 1$ and 
	$n^{1-\xi a_n}\log^{-\alpha}(n) = O(1);$ and $L_n\asymp n^{-1}\log^{\alpha}(n)$. 
	\smallskip
	\item[(iii)] \underline{Polynomial Decay:  $A(r) \asymp r^{-\xi}$.}  Then $R_n\asymp n^{1/(\alpha+\xi)}$;
	and $L_n \asymp n^{-\xi/(\alpha+\xi)}$. 
	\smallskip
	\item[(iv)] \underline{Logarithmic Decay:  $A(r) \asymp \log^{-\xi}(r)$.} Then  
	$R_n=a_n n^{1/\alpha}\log^{-\xi/\alpha}(n)$ with  $a_n$ satisfying $a_n = O(1)$ and 
	$\liminf_{n\rightarrow \infty} \left[\log(a_n)/\log(n)\right] >- \alpha^{-1};$  and $L_n \asymp \log^{-\xi}(n)$.
\end{itemize}
\end{theorem}

This result provides precise descriptions of various restrictions on the deterministic sequence $c_n$ in the first row of Table~\ref{tab:optimal_rate},  although their details are mostly secondary to the theoretical and practical insights discussed in Section~\ref{sec:rate_optimal_resolution}.  
Moreover, Theorem \ref{th:cont}, as well as the later theorems,
relies only on the rates of $A(r)$ and $\varepsilon(r,n)$, and thus can be applied to general sieves with the same rates of $A(r)$ and $\varepsilon(r,n)$.
We remark that in the derivations above we can replace the expected error $\varepsilon(r, n)$ by 
$\varepsilon(r, \trainingset_n)$, which depends on the actual training set, 
as in \eqref{eq:decomposition}. 
That is, we can seek resolution levels $\{r_n\}$ such that $A(r_n) + \varepsilon(r_n, \trainingset_n)$ converges to zero in probability in the fastest way. The results remain the same if we replace $``\asymp"$ by $``\asympprob"$. 
Indeed, 
for the linear model \eqref{eq:linear_model} with positive $\tau^2$ we show in the Appendices that  
(a) for any resolution $\{r_n\}$, $r_n/n = o(1)$ is necessary for the actual estimation error $\varepsilon(r, \trainingset_n)$ to be $o_{\pr}(1)$, 
and (b) 
when $r_n/n = o(1)$, 
$\varepsilon(r, \trainingset_n) \asympprob r_n/n$.
Therefore, Theorem \ref{th:cont} applies with $\alpha = 1$ and $``\asymp"$ replaced by $``\asympprob"$.

\subsection{Specific results for linear regression without variance}\label{sec:linear_zero_tau}

When $\tau^2=0$, however, we are entering a rather different world. 
Under model \eqref{eq:linear_model} with zero $\sigma^2_\eta(=\tau^2)$, the response $Y$ is (almost surely) a deterministic function of
the countably many covariates. This is not merely a philosophical contemplation, but a mathematical reality.  Indeed, any random variable can be obtained deterministically from a set of uniform variables on the unit interval, and any such uniform variable admits the binary expansion $\sum_{i=1}^\infty 2^{-i} U_i$, where $\{U_i,  i\ge 1\}$ are i.i.d. Bernoulli(1/2); see \cite{doi:10.1080/01621459.2018.1537921} for an investigation of using this deterministic expansion to study statistical independence.  

Of course, empirically it is impossible to test whether $\tau^2=0$. Hence one would expect or at least hope that it is inconsequential for practical purposes to set $\tau^2=0$ or not, 
as alluded to in \cite{meng2014trio}. Therefore we were surprised initially when we saw the critical dependence of our asymptotic results on whether $\tau^2=0$ or not. When $\tau^2=0$, the asymptotic error $\varepsilon(r,n)$ is no longer dominated by the usual $r/n$ order, but by $A(r)$ itself, as discussed previously. 
Specifically,
contrasting with the case where $\tau^2>0$, 
$r/n = o(1)$ is no longer a necessary requirement for  $\varepsilon(r,n)$ to converge to zero, because $A(r)$ can drive
the error to zero even if $r/n\rightarrow 1$, as seen in \eqref{eq:pred_loss_linear}.  This fact leads to different results from Theorem~\ref{th:cont},
as summarized below. We emphasize that the following theorem, although focuses on the linear model, also holds for cases where the estimation error following the same rate as that in \eqref{eq:pred_loss_linear}.

\begin{theorem}\label{th:contzero}
Under model \eqref{eq:linear_model} with $\tau^2=0$ and $L^2$ loss, 
the rate-optimal resolution $R_n$ and the corresponding minimal prediction error $L_n = A(R_n) + \varepsilon(R_n, n)$ have the following forms under each condition on the decay rate of $A(r)$, where all $\xi>0$.  
\begin{itemize}
	\item[(i)] \underline{Hard Thresholding:  $A(r)=0$ for $r\geq r_0$, and $A(r)>0$ for $r< r_0$.}  The optimal resolution is any $R_n$ such that 
	$\liminf_{n\rightarrow \infty} R_n\geq r_0$ and $R_n \le n-3$; 
	and  
	$L_n = 0$ for sufficiently large $n$. 
	\item[(ii)] \underline{Exponential Decay: $A(r) \asymp e^{-\xi r}$.}  $R_n = n-O(1)$ with $R_n \le n-3;$ and $L_n\asymp n e^{-\xi n}$. 
	\item[(iii)] \underline{Polynomial Decay:  $A(r) \asymp r^{-\xi}$.}  
	$R_n = a_n n $ with $a_n$ satisfying $a_n \asymp 1$ and $\limsup a_n < 1$;
	and $L_n \asymp n^{-\xi}$. 
	\item[(iv)] \underline{Logarithmic Decay:  $A(r) \asymp \log^{-\xi}(r)$.} 
	 Optimal resolution is any $R_n$ such that\\ 
	$\limsup R_n/n < 1$, $\liminf \frac{\log R_n}{\log n} > 0;$  
	and $L_n \asymp \log^{-\xi}(n)$.
\end{itemize}
\end{theorem}

The most unexpected finding here is that, unlike the case with $\tau^2>0$  where no optimal $R_n$ approaches  over-fitting, that is, having $R_n$ close to $n$, all four cases here permit or even require $R_n$ to be the same order as $n$.
When $A(r)$ has a hard threshold or decays exponentially, we can even allow $R_n = n-3$, almost the largest resolution level by which we can fit an ordinary least squares given sample size $n$ (recall we have $r+1$ unknown parameters at resolution $r$).  When $A(r)$ decays polynomially or logarithmically, we can choose $R_n = cn$ for some constant $c\in (0,1)$. 
That is, the usual concerns with over-fitting disappear. Another unexpected finding is that the logarithmic case permits $R_n \asymp n^{k}$ for $k \in (0,1)$, which is smaller than the polynomial case, against our intuition that slower decay should require a larger number of covariates. 
However, this does not contradict Theorem~\ref{th:cont}, which applies only to cases with $\tau^2>0$.

These unexpected theoretical results compel us to think harder about our intuitions built from the results in
Section \ref{sec:gen_result_linear}, which are consequences from the principle of bias-variance trade-off. 
Does the principle fail here,  as some declared about the ``double descents" phenomena in machine learning, which apparently can also prefer over-fitted models \citep[e.g.,][]{belkin2019reconciling,hastie2019surprises,nakkiran2019deep}? Whereas more research is needed to understand the deterministic regime as identified by Theorem~\ref{th:contzero}, our current understanding is that the bias-variance trade-off is sound and well. 
In a world with zero variance, the optimal trade-off should place all its bets on the bias term.  In a deterministic world, the more mathematical constraints imposed for solving a set of equations, the smaller is the set of potential solutions. Without any variance, any specific individual case is a hard mathematical constraint for reconstructing the deterministic relationship between the outcome and the predictors. It is not surprising therefore---retrospectively---that the mathematics is instructing us to use as higher resolution as possible, except for saving some degrees of freedom to take care of the ``pseudo-variance" caused by $A(r)$, when it does not decay sufficiently rapidly. 

Attempting to understand this preference for over-fitting by the deterministic setup, we realize that the ``double descents" phenomenon may not be due to over-fitting as currently depicted, or at least it can also occur within the ``under-fitting" region.  
In the current literature,  ``double descents" 
refers to the phenomenon that as $p$  increases, the prediction error or risk first decreases due to the bias reduction, 
and then increases due to the inflated variance.
However, as $p$ exceeds (effective) data size, the prediction error decreases again, i.e, it exhibits a double descent phenomenon.
Many researchers have tried to understand this phenomenon, and most of the studies attribute it to over-parameterization and that the fitted model tends to be the smoothest one interpolating all training samples;
see, e.g., \citet{belkin2019reconciling,hastie2019surprises}.

The section below demonstrates that double and indeed multiple descents can occur without over-parameterization. This fact suggests that the issue of ordering covariates
discussed in Section \ref{sec:ordering} is an intrinsic one, and that the reasons for the double descents phenomena in machine learning might be more nuanced than over-parametrization.

\subsection{No surprises:  Double and multiple descent phenomenon}\label{sec:double_descent}

We first consider a setting which demonstrates a double descent phenomena within the under-fitting region.
We assume that the resolution bias has the following form:
\begin{align}\label{eq:approx_error_double}
    A(r) = 
    \begin{cases}
    r^{-1}, & \text{if } r \le \underline{r}, \\
    \frac{1 + \exp(\underline{r} - \overline{r})}{\underline{r}} \cdot \frac{1}{1 + \exp(r - \overline{r})}, 
    & \text{if } r > \underline{r}, 
    \end{cases}
\end{align}
where $\underline{r} \le \overline{r}$ are two positive integers, 
and 
the coefficient $\{ 1 + \exp(\underline{r} - \overline{r}) \}/\underline{r}$ for $r>\underline{r}$ is chosen 
such that $A(r)$ is a continuous function of $r$. 
Figure \ref{fig:double_descent}(a) plots the resolution bias against the resolution when $\underline{r} = 30$ and $\overline{r} = 60$. 
Figure \ref{fig:double_descent}(b) shows the average prediction loss \eqref{eq:pred_loss_linear} when $\tau^2=0$ and $n = 100$, 
which clearly demonstrates a ``double-descent'' phenomenon. 
Comparing Figures \ref{fig:double_descent}(a) and (b), we can see that the double-descent pattern of the prediction error is driven by the varying importance of the added covariates. That is, when we add covariates with little predictive power, we are essentially adding noise to our prediction and hence increase the predictive error, until we add more powerful covariates to (again) bring the error down. 

\begin{figure}[ht]
	\centering
	\begin{subfigure}{.39\textwidth}
		\centering
		\includegraphics[width=1\linewidth]{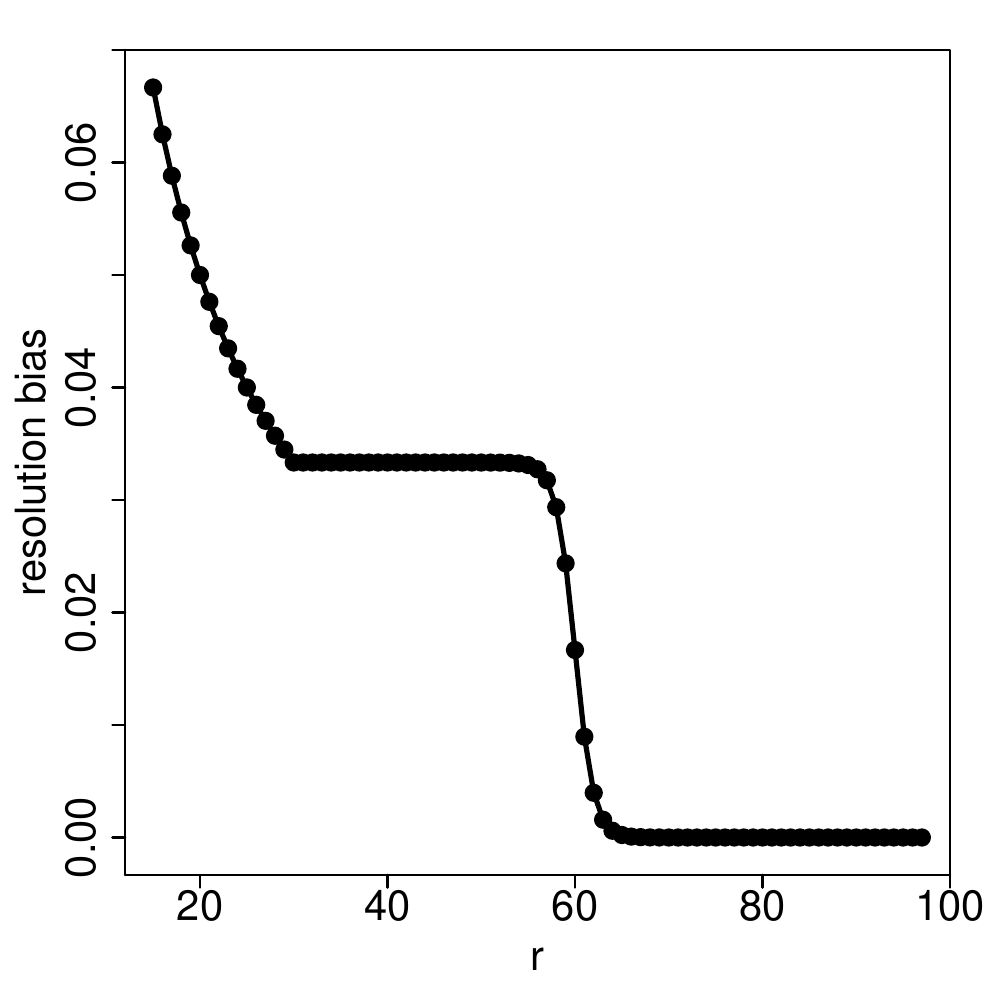}
		\caption{\centering Resolution bias}
	\end{subfigure}%
	\begin{subfigure}{.39\textwidth}
		\centering
		\includegraphics[width=1\linewidth]{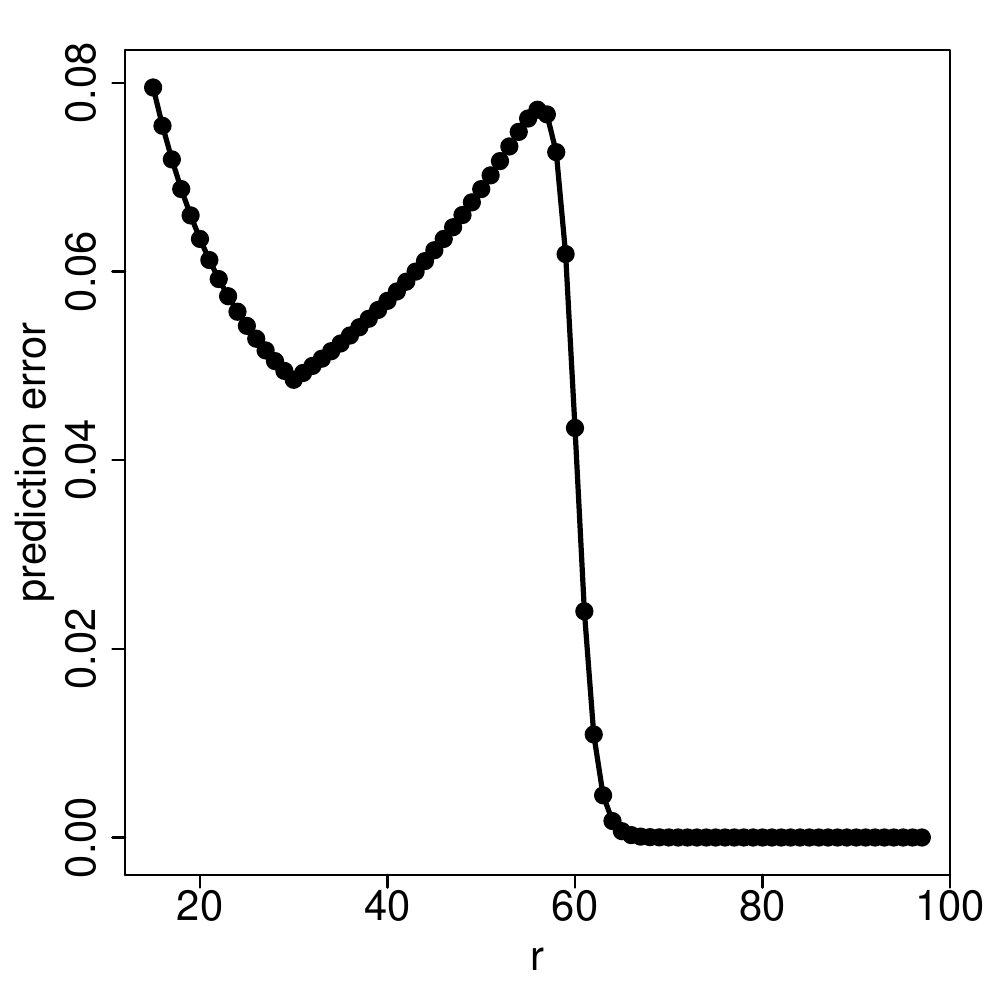}
		\caption{Prediction error}
	\end{subfigure}
	\caption{
	Figures plotting the resolution bias in \eqref{eq:approx_error_double}, as well as the corresponding prediction error with $\tau^2=0$, against the resolution $r$. 
	}\label{fig:double_descent}
\end{figure}                                                                                
With this insight, it is easy to demonstrate multiple-descent phenomenon for as many descents as we want. For example, 
we can take 
\begin{align}\label{eq:approx_error_multiple}
    A(r) = 
    \begin{cases}
    \I\{r \le \underline{r}_1\} \cdot r^{-1} + \I\{r > \underline{r}_1\} \cdot \frac{1 + \exp(\underline{r}_1 - \overline{r}_1)}{\underline{r}_1} \cdot \frac{1}{1 + \exp(r - \overline{r}_1)}, & \text{if } r \le \overline{r}_1, \\
    c_2 \I\{r \le \underline{r}_2\} \cdot r^{-1} + c_2 \I\{r > \underline{r}_2\} \cdot \frac{1 + \exp(\underline{r}_2 - \overline{r}_2)}{\underline{r}_2} \cdot \frac{1}{1 + \exp(r - \overline{r}_2)},
    & \text{if }  \overline{r}_1 < r \le \overline{r}_2, \\
    c_3 \I\{r \le \underline{r}_3\} \cdot r^{-1} + c_3 \I\{r > \underline{r}_3\} \cdot \frac{1 + \exp(\underline{r}_3 - \overline{r}_3)}{\underline{r}_3} \cdot \frac{1}{1 + \exp(r - \overline{r}_3)},
    & \text{if }  \overline{r}_2 < r \le \overline{r}_3,\\
    \ \ \ldots 
    \end{cases}
\end{align}
where 
$\underline{r}_1 \le \overline{r}_1 \le \underline{r}_2 \le \overline{r}_2 \le \underline{r}_3 \le \overline{r}_3 \le \ldots$ 
and $c_k$'s are chosen such that $A(r)$ is a continuous function of $r$. 
Figure \ref{fig:multiple_descent}(a) plots the resolution bias $A(r)$ against the resolution $r$ when 
$\overline{r}_k = \underline{r}_k + 30 = 60k$ for $k\ge 1$. 
From Figure \ref{fig:multiple_descent}(a), 
we can see that, as $r$ increases, 
the resolution bias keeps repeating the pattern in Figure \ref{fig:double_descent}(a), i.e., 
the importance of added covariates keeps fluctuating. 
Figure \ref{fig:multiple_descent}(b) plots the logarithm of the average prediction error in \eqref{eq:pred_loss_linear} against the resolution when the sample size $n=300$ and the intrinsic error $\tau^2=0$. 
Clearly, Figure \ref{fig:multiple_descent}(b) exhibits a multiple-descent phenomenon. 
However, in contrast to Figure \ref{fig:double_descent}(b), the prediction error does not die down in the end. 
This is because the resolution bias in Figure \ref{fig:double_descent}(a) decays exponentially, while that in Figure \ref{fig:multiple_descent}(a) interweaves between exponential and polynomial decays, not covered by our theorems.

\begin{figure}[ht]
	\centering
	\begin{subfigure}{.39\textwidth}
		\centering
		\includegraphics[width=1\linewidth]{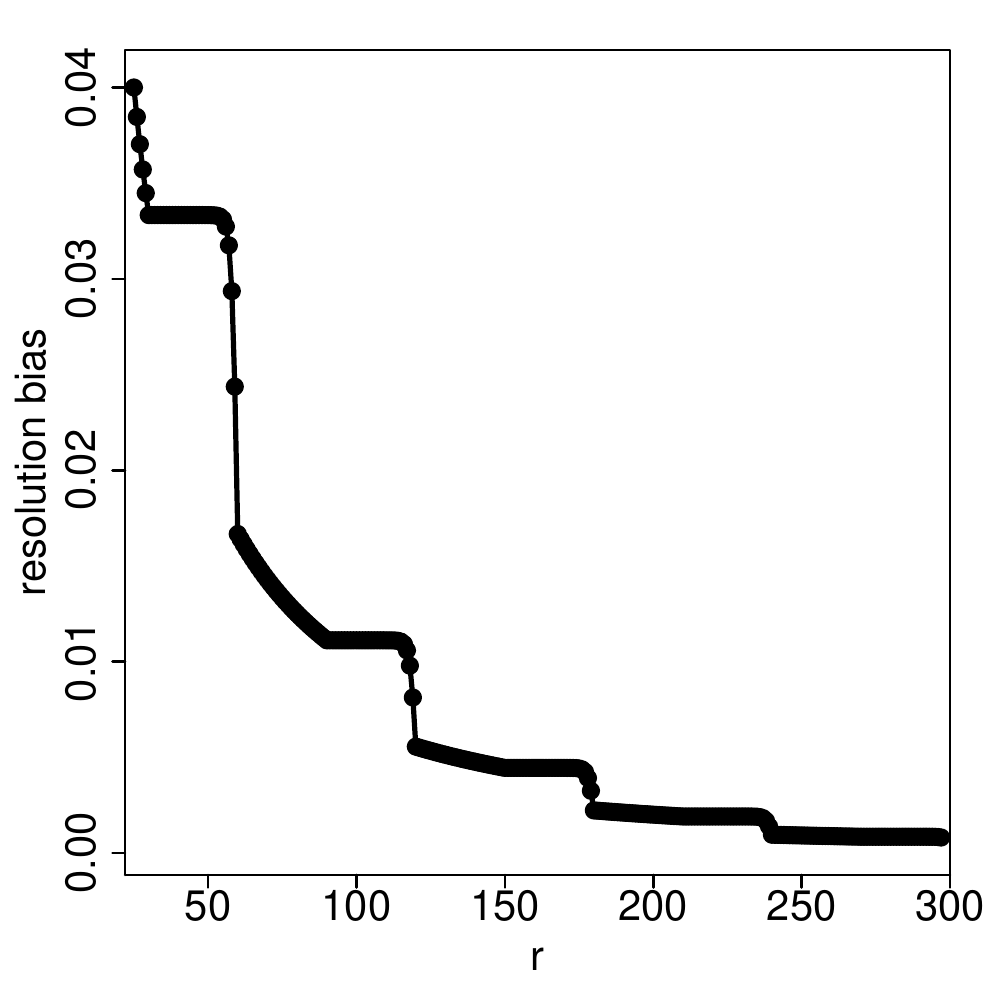}
		\caption{\centering Resolution bias}
	\end{subfigure}%
	\begin{subfigure}{.39\textwidth}
		\centering
		\includegraphics[width=1\linewidth]{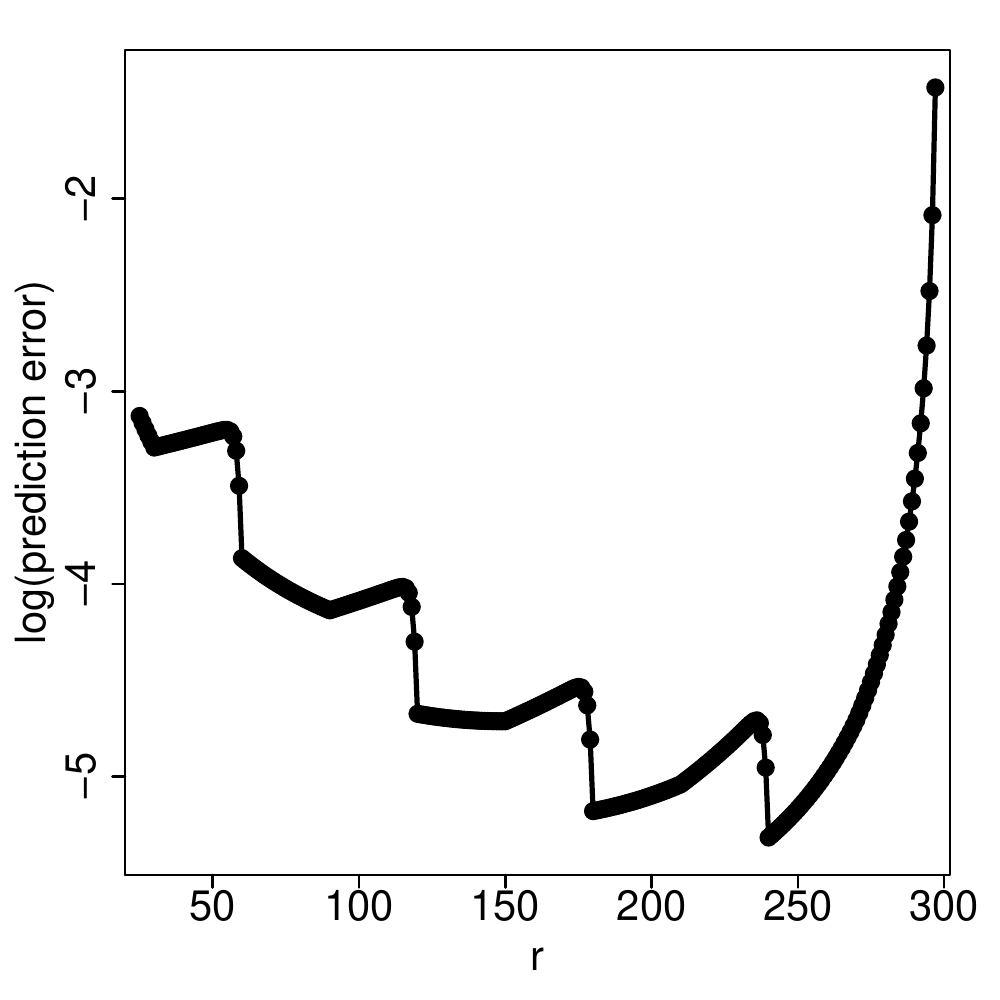}
		\caption{Prediction error}
	\end{subfigure}
	\caption{
	Figures plotting the resolution bias in \eqref{eq:approx_error_multiple}, as well as the corresponding prediction error (with $\tau^2=0$), against the resolution $r$. 
	}\label{fig:multiple_descent}
\end{figure}  

From the above discussion, it is not difficult to see that double or multiple descent phenomena are driven by the varying decay of resolution bias and inflation of the estimation error. 
Depending on which of these two terms is dominating, the prediction error can either decrease or increase, and can thus exhibit multiple descent patterns.  
{\newchange A reviewer points out that the multiple descent phenomenon can also occur when most of the covariates are irrelevant and the relevant ones appear sporadically.}
Such phenomena are also not restricted to regression settings. For example, in the midst of revising this article,
we learned about \citet{liang2020multiple}, which demonstrated multiple descent phenomena in kernel machines and neural networks.

We remark that, for any monotonically decreasing function $A(r)$,
we can construct a linear model with $A(r)$ as its decay rate, so
all the examples above are realizable. Let
$X_0 = 1$, $\{X_1, X_2, X_3, \ldots\}$ be i.i.d standard normal random variables, and $\eta \sim \mathcal{N}(0, 
\sigma^2_{\eta})$.  
Define $\beta_0$ to be any constant, and $\beta_r = \sqrt{A(r-1) - A(r)},$ for any $r \ge 1.$
Then the corresponding linear model \eqref{eq:linear_model} has the desired resolution bias $A(r)$. We will use this construction in the following simulation study.

\begin{figure}[ht]
	\centering
	\begin{subfigure}{.33\textwidth}
		\centering 
		\includegraphics[width=1\linewidth]{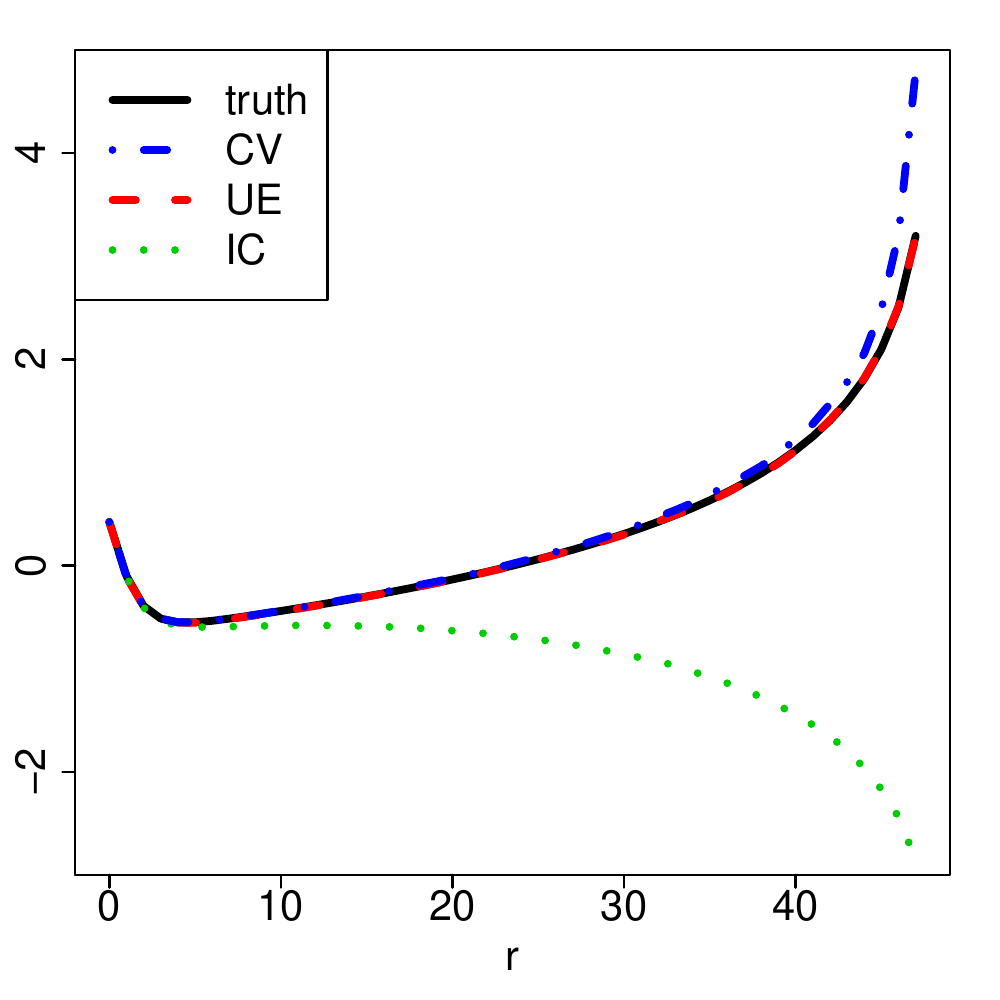}
		\caption{Exponential, $\tau^2=\frac{1}{2}$}
	\end{subfigure}%
	\begin{subfigure}{.33\textwidth}
		\centering
		\includegraphics[width=1\linewidth]{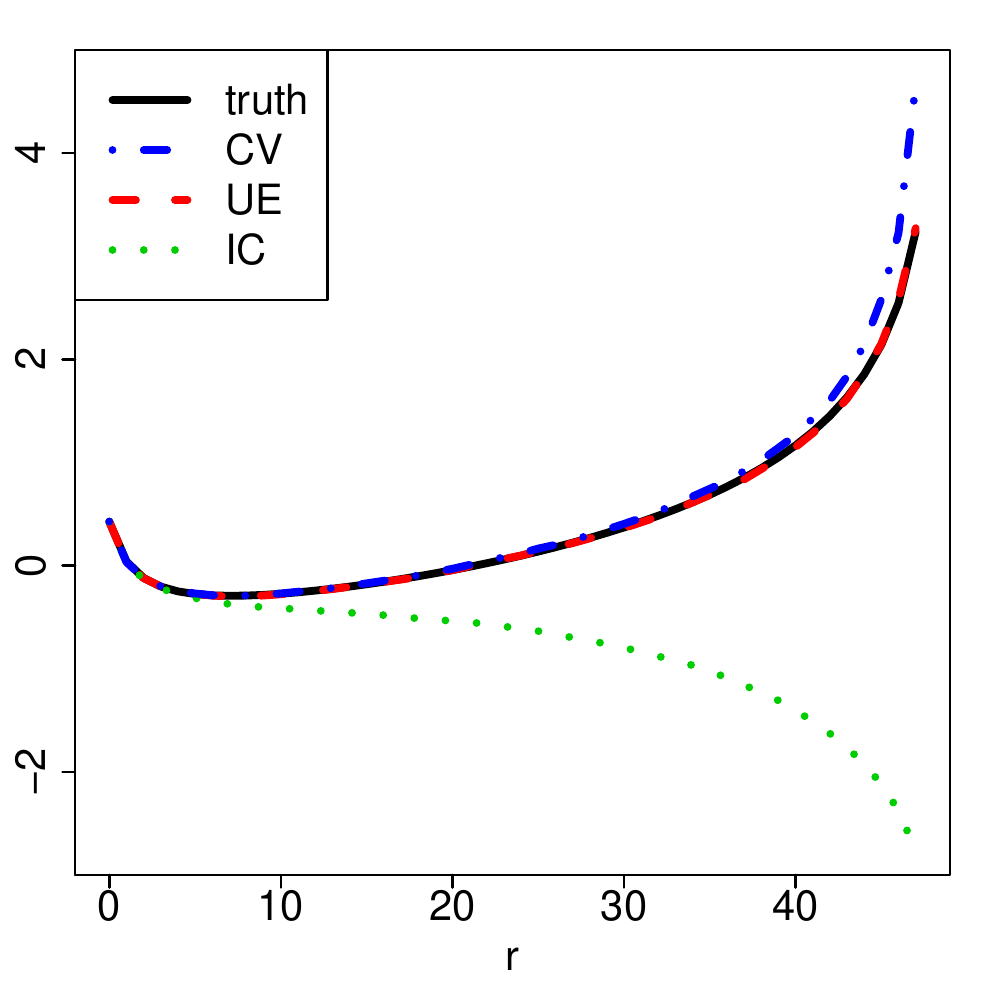}
		\caption{Polynomial, $\tau^2=\frac{1}{2}$}
	\end{subfigure}%
	\begin{subfigure}{.33\textwidth}
		\centering
		\includegraphics[width=1\linewidth]{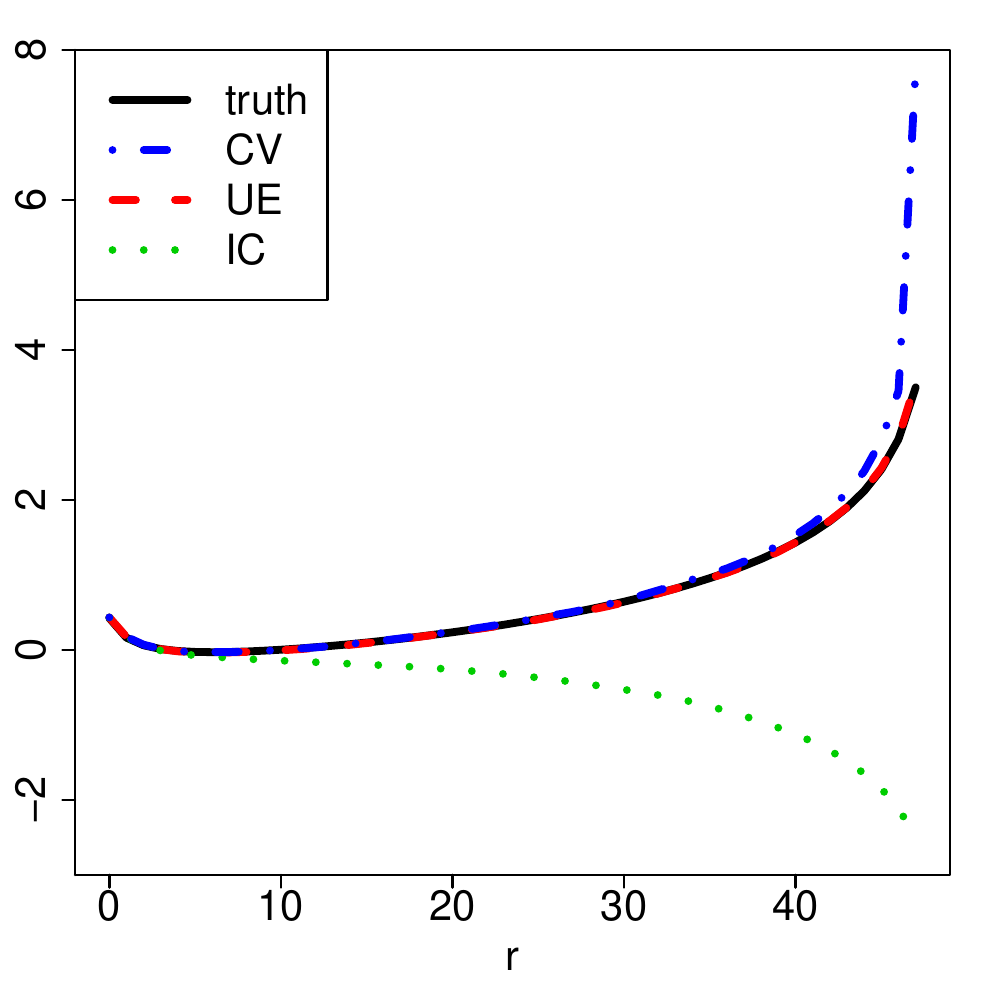}
		\caption{Logarithmic,$\tau^2=\frac{1}{2}$}
	\end{subfigure}
	
	\begin{subfigure}{.33\textwidth}
		\centering
		\includegraphics[width=1\linewidth]{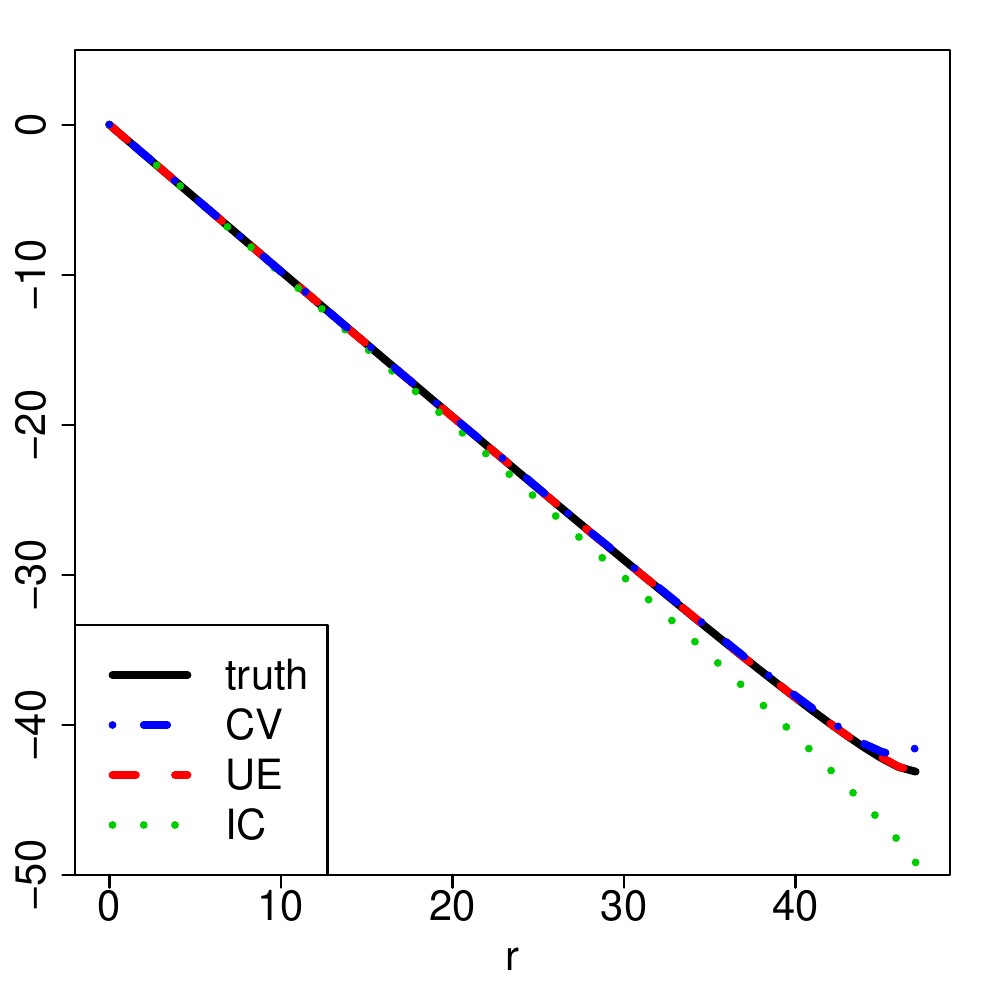}
		\caption{\centering Exponential, $\tau^2 = 0$}
	\end{subfigure}%
	\begin{subfigure}{.33\textwidth}
		\centering
		\includegraphics[width=1\linewidth]{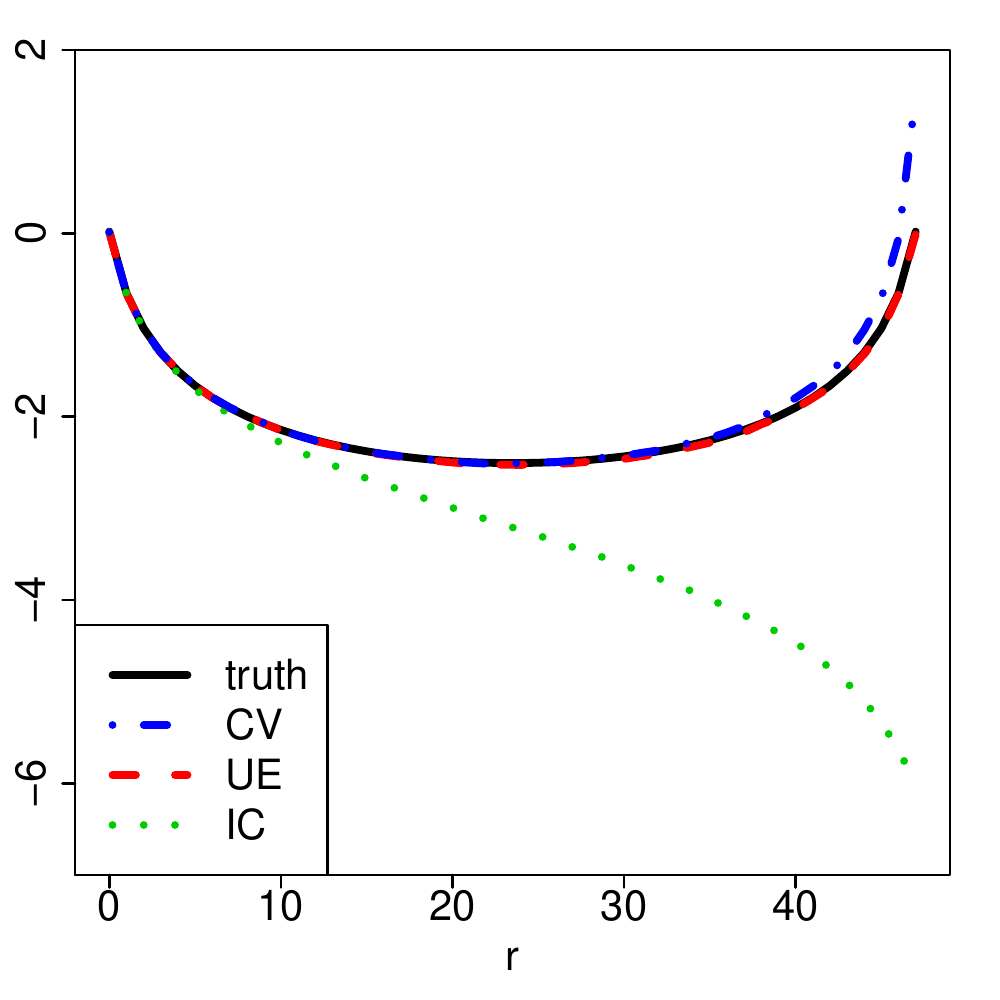}
		\caption{Polynomial, $\tau^2 = 0$}
	\end{subfigure}%
	\begin{subfigure}{.33\textwidth}
		\centering
		\includegraphics[width=1\linewidth]{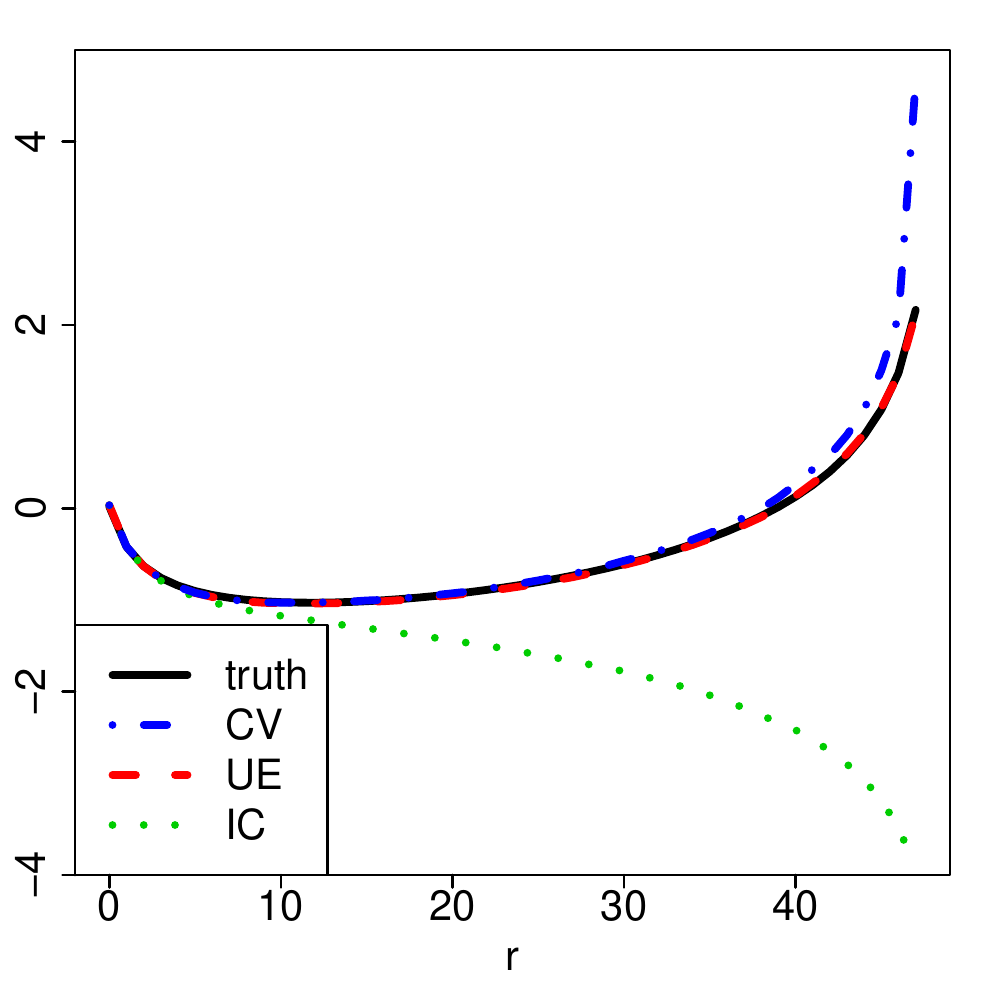}
		\caption{Logarithmic, $\tau^2 = 0$}
	\end{subfigure}                
	\caption{
		The performance of three strategies 
		CV, UE and IC
		for estimating prediction error when $n=50$ and respectively with $\tau^2 = 0.5$ (top row) and $\tau^2 = 0$ (bottom row) . 
		The $x$-axis denotes resolution level $r$, 
		and the $y$-axis denotes the logarithm of the 
		true and estimated
		average prediction error over 500 simulated training sets. 
		The resolution biases follow the decay rates of $e^{-r}$, $r^{-1}$ and $\{\log(r)\}^{-1}$, respectively, for the three scenarios in (a)--(f).  
	}                                                                         
	\label{fig:estimate_pred_loss}                                               
\end{figure} 

\subsection{Finite sample performance -- Preliminary findings}\label{sec:finite}  

Whereas theoretical results are extremely useful for providing deep understanding and revealing new insights, we must be mindful that they may or may not match the empirical findings with finite samples. As a first step towards a comprehensive (and very challenging) study of our MR framework with
finite samples,  we conducted a simulation study using the normal linear model in Section~\ref{sec:linear_model}. The simplicity of this model allows us to compute the optimal resolution and minimal prediction error exactly for any given $n (\ge 3)$, which can then be used as benchmarks to investigate the performance of various estimators for the optimal resolution. However, the model is still sufficiently rich and realistic to both confirm some of the asymptotic findings, including the resistance to over-fitting in the absence of intrinsic variation, and to reveal complications with finite samples that are not captured by the asymptotic results. 

Due to space limitations, we report only findings on three ways of estimating prediction error curves in finite samples as functions of the resolution $r$, which then can be minimized for estimating optimal resolution. The three methods are based on cross validation (CV), an unbiased estimator (UE), and an information criteria (IC); see Appendix~\ref{sec:practical} for details and all other findings. Figure~\ref{fig:estimate_pred_loss} plots the logarithm of averages of the three estimators over 500 Monte Carlo replications against the resolution level $r$, under different choices of the decay rate $A(r)$ and intrinsic variance $\tau^2$, all with $n=50$. 
                              
We see that UE worked well by being unbiased, CV performed well except when venturing into the over-fitting region, and IC failed badly other than when $r$ is small. 
The only exception is when there is no bias-variance trade-off, as depicted in plot (d), where the optimal resolution reaches the sample size, in which case the gross over-fitting tendency of IC brings benefit instead of damage. All six curve shapes are consistent with the theoretical findings in Theorem~\ref{th:cont} (for $\tau^2>0$) and in  Theorem~\ref{th:contzero} (for $\tau^2=0$).

\section{Predictions with Infinitely Many Categorical Predictors}\label{sec:theory_tree}

\subsection{Regression tree models with infinitely many categorical covariates}\label{sec:regression_tree}

We now introduce regression tree models with infinitely many categorical covariates, and then use them to illustrate some general results on rate optimal resolution and prediction. 
Specifically, we assume both target and training populations satisfy
\begin{align}\label{eq:binary_model}
    X_1, X_2, \ldots \text{are i.i.d. with } \pr(X_i = k) = M^{-1} \text{ for } k = 1,2, \ldots, M, 
    \quad 
    \V(Y) < \infty,  
\end{align}
and the dependence of $Y$ on $\{X_1, X_2, \ldots\}$ is arbitrary, 
where $M\ge 2$. That is, 
\eqref{eq:binary_model}
is a regression tree in which each covariate increases the depth of the tree by one, and hence it is a tree of (potentially) infinite depth.
The loss function is again the square loss: $\loss_{\odot}(y, \hat{y}) = \loss (y, \hat{y}) = (y - \hat{y})^2$, 
and the prediction function at resolution $r$ is fully saturated, that is, it can have different values for different covariates up to resolution $r$, 
$$g(\vec{\bm{x}}_{r}, \bm{\theta}_r) = \sum_{\vec{\bm{a}}_{r}\in \{1, 2, \ldots, M\}^{r+1}} \I (\vec{\bm{x}}_{r}=\vec{\bm{a}}_{r}) \bm{\theta}_{r}(\vec{\bm{a}}_{r}), $$ 
where the summation is essentially over $M^r$ terms because $X_0\equiv 1$, 
$\dim(\bm{\theta}_r) = M^r$ and $\bm{\theta}_r(\vec{\bm{a}}_{r})$ denotes the coordinate corresponding to covariate value $\vec{\bm{a}}_{r}$.

Given a training set $\trainingset_n$, for each resolution $r$, we use $n(\vec{\bm{x}}_{r})$ to denote the number of units with covariate value $\vec{\bm{x}}_{r}$. 
When $n(\vec{\bm{x}}_{r}) >0$, 
minimizing the empirical risk will lead to taking the 
sample average of the outcome of these $n(\vec{\bm{x}}_{r})$ individuals. The matter is more complicated when $n(\vec{\bm{x}}_{r}) = 0$. Here we adopt the ``highest-resolution imputation".  
That is, for each individual of interest, 
we find training samples that have the same covariates up to a resolution that is as large as possible but is truncated at $r$,  
and then use their average response as a prediction for this individual. Note that this estimator is unique conditioning on the given order of the predictors. 
Consequently, our estimator for the parameter $\bm{\theta}_{r}$ has the following form:
\begin{align}\label{eq:theta_binary}
\hat{\bm{\theta}}_{r}(\vec{\bm{x}}_{r}) & = 
\begin{cases}
\frac{1}{n(\vec{\bm{x}}_{r})} \sum_{i: \vec{\bm{x}}_{ir}=\vec{\bm{x}}_{r}} Y_i, & \text{if } n(\vec{\bm{x}}_{r}) >0, \\
\frac{1}{n(\vec{\bm{x}}_{k})} \sum_{i: \vec{\bm{x}}_{ik}=\vec{\bm{x}}_{k}} Y_i, & \text{if } n(\vec{\bm{x}}_{k})>0 \text{ and }  n(\vec{\bm{x}}_{k+1})=0, \text{ for } 0 \le k < r.
\end{cases}
\end{align}
This estimator is always well-defined, because $n(\vec{\bm{x}}_{0})=n>0$.
Under  model \eqref{eq:binary_model}, 
we can derive that 
(i)
the ultimate risk is $\tau^2 =\E[
\Var(Y \mid \vec{\bm{X}}_\infty)]$, 
(ii)
the resolution bias is 
\begin{align*}
A(r) 
= \sum_{k=r+1}^{\infty}
\left\{
\E[
\Var(Y \mid \vec{\bm{X}}_{k-1})
] - \E[
\Var(Y \mid \vec{\bm{X}}_{k})
]\right\},
\end{align*}
and (iii) the estimation error is $\varepsilon(r,\trainingset_n) = 
\E[
\hat{\bm{\theta}}_{r}(\vec{\bm{X}}_{r}) - \E(Y\mid \vec{\bm{X}}_{r})
]^2.$ 
The expectation of $\varepsilon(r,\trainingset_n)$ 
over the training sets has three terms,
as indicated and simplified below: 
\begin{align}\label{eq:epsilon_r_n_binary}
\varepsilon(r,n) & = 
\left[ A(r) +  \tau^2 \right]  \cdot 
\E_n\left[ \frac{\I( n(\vec{\bm{1}}_{r}) > 0 )}{n(\vec{\bm{1}}_{r})}
\right]\hskip 1.6in  ({\rm Var\ when\ } n(\vec{\bm{1}}_{r})>0)
\nonumber
\\
& \quad \ + 
\sum_{k=0}^{r-1} 
\left[A(k) +  \tau^2 \right]
\cdot
\E_n\left[
\frac{\I(n(\vec{\bm{1}}_{k}) > 0, n(\vec{\bm{1}}_{k+1}) = 0)}{n(\vec{\bm{1}}_{k})} \right] \qquad \quad ({\rm Var\ when\ } n(\vec{\bm{1}}_{r})=0)
\nonumber
\\
& \quad \ 
+ 
\sum_{k=0}^{r-1} 
\left[
A(k) - A(r)
\right]
\cdot 
\E_n\left[ 
\I(n(\vec{\bm{1}}_{k}) > 0, n(\vec{\bm{1}}_{k+1}) = 0)
\right]\qquad \ \ \  ({\rm Bias\ when\ } n(\vec{\bm{1}}_{r})=0)
\nonumber
\\
& = 
\E_n \left[ \frac{A(\mathcal{K} \wedge r) + \tau^2 }{n\left( \vec{\bm{1}}_{\mathcal{K}\wedge r} \right) } \right]
+ 
\sum_{k=0}^{r-1} 
[A(k) - A(r)]
\cdot 
\E_n\left[ 
\I(n(\vec{\bm{1}}_{k}) > 0, n(\vec{\bm{1}}_{k+1}) = 0)
\right],
\end{align}
where 
$n(\vec{\bm{1}}_{k})$ denotes the number of training samples with covariate value $\vec{\bm{x}}_{ik} = \vec{\bm{1}}_{k}$, 
$\mathcal{K}$ is the maximum integer $k$ such that $n(\vec{\bm{1}}_{k}) > 0$, 
and $\mathcal{K} \wedge r = \min\{ \mathcal{K}, r \}$. 
Note that here
$n(\vec{\bm{1}}_{k}) \sim \text{Binomial}(n, M^{-k})$ and $n(\vec{\bm{1}}_{k+1}) \mid n(\vec{\bm{1}}_{k}) \sim \text{Binomial}(n(\vec{\bm{1}}_{k}), M^{-1})$ for any $k\ge 0$.
We stress that it is the assumption that all $\vec{\bm{x}}_{k}$'s are uniformly distributed that permits us to replace $n(\vec{\bm{x}}_{k})$ by 
$n(\vec{\bm{1}}_{k})$, which greatly simplifies the derivation; see Appendix~\ref{app:theory_tree} for deriving error decomposition under model \eqref{eq:binary_model}.

\subsection{General results inspired and illustrated by regression tree}\label{sec:general_regression_tree}

Under \eqref{eq:binary_model}, 
when $\tau^2 > 0$, we can show 
that 
for any sequence $\{r_n\}$, 
a necessary condition for  $\varepsilon(r_n,n) = o(1)$ is that $\lim_{n\rightarrow \infty} M^{r_n}/n \rightarrow 0$. 
Moreover, 
under this condition, the convergence rate of $\varepsilon(r,n)$ is $M^r/n$, i.e., $\varepsilon(r,n) \asymp M^r/n \asymp \dim(\bm{\theta}_r)/n$. 
Again, these intuitive results require some rather technical proofs,
given in the Appendices.

This inspires us to consider more general cases with categorical covariates in which $\dim (\bm{\theta}_r)\asymp  \alpha^r$ for some $\alpha>1$; for example, $\alpha=2$ if the covariates are all binary,  and 
the prediction function $g(\vec{\bm{x}}_{r}, \bm{\theta}_r)$ can have different values for each of the $2^r$ possible values of $\vec{\bm{x}}_{r}$. This contrasts with the previous case featuring continuous covariates in which the dimension of parameters increases polynomially with the resolution. 
The following theorem is the counterpart of Theorem~\ref{th:cont} under the exponential estimation error. 

\begin{theorem}\label{th:disc}
Same notation and setup as in Theorem~\ref{th:cont}, except that we now assume \textit{exponential estimation error}:  $\varepsilon(r,n)\asymp {\alpha}^r/n$, 
for some $\alpha>1$. As in Theorem~\ref{th:cont}, all $\xi>0$.
\begin{itemize}
	\item[(i)] \underline{Hard Thresholding:  $A(r)=0$ for $r\geq r_0$, and $A(r)>0$ for $r< r_0$.}  Then $R_n \asymp 1$ with the constraint that $\liminf_{n\rightarrow \infty} R_n\geq r_0$, and  
	$L_n \asymp n^{-1}$. 
	
	\smallskip
	
	\item[(ii)] 
	\underline{Exponential Decay: $A(r) \asymp e^{-\xi r}$.}  Then $R_n =[\log(n)+\log(a_n)][\log(\alpha)+\xi]^{-1}$ with $a_n\asymp 1$; and $L_n\asymp n^{-\xi/\{\log(\alpha)+\xi\}}$. 
	
	\smallskip
	
	\item[(iii)] \underline{Polynomial Decay:  $A(r) \asymp r^{-\xi}$.} Then
	$R_n = a_n\log(n)$ with $a_n$ satisfying 
	$a_n \asymp 1$
	and 
	$n^{a_n \log(\alpha)-1}\log^\xi (n) =O(1);
	$
and $L_n\asymp \log^{-\xi}(n)$. 
	
	\smallskip
	
	\item[(iv)] \underline{Logarithmic Decay:  $A(r) \asymp \log^{-\xi}(r)$.}
	Then $R_n=a_n \log(n)$ with $a_n$ satisfying 
	$$\liminf_{n \rightarrow \infty} 
	\frac{\log(a_n)}{\log\log(n)} >-1,  \quad {\rm and} \quad  \frac{[\log\log(n)]^\xi}{n^{1-a_n \log(\alpha)}} =  O( 1 );
	$$
	and $L_n\asymp [\log\log(n)]^{-\xi}$.
\end{itemize}
\end{theorem}

\subsection{Specific results for deterministic regression tree}\label{sec:determin_categorial}

Similar to Section \ref{sec:linear_zero_tau}, 
we consider the case in which the ultimate risk $\tau^2=0$, and we will see again below how this leads to rather different asymptotic behavior. But unlike Section \ref{sec:linear_zero_tau}, even when we restrict ourselves to the regression tree model, the exact asymptotic rate for the estimation error is still difficult to obtain other than when $A(r)$ has a hard-thresholding decay. We therefore adopt a two-step strategy. We first establish an upper bound of the estimation error, yielding a corresponding upper bound for the prediction error, which can then be optimized to obtain the minimal upper-bound rate. 
We then prove that these minimal upper-bound rates are also the maximal lower-bound rates, except for a couple of cases where our proof fails, and hence whether the upper-bound rates are optimal or sharp
is still an open problem.

Specifically, as proved in the Appendices, the estimation error can be bounded by 
\begin{align*}
	\varepsilon(r,n) \le \frac{2M}{n} \sum_{k=0}^r M^k A(r) \equiv \overline{\varepsilon}(r, n).
\end{align*}
Furthermore,  $\overline{\varepsilon}(r, n)$ under varying decay rates for $A(r)$ has the following form:
\begin{align}\label{eq:upper_bound_rate}
	\overline{\varepsilon}(r, n) 
	\asymp 
	\begin{cases}
	n^{-1}, & \text{ if } A(r) \text{ has a hard threshold or } A(r) \asymp e^{-\xi r} \text{ with } \xi>\log(M),\\
	\frac{r}{n}, & \text{ if } A(r) \asymp e^{-\xi r} \text{ with } \xi=\log(M), \\
	A(r)\frac{ M^r}{n}, & \text{ if } A(r) \asymp e^{-\xi r} \text{ with } \xi<\log(M), A(r) \asymp r^{-\xi} \text{ or } A(r) \asymp \log^{-\xi} (r). 
	\end{cases}
\end{align}
From \eqref{eq:upper_bound_rate}, 
compared to $\varepsilon(r, n) \asymp M^r /n$ when $\tau^2 > 0$, 
we can see that the rate of the estimation error depends also on the resolution bias
and converges to zero more quickly;  this is similar to the discussion in Section  \ref{sec:linear_zero_tau} under the linear model. 
Moreover, $M^r /n = o(1)$ is no longer necessary for $\varepsilon(r, n) = o(1)$. 
In particular and somehow surprisingly, 
when the resolution bias decays exponentially with rates faster than or equal to $M^{-r}$, 
the estimation error behaves like the usual parameter setting as in Theorem \ref{th:cont} with a fixed number of (or $r$) unknown parameters at resolution $r$,
even the model at each resolution $r$ allows potentially $M^r$ unknown parameters. 

The following theorem summarizes sufficient conditions for the prediction error to achieve certain (upper-bound) rates 
under varying decay rate of the resolution bias. 
 
\begin{theorem}\label{thm:binary_zero_tau}
Under the model \eqref{eq:binary_model} with $\tau^2=0$ and $L^2$ loss, 
let $L_n = A(R_n) + \varepsilon(R_n, n)  \le A(R_n) + \overline{\varepsilon}(R_n, n)\equiv \overline{L}_n$. 
The rate-optimal resolution $R_n$ or $\overline{R}_n$ and the corresponding optimal $L_n$ or $\overline{L}_n$ (respectively) have the following forms under each $A(r)$, where $\xi>0$. 
\begin{itemize}
	\item[(i)] \underline{Hard Thresholding:  $A(r)=0$ for $r\geq r_0$, and $A(r)>0$ for $r< r_0$.}  Then 
	$R_n$ satisfies that $\liminf_{n\rightarrow \infty}R_n \ge r_0$; 
	and  
	$L_n \asymp (1-M^{-r_0})^n$. 
	\smallskip
	\item[(ii)] \underline{Exponential Decay: $A(r) \asymp e^{-\xi r}$.} 
	\begin{itemize}
	    \item[(a)] If $\xi> \log(M)$, then $\overline{R}_n$ satisfies $n e^{-\xi \overline{R}_n} = O(1)$; and $\overline{L}_n \asymp n^{-1}$. 
	    \item[(b)] If $\xi=\log(M)$, then $\overline{R}_n = a_n \log(n)$ with $a_n$ satisfying $a_n \asymp 1$ and \\ $n^{1-a_n \log(M)}/\log(n) = O(1);$ and $\overline{L}_n\asymp n^{-1}\log(n)$.
	    \item[(c)] If $\xi<\log(M)$, then $\overline{R}_n = a_n \log(n)$ with $a_n$ satisfying  $n^{a_n\log(M) - 1 }  \asymp 1$; and $\overline{L}_n \asymp n^{-\xi / \log(M)}$. 
	\end{itemize}
	\smallskip
	\item[(iii)] \underline{Polynomial Decay:  $A(r) \asymp r^{-\xi}$.}  
	Then
	$\overline{R}_n = a_n \log(n) $ with $a_n$ satisfying 
	$
	a_n \asymp 1
	$
	and 
	$
	n^{a_n \log(M) - 1 } = O(1);
	$
	and $\overline{L}_n \asymp \log^{-\xi}(n)$. 
	\smallskip
	\item[(iv)] \underline{Logarithmic Decay:  $A(r) \asymp \log^{-\xi}(r)$.} 
	Then
	$\overline{R}_n = a_n \log(n) $ with $a_n$ satisfying 
	$$\liminf_{n \rightarrow \infty} 
	\frac{\log(a_n)}{\log\log(n)} >-1,  \quad {\rm and} \quad  n^{a_n \log(M) - 1}=  O( 1 );
	$$
	and $\overline{L}_n \asymp [\log\log(n)]^{-\xi}$.
\end{itemize}
\end{theorem}

Next we prove  that the optimal rates for the upper bounds of prediction errors are also precisely the optimal rates for the true prediction errors, except for the exponential decay case with $\xi\ge \log(M)$, where we can only conjecture but not prove that the results also hold.  
The following theorem summarizes our results, where for completeness, we include the hard thresholding case, even though Theorem \ref{thm:binary_zero_tau} is exact in that case. 
Specifically, we say $l_n$ is an asymptotic lower bound for the prediction error $A(r)+\varepsilon(r,n)$ and denote it as $A(r)+\varepsilon(r, n) \gtrsim l_n$, 
if $l_n = O(A(r_n) + \varepsilon(r_n, n))$,  for any sequence $\{r_n\}$.

\begin{theorem}\label{thm:lower_bound_binary_tau2_zero}
Under model \eqref{eq:binary_model} with $\tau^2=0$ and $L^2$ loss, 
an asymptotic lower bound for  
$\varepsilon(r,n)+A(r)$ has the following form under each condition on $A(r)$, where $\xi>0$. 
\begin{itemize}
    \item[(i)] \underline{Hard Thresholding:  $A(r)=0$ for $r\geq r_0$, and $A(r)>0$ for $r< r_0$.}  Then  $A(r) + \varepsilon(r, n) \gtrsim (1-M^{-r_0})^n$.
	\smallskip
	
	\item[(ii)] \underline{Exponential Decay: $A(r) \asymp e^{-\xi r}$.} 
	 Then  $A(r) + \varepsilon(r, n) \gtrsim n^{-\xi / \log(M)}$. 
	\smallskip
	
	\item[(iii)] \underline{Polynomial Decay:  $A(r) \asymp r^{-\xi}$.}  
	Then  $A(r) + \varepsilon(r, n) \gtrsim \log^{-\xi}(n)$. 
	
	\smallskip
	\item[(iv)] \underline{Logarithmic Decay:  $A(r) \asymp \log^{-\xi}(r)$.} 	
	Then  $A(r) + \varepsilon(r, n) \gtrsim [\log\log(n)]^{-\xi}$. 
		
\end{itemize}
\end{theorem}

Comparing Theorems \ref{thm:binary_zero_tau} and \ref{thm:lower_bound_binary_tau2_zero}, we see the upper and lower bounds on $L_n$ match except when $A(r)\asymp e^{-\xi r}$ and $\xi\ge \log(M)$.  Also comparing both theorems to Theorem \ref{th:disc} with $\tau^2>0$, 
it is not surprising that the prediction error can achieve the same rate as that in Theorem \ref{th:disc} with  polynomial or logarithmic rates. 
This is because the estimation error when $\tau^2=0$ converges to zero more quickly than when $\tau^2 > 0$, as shown in \eqref{eq:upper_bound_rate}. 
However, 
in Theorem  \ref{thm:binary_zero_tau}  with polynomial or logarithmic rates, we allow $R_n  =  \log(n)/\log(M)$, and thus the number of unknown  parameters $M^{R_n}$ can be the same order as the sample size $n$. 
When the resolution bias $A(r)$ decays exponentially, the prediction error is able to achieve faster rate than  that in Theorem \ref{th:disc}. 

More importantly, 
when the resolution bias $A(r)$ has  a hard threshold or decays exponentially more quickly than $M^{-r}$, 
then the prediction error can achieve the usual rate $n^{-1}$, and the resolution $R_n$ is allowed to even be infinity. 
In particular, with infinite resolution, for each individual of interest, 
we are essentially trying to find the training samples that are closest to this individual (in terms of exactly the same covariates up to a certain resolution),
and use the average response from these training samples as our prediction. 
This is similar to the discussion in Section \ref{sec:linear_zero_tau}, where the usual bias-variance trade-off now puts all its considerations on the bias term in the deterministic world. 

Finally, we remark on the construction of model \eqref{eq:binary_model} with specific resolution bias $A(r)$ and ultimate risk $\tau^2$. 
Let $\beta_0$ be any constant, and $\beta_k = M/\sqrt{M-1}\cdot\sqrt{A(k-1) - A(k)}$ for $k\ge 1$.  
Define 
$
Y = \sum_{k=0}^\infty \beta_k  [\I(X_k=1) - M^{-1}] + \eta, 
$
where 
$X_1, X_2, \ldots$ are iid uniform on $\{1, 2, \ldots, M\}$, 
$\eta$ has mean zero and variance $\tau^2$, 
and $\vec{\bm{X}}_{\infty}$ and $\eta$ are independent.
Then the corresponding model \eqref{eq:binary_model} has the desired resolution bias and ultimate risk.

\section{From the Past to Future}\label{sec:main_practical}

\subsection{A logical consequence of the large-$p$-small-$n$ framework}\label{sec:unintend}

We appreciate the value of permitting $p$ to vary with $n$ as a \textit{mathematical strategy} for approximations,  because it can capture the magnitude of $p$ in 
relation
to $n$ toward determining which approximation terms can or cannot be ignored. But the same cannot be said about the 
\textit{statistical understanding} of the behavior of the resulting model in real applications. As discussed in Section \ref{sec:theory_linear} and further argued below,
this is not merely a logical or philosophical issue, but an issue of revealing correctly the actual behavior of our prediction models in practice. 

Specifically, for most practical problems, the underlying generative models, however the way nature adopts or we conceptualize them, precede our data collection effort. We therefore can permit our data collection process to be influenced by the generative model, but not vice versa. Nature does not alter its behavior in anticipation of the sample size we may choose. Consequently, when we assume $p> n$ and permit 
$n\rightarrow \infty$, it forces the logical conclusion that  $p=\infty$, if $p$ indexes a feature of nature's generative model. 

One may argue that $p$ in the large-$p$-small-$n$ asymptotics should not be conceptualized as an index of nature's behavior, but only as a human's approximation, like our primary resolution $R_n$. 
However, in the large-$p$-small-$n$ framework,  it is often assumed that the amount of total variation in the outcome that can be explained by $p$ predictors is a constant when we increase $n$ and hence $p$ because $p$ grows with $n$
 \citep[e.g.,][]{belkin2019reconciling,hastie2019surprises}. 
But if $p$ is meant to represent the number of predictors we humans use for predicting an outcome, then this assumption of fixed explainability defeats the purpose of using more predictors to improve the explainability of the predictors. When our mathematical formulation prohibits improvements,
the resulting theoretical results may mislead us when they are used for building our intuitions, even though they may provide useful mathematical approximations for computational purposes.  

As an illustration, let  $\delta_i^2=\E[(\mu_{i}-\mu_{i-1})^2]$, which measures the incremental contribution of the information in $\F_{i}$ in excess to that in $\F_{i-1}$ for explaining  the variability in $Y$ (over the population as defined by $\F_0$). Taking $r=0$ in (\ref{eq:keyi}), we have  
\begin{equation}\label{eq:keyii}
\Var(Y|\F_0)\equiv \sigma^2_0=\E[\sigma^2_{\infty}]+\sum_{i=1}^{\infty}\delta^2_i.
\end{equation}
This implies that, as $i$ increases, $\delta_i$ must be vanishingly small when $\Var(Y|\F_0)< \infty$, a trivial condition for virtually all real-life problems. This implies that the value of $p$ in the current large-$p$-small-$n$ regime cannot possibly be a sensible index of model complexity to be used in linear fashion, because increasing, say, from $p=2$ to $p=4$ could be far more consequential than moving from $p=22$ to $24$.  Yet it has been a common practice in the current literature of machine learning or statistics to plot prediction errors against $p$. 
It is therefore refreshing to see some recent work for studying and plotting the error against more meaningful indexes, such as a spectral decay in \cite{liang2019just}.  

More broadly, the predictability of any set of covariates depends on at least (I) how any of them influence the outcome in the absence of other predictors and (II) how they are related to each other. Neither of the two can be adequately captured in general by merely their size. In this article we therefore adopt
the direct measure of the decay rate in prediction error as we increase the resolution level
(e.g., employing more predictors). 
As demonstrated in Theorems \ref{th:cont}--\ref{thm:lower_bound_binary_tau2_zero},
this resolution decay rate plays a critical role in determining the optimal resolution, as well as in revealing further some problematic aspects of the current large-$p$-small-$n$ framework.

\subsection{Applications to personalized treatment}
This work was initiated by the need for establishing a statistically principled and scientifically sound theory of personalized treatments \citep{meng2014trio}. Therefore, we provide a very brief review of two types of methods in the literature.
The first type focuses on modeling the potential outcome of each patient given his or her covariates under each treatment arm, 
and it uses the resulting predictions to identify optimal treatment regimes; 
see \citet{murphy2003optimal, robins2004optimal,  zhao2009reinforcement} and \citet{Metalearners2019}. 
The second type focuses on a posited class of treatment regimes and tries to find the one that maximizes the overall outcome for all units; 
see \citet{zhao2012estimating, laber2015tree} and \citet{kosorok2019precision}. 

Our results provide useful theoretical guidance and insight to both types of applications, because they are applicable to different populations of interest or target individuals, as captured by $\F_0$ and $\F_\infty$ respectively. For either approach, the key feature of our framework is the complete avoidance of imposing a relationship between $p$ and $n$,  and hence it is suitable for investigating an arbitrarily large number of covariates.
Indeed, as we have seen in Sections~\ref{sec:theory_linear}-\ref{sec:theory_tree}, the MR framework can handle predictions with potentially infinitely many covariates.
  
\subsection{The method of sieves for infinite-dimension estimation}\label{sec:sieve}
The method of sieves \citep{grenander1981} deals with infinite-dimension estimation problems, by restricting the parameter estimation to a subset of the parameter space whose dimension grows with the sample size at some judiciously chosen rates \citep[e.g.,][]{Geman1982, shen1994,shen1997,johnstone2011gaussian}. 
The sequence of the subsets is then called a sieve, which can be viewed as a counterpart to MR's information filtration indexed by the resolution level $r$. 

Whereas wavelets and sieve methods share similar mathematical constructs,  our focus differs from the classical literature on sieves in several ways. 
First, we focus on prediction instead of parameter estimation. 
Second, for non-/semi-parametric estimation, the sieves for certain functional classes are well-understood. 
Under the MR framework, the resolution bias due to a sieve is generally more complicated, and the order of the covariates or equivalently the choice of sieve plays an important role in prediction error, as shown in Theorem \ref{th:ordering}. 
Third, 
we try
to understand both sufficient and necessary conditions for asymptotically optimal prediction (as in Theorems \ref{th:cont}--\ref{thm:lower_bound_binary_tau2_zero}),
where the literature on sieves typically focus on upper bounds for the estimation convergence rate.

\subsection{Much more work is needed}

A most needed theoretical insight is on deciding a reasonable ordering in practice, going beyond the results in Section~\ref{sec:ordering}. We do not expect any kind of ``automated choice" results, in theory or in practice, because of the no-free lunch principle. Since it is impossible to have a direct learning population, judgements and assumptions are inevitable. However, it is possible to obtain relatively general results for some specified (and practically meaningful) problems.
{\newchange Moreover, one may  borrow ideas from regularization methods in the large-$p$-small-$n$ framework, which can explore  all possible choices of the subsets of the covariates (e.g., $2^p$ in Lasso) without any pre-ordering. How to do so effectively within the MR framework is a challenging problem given $p$ potentially is $\infty$, although the observed number of covariates is always finite in practice.}

As mentioned earlier, we were intrigued by the world without variance. 
We wonder, without ever being able to determine which world we are in, how could we be allowed to see its consequences? The answer seems to lie in the fact that $\sigma^2_\infty=0$ is a necessary but not sufficient condition for the no bias-variance trade-off phenomenon. 
As seen in the bottom row of Figure~\ref{fig:estimate_pred_loss}, 
this phenomenon
did not occur when $A(r)$ decays too slowly, e.g., polynomially or logarithmically. 
{\newchange Note that we can always artificially create infinitely many covariates by certain series expansions of the basic covariates.  
The observation in the world without variance should motivate us to investigate the performance of non-parametric sieve regression when the response is indeed a deterministic function of the covariates.}
This observation also suggests the possibility for a black-box procedure to resist (empirically verifiable) over-fitting, when the number of patterns detectable with sufficient frequencies is far fewer than theoretically possible. In such cases, \textit{exhaustive learning} is practically possible with sufficiently large training samples, hence there is no need 
for 
``intrinsic variance" to capture model imperfection, avoiding the creation of a petri dish for over-fitting. This possibility suggests a systematic investigation of the deterministic MR framework for complex machine learning models to see if it indeed provides an alternative explanation of the over-fitting resistant nature of these models.

\section*{Acknowledgments}

The authors thank colleagues, 
especially James Bailie, Robin Gong, Tengyuan Liang, and Kai Zhang, 
as well as several meticulous reviewers for encouragements and comments, which have greatly improved both the content and presentation.  
They also acknowledge partial financial support from NSF grants.

\bibliographystyle{plainnat}
\bibliography{multiresolution}


\newpage
\setcounter{equation}{0}
\setcounter{section}{0}
\setcounter{figure}{0}
\setcounter{example}{0}
\setcounter{proposition}{0}
\setcounter{corollary}{0}
\setcounter{theorem}{0}
\setcounter{table}{0}
\setcounter{condition}{0}

\renewcommand {\theproposition} {A\arabic{proposition}}
\renewcommand {\theexample} {A\arabic{example}}
\renewcommand {\thefigure} {A\arabic{figure}}
\renewcommand {\thetable} {A\arabic{table}}
\renewcommand {\theequation} {A\arabic{section}.\arabic{equation}}
\renewcommand {\thelemma} {A\arabic{lemma}}
\renewcommand {\thesection} {A\arabic{section}}
\renewcommand {\thetheorem} {A\arabic{theorem}}
\renewcommand {\thecorollary} {A\arabic{corollary}}
\renewcommand {\thecondition} {A\arabic{condition}}
\renewcommand {\thepage} {A\arabic{page}}

\setcounter{page}{1}

\begin{center}
\bf \LARGE
Appendices
\end{center}


\section{Proof of Theorem~\ref{th:ordering}}

Recall that $\vec{\bm{X}}_{\infty}$ and $\vec{\bm{X}}'_{\infty}$ are two orderings of the covariates with resolution biases $A(\cdot)$ and $A'(\cdot)$, estimation errors $\varepsilon(\cdot, n)$ and $\varepsilon'(\cdot, n)$, and rate-optimal resolutions $R_n$ and $R_n'$. 
Then the prediction error, excluding the ultimate risk, using resolution $R_n$ under the new ordering satisfies
\begin{align*}
    A'(R_n) + \varepsilon'(R_n, n) 
    & = A'(R_n) + \{A'(R_n) + \tau^2\} \frac{\varepsilon'(R_n, n)}{A'(R_n) + \tau^2}
    \\
    & \asymp
    A'(R_n) + \{A'(R_n) + \tau^2\} \frac{\varepsilon(R_n, n)}{A(R_n) + \tau^2}
    \\
    & = \frac{A'(R_n)}{A(R_n)} A(R_n)  + \frac{A'(R_n) + \tau^2}{A(R_n) + \tau^2} \varepsilon(R_n, n). 
\end{align*}
Thus, a sufficient condition for the new order to achieve the optimal rate under the original ordering is that $A’(R_n) = O(A(R_n))$. 

By the property of prediction function families and the definition of $M_r(A, A')$, 
families of prediction functions at resolution $\infty$ are the same under both orderings, 
and any prediction function at resolution $r-M_r(A, A')$ under ordering $\vec{\bm{X}}_{\infty}$ belongs to the family of prediction functions at resolution $r$ under ordering $\vec{\bm{X}}'_{\infty}$. 
These imply that
\begin{align*}
A'(r) & = 
\min_{\bm{\theta}_r}\E[\loss_{\odot}(Y, g(\vec{\bm{X}}'_{r};\bm{\theta}_{r}))
-
\min_{\bm{\theta}_\infty}\E[\loss_{\odot}(Y, g(\vec{\bm{X}}'_{\infty};\bm{\theta}_{\infty}))
\\
& \le 
\min_{\bm{\theta}_{r-M_r(A, A')}}\E[\loss_{\odot}(Y, g(\vec{\bm{X}}_{r-M_r(A, A')};\bm{\theta}_{r-M_r(A, A')}))
-
\min_{\bm{\theta}_\infty}\E[\loss_{\odot}(Y, g(\vec{\bm{X}}_{\infty};\bm{\theta}_{\infty}))\\
& = 
A(r-M_r(A, A')). 
\end{align*}
Thus, 
a sufficient condition for $A'(R_n) = O(A(R_n))$ is that 
$A(r_n-M_{r_n}(A, A')) = O(A(r_n))$ for any sequence $r_n \rightarrow \infty$. Below we consider three decay rates for $A(\cdot)$. 
\begin{itemize}
	\item[(i)] 
	For any sequence $r_n\rightarrow \infty$, because $A(r) \asymp e^{-\xi r}$ and $\limsup_{n\rightarrow\infty}M_{r_n}(A, A') \le \limsup_{r \rightarrow \infty} M_r(A, A')\le C$,
	\begin{align*}
	\frac{A(r_n-M_{r_n}(A, A'))}{A(r_n)}
	& = 
	O\left( 
	\frac{e^{-\xi(r_n-M_{r_n}(A, A'))}}{e^{-\xi r_n}}
	\right)
	= 
	O\left( e^{\xi M_{r_n}(A, A')}
	\right)=O(1).
	\end{align*}
	
	\item[(ii)] 
	For any sequence $r_n\rightarrow \infty$, because $A(r)\asymp r^{-\xi}$ and $\limsup_{n \rightarrow \infty} M_{r_n}(A, A')/r_n \le \limsup_{r \rightarrow \infty} M_r(A, A')/r < 1$, 
	\begin{align*}
	\frac{A(r_n-M_{r_n}(A, A'))}{A(r_n)}
	& = 
	O\left( 
	\frac{(r_n-M_{r_n}(A, A'))^{-\xi}}{
		r_n^{-\xi}
	}
	\right)
	= 
	O\left(
	\left(
	1 - \frac{M_{r_n}(A, A')}{r_n}
	\right)^{-\xi}
	\right)=O(1).
	\end{align*}
	
	\item[(iii)] 
	For any sequence $r_n\rightarrow \infty$, because $A(r)\asymp \log^{-\xi}(r)$ and $a_{r_n} = O(1)$, 
	\begin{align*}
	\frac{A(r_n-M_{r_n}(A, A'))}{A(r_n)}
	& = 
	O\left( 
	\frac{
		\log^{-\xi} (r_n-M_{r_n}(A, A'))
	}{
		\log^{-\xi}(r_n)
	}
	\right)
	= 
	O\left( 
	\left\{
	\frac{
		\log (r_n-M_{r_n}(A, A'))
	}{
		\log(r_n)
	}
	\right\}^{-\xi}
	\right)\\
	& = 
	O\left( 
	\left\{
	\frac{
		\log (r_n^{1/a_{r_n}})
	}{
		\log(r_n)
	}
	\right\}^{-\xi}
	\right)
	 =
	O\left(
	a_{r_n}^{\xi}
	\right) = O(1).
	\end{align*}
\end{itemize}

\section{Technical Details for the Linear Models in Section \ref{sec:theory_linear}}\label{sec:proof_linear}

\subsection{Expression for estimation error $\varepsilon(r,n)$}\label{sec:proof_est_error_linear}

To facilitate the discussion, we introduce the Gram--Schmidt orthogonalization of the original covariates $\vec{\bm{X}}_{\infty}$, and denote the orthogonalized covariates by $\vec{\bm{W}}_{\infty}$. 
Specially, $W_0 = 1$ is a constant, 
$W_1, W_2, \ldots$ are i.i.d.\  standard normal random variables, 
and for any $r\ge 1$, 
$\vec{\bm{W}}_{r}$ and $\vec{\bm{X}}_{r}$ are linear transformations of each other. 
The joint distribution of $(Y, \vec{\bm{W}}_{\infty})$ also follows a linear model: 
$Y = \delta_0 + \sum_{k=1}^R \delta_kW_{k} + \eta,$
where $\delta_k^2$ is equivalently the variance of $Y$ explained by $X_k$ in addition to that by the previous covariates, 
$\eta$ has conditional mean 0 and conditional variance $\tau^2 >0$ given $\vec{\bm{W}}_{\infty}$. 
All of these quantities $\delta_k^2$, $\eta$ and $\tau^2$ are the same as defined in Section \ref{sec:linear_model}. 

Let $\{(y_{i}, \vec{\bm{w}}_{i\infty}): 1\le i \le n\}$ be the training set with orthogonalized covariates, 
and $\hat{\bm{\delta}}_r$ be the corresponding least squares coefficient at resolution $r$. 
Define $\bm{B}_r$ as the linear transformation from $\vec{\bm{X}}_{r}$ to $\vec{\bm{W}}_{r}$, i.e., $\vec{\bm{W}}_{r} = \bm{B}_r \vec{\bm{X}}_{r}$. 
Clearly, $ \bm{\theta}^*_r = \bm{B}_r^\top \bm{\delta}_r$ and $\hat{\bm{\theta}}_r  = \bm{B}_r^\top \hat{\bm{\delta}}_r$. 
From \eqref{eq:quad}, the estimation error $\varepsilon(r, \trainingset_n)$ has the following equivalent forms:
\begin{align*}
    \varepsilon(r, \trainingset_n)
    & = 
    \E \big\{ \big( \vec{\bm{X}}_{r}^\top \hat{\bm{\theta}}_r - \vec{\bm{X}}_{r}^\top \bm{\theta}^*_r \big)^2 \big\}
    = 
    (\hat{\bm{\theta}}_r - \bm{\theta}^*_r)^\top \E \big( \vec{\bm{X}}_{r} \vec{\bm{X}}_{r}^\top \big) (\hat{\bm{\theta}}_r - \bm{\theta}^*_r). 
\end{align*}
By the construction of $\vec{\bm{W}}_{\infty}$, 
the estimation error further simplifies to 
\begin{align*}
    \varepsilon(r, \trainingset_n)
    & = 
    (\hat{\bm{\theta}}_r - \bm{\theta}^*_r)^\top \E \big( \vec{\bm{X}}_{r} \vec{\bm{X}}_{r}^\top \big) (\hat{\bm{\theta}}_r - \bm{\theta}^*_r)
    = 
    (\hat{\bm{\delta}}_r - \bm{\delta}_r)^\top \bm{B}_r \E \big( \vec{\bm{X}}_{r} \vec{\bm{X}}_{r}^\top \big)
    \bm{B}_r^\top 
    (\hat{\bm{\delta}}_r - \bm{\delta}_r)
    \\
    & = 
    (\hat{\bm{\delta}}_r - \bm{\delta}_r)^\top \E \big( \vec{\bm{W}}_{r} \vec{\bm{W}}_{r}^\top \big)
    (\hat{\bm{\delta}}_r - \bm{\delta}_r)
    = (\hat{\bm{\delta}}_r - \bm{\delta}_r)^\top (\hat{\bm{\delta}}_r - \bm{\delta}_r). 
\end{align*}%
Let $\tilde{\bm{Y}} = (y_1, \ldots, y_n)^\top$ denote all responses in the training set,  $\tilde{\bm{W}}_{r}$ be the $n\times (r+1)$ matrix consisting of $\vec{\bm{w}}_{ir}$'s, and $\tilde{\bm{W}}_{1:r}$ be the $n\times r$ submatrix consisting of the last $r$ columns of $\tilde{\bm{W}}_{r}$, i.e., 
$
\tilde{\bm{W}}_{r} = (\bm{1}_n, \tilde{\bm{W}}_{1:r}). 
$

We first simplify $\varepsilon(r, n)$. 
By definition, 
$
\hat{\bm{\delta}}_r = \left(
\tilde{\bm{W}}_{r}^\top \tilde{\bm{W}}_{r}
\right)^{-1}
\tilde{\bm{W}}_{r}^\top \tilde{\bm{Y}},
$
and thus 
$
\hat{\bm{\delta}}_r - \bm{\delta}_r = 
\left(
\tilde{\bm{W}}_{r}^\top \tilde{\bm{W}}_{r}
\right)^{-1}
\tilde{\bm{W}}_{r}^\top
\left( \tilde{\bm{Y}} - \tilde{\bm{W}}_{r} \bm{\delta}_r\right).
$ Conditional on $\tilde{\bm{W}}_{r}$, 
$\tilde{\bm{Y}} - \tilde{\bm{W}}_{r} \bm{\delta}_r$ 
follows a multivariate normal distribution with mean zero and covariance matrix 
$\{A(r)+\tau^2\} \bm{I}_{n}$. This implies 
\begin{align*}
    & \quad \ \E_n \left[\varepsilon(r, \trainingset_n) \mid \tilde{\bm{W}}_{r} \right]=
\text{tr}
\left\{
\E_n
\left[
\left( \hat{\bm{\delta}}_r - \bm{\delta}_r \right)
\left( \hat{\bm{\delta}}_r - \bm{\delta}_r \right)^\top
\big| \tilde{\bm{W}}_{r}
\right]
\right\}\\
& = 
\text{tr}
\left\{
\left(
\tilde{\bm{W}}_{r}^\top \tilde{\bm{W}}_{r}
\right)^{-1}
\tilde{\bm{W}}_{r}^\top
\cdot
\E_n
\left[
\left( \tilde{\bm{Y}} - \tilde{\bm{W}}_{r} \bm{\delta}_r\right)
\left( \tilde{\bm{Y}} - \tilde{\bm{W}}_{r} \bm{\delta}_r\right)^\top
\big| \tilde{\bm{W}}_{r}
\right]
\cdot
\tilde{\bm{W}}_{r}
\left(
\tilde{\bm{W}}_{r}^\top \tilde{\bm{W}}_{r}
\right)^{-1}
\right\}
\\
& = 
\text{tr}
\left\{
\left(
\tilde{\bm{W}}_{r}^\top \tilde{\bm{W}}_{r}
\right)^{-1}
\tilde{\bm{W}}_{r}^\top
\cdot
\left[A(r) + \tau^2\right] \bm{I}_{n}
\cdot
\tilde{\bm{W}}_{r}
\left(
\tilde{\bm{W}}_{r}^\top \tilde{\bm{W}}_{r}
\right)^{-1}
\right\}
\\
& 
= \left[A(r)+ \tau^2\right]
\text{tr}\left\{
\left(\tilde{\bm{W}}_{r}^\top \tilde{\bm{W}}_{r}\right)^{-1}
\right\}. 
\end{align*}
Consequently, the estimation error $\varepsilon(r, n)$ reduces to 
\begin{align*}
\varepsilon(r, n) & = \E_n [\varepsilon(r, \trainingset_n)] = 
\E_n \left\{
\E_n \left[ \varepsilon(r, \trainingset_n) \mid \tilde{\bm{W}}_{r} \right]
\right\} = 
[A(r)+ \tau^2]
\cdot 
\text{tr}\left\{
\E_n \left[
\left(\tilde{\bm{W}}_{r}^\top \tilde{\bm{W}}_{r}\right)^{-1}
\right]\right\}. 
\end{align*}

We now give a closed form expression for the estimation error $\varepsilon(r, n)$. 
The sample mean and sample covariance matrix of the $\vec{\bm{w}}_{i,1:r}$'s are, respectively,
\begin{align*}
\bar{\bm{W}}_{1:r} = \frac{1}{n}\sum_{i=1}^n \vec{\bm{w}}_{i,1:r}, \ \ 
\bm{S}^2_{1:r} = \frac{1}{n-1}\sum_{i=1}^n 
\left(
\vec{\bm{w}}_{i,1:r} - \bar{\bm{W}}_{1:r}
\right)
\left(
\vec{\bm{w}}_{i,1:r} - \bar{\bm{W}}_{1:r}
\right)^\top. 
\end{align*}
Using the block matrix inversion formula,  
we obtain
\begin{align*}
(\tilde{\bm{W}}_{r}^\top \tilde{\bm{W}}_{r})^{-1} & = 
\begin{pmatrix}
n & n \bar{\bm{W}}_{1:r}^\top\\
n \bar{\bm{W}}_{1:r} & (n-1)\bm{S}^2_{1:r}+n\bar{\bm{W}}_{1:r}\bar{\bm{W}}_{1:r}^\top
\end{pmatrix}^{-1}\\
& =  
\begin{pmatrix}
n^{-1} + \bar{\bm{W}}_{1:r}^\top 
\left\{
(n-1)\bm{S}^2_{1:r}
\right\}^{-1}
\bar{\bm{W}}_{1:r} & 
- \bar{\bm{W}}_{1:r}^\top \left\{
(n-1)\bm{S}^2_{1:r}
\right\}^{-1}\\
-  \left\{
(n-1)\bm{S}^2_{1:r}
\right\}^{-1}
\bar{\bm{W}}_{1:r} & 
\left\{
(n-1)\bm{S}^2_{1:r}
\right\}^{-1}
\end{pmatrix}.
\end{align*}
Standard properties of the multivariate normal imply that 
(i) $\bar{\bm{W}}_{1:r} \ind \bm{S}^2_{1:r}$, 
(ii) $\left\{
(n-1)\bm{S}^2_{1:r}
\right\}^{-1}$
follows the standard inverse Wishart distribution with degrees of freedom $n-1$, 
and 
(iii) 
$n(n-1)\bar{\bm{W}}_{1:r}^\top 
\left\{
(n-1)\bm{S}^2_{1:r}
\right\}^{-1}
\bar{\bm{W}}_{1:r}$
follows 
Hotelling's $T$-squared distribution with dimensionality parameter $r$ and  degrees of freedom $n-1$. 
Consequently,
\begin{align*}
\E_n\left[\left(\tilde{\bm{W}}_{r}^\top \tilde{\bm{W}}_{r}\right)^{-1}
\right] =  \frac{1}{n-r-2}
\begin{pmatrix}
\frac{n-2}{n} & 0\\
0 & \bm{I}_r
\end{pmatrix},
\end{align*}
which yields \eqref{eq:epsilon_r_n_linear}. 

From \eqref{eq:epsilon_r_n_linear}, we 
see immediately that $\varepsilon(r,n) \ge \tau^2\cdot r/n$. Hence, when $\tau>0$,
a necessary condition for the estimation error to be $o(1)$ 
is that $r/n = o(1)$. Furthermore, 
when $r/n = o(1)$ and $\tau>0$, we have 
\begin{align*}
\varepsilon(r,n) & =
\left[A(r)+\tau^2 \right]
\cdot 
\frac{ 1/n - 2/n^2 +r/n}{1-r/n-2/n}
\asymp \frac{r}{n}. 
\end{align*}
Clearly, this result does not hold in general when $\tau=0.$

\subsection{Asymptotic rate of $\varepsilon(r, \trainingset_n)$ when $\tau^2 > 0$}

First, we give lower and upper bounds for the estimation error $\varepsilon(r, \trainingset_n)$. 
Let $\tilde{\bm{W}}_{r} = \bm{U}\bm{D}\bm{V}^\top$ be the singular value decomposition of $\tilde{\bm{W}}_{r}$, where $\bm{U}\in \mathbb{R}^{n\times(r+1)},$ $\bm{V}\in \mathbb{R}^{(r+1)\times(r+1)}$, and $\bm{D}$ is a $(r+1)\times (r+1)$ matrix with decreasing diagonal elements $d_1\geq d_2\geq \ldots \geq d_{r+1}.$ 
The difference between $\hat{\bm{\delta}}_r$ and $\bm{\delta}_r$ has the following equivalent forms: 
\begin{align*}
\hat{\bm{\delta}}_r - \bm{\delta}_r & = 
\left(
\tilde{\bm{W}}_{r}^\top \tilde{\bm{W}}_{r}
\right)^{-1}
\tilde{\bm{W}}_{r}^\top
\left( \tilde{\bm{Y}} - \tilde{\bm{W}}_{r} \bm{\delta}_r\right) = 
\left(
\bm{V} \bm{D} \bm{U}^\top \bm{U}\bm{D}\bm{V}^\top
\right)^{-1}
\bm{V} \bm{D} \bm{U}^\top
\left( \tilde{\bm{Y}} - \tilde{\bm{W}}_{r} \bm{\delta}_r\right)\\
& = (\bm{V}^\top)^{-1} \bm{D}^{-1} \bm{U}^\top 	\left( \tilde{\bm{Y}} - \tilde{\bm{W}}_{r} \bm{\delta}_r\right),
\end{align*}
and thus the estimation error reduces to 
\begin{align*}
\varepsilon(r, \trainingset_n) & = 
\left\| \hat{\bm{\delta}}_r - \bm{\delta}_r \right\|_2^2
=
\left\| (\bm{V}^\top)^{-1} \bm{D}^{-1} \bm{U}^\top 	\left( \tilde{\bm{Y}} - \tilde{\bm{W}}_{r} \bm{\delta}_r\right) \right\|_2^2
= \left\| \bm{D}^{-1} \bm{U}^\top 	\left( \tilde{\bm{Y}} - \tilde{\bm{W}}_{r} \bm{\delta}_r\right) \right\|_2^2. 
\end{align*}
Define $\bm{\zeta} = \{A(r)+\tau^2\}^{-1/2} \bm{U}^\top( \tilde{\bm{Y}} - \tilde{\bm{W}}_{r} \bm{\beta}_r)$. 
We can verify that $\bm{\zeta} \sim \mathcal{N}(\bm{0}, \bm{I}_{r+1})$, and the estimation error reduces to 
\begin{align}\label{eq:bound_estimation_linear}
\varepsilon(r, \trainingset_n) & = 
\left\| \left[A(r)+\tau^2 \right]^{1/2} \bm{D}^{-1} \bm{\zeta} \right\|_2^2
=
\left[A(r)+\tau^2 \right] \sum_{i=1}^{r+1} \frac{\zeta_i^2}{d_i^2}
\ \ 
\begin{cases}
\ge 	\tau^2 \|\bm{\zeta}\|_2^2 \cdot d_1^{-2}\\
\le  \left[A(0) + \tau^2 \right]\|\bm{\zeta}\|_2^2 \cdot d_{r+1}^{-2}
\end{cases},
\end{align}
where $d_1$ and $d_{r+1}$ are the largest and smallest singular values of $\tilde{\bm{W}}_{r}$, respectively.

Second, we study probability bounds for the singular values $d_1$ and $d_{r+1}$. 
Note that 
$
\bm{U}\bm{D}^2\bm{U}^T = \tilde{\bm{W}}_{r}\tilde{\bm{W}}_{r}^\top = 
(\bm{1}_n, \tilde{\bm{W}}_{1:r})
(\bm{1}_n, \tilde{\bm{W}}_{1:r})^\top
= \bm{1}_n\bm{1}_n^\top + \tilde{\bm{W}}_{1:r}\tilde{\bm{W}}_{1:r}^\top.
$
Let $s_{\min}$ and $s_{\max}$  be the smallest and largest singular values of $\tilde{\bm{W}}_{1:r}$. 
Because the smallest and largest 
eigenvalues 
of $ \bm{1}_n\bm{1}_n^\top$ are, respectively, $0$ and $n$, 
from Weyl's inequality, we have 
$
s_{\min}^2 + 0 \le d_{r+1}^2 \le d_1^2\leq s_{\max}^2 + n.
$
From \citet[][Corollary 5.35]{vershynin2010introduction}, for every $t\geq 0$, with probability at least $1-2\exp(-t^2/2)$ one has 
$\sqrt{n}-\sqrt{r} - t\le s_{\min} \le s_{\max} \leq \sqrt{n} + \sqrt{r} +t.$ 
Letting $t = \sqrt{n}$, we have 
\begin{align}\label{eq:bound_d_1}
    \frac{d_1^2}{r} \le
\frac{s^2_{\max}}{r} + \frac{n}{r}
=
O_{\pr}\left\{ \frac{(2\sqrt{n}+\sqrt{r})^2}{r} + \frac{n}{r} \right\}
= 
O_{\pr}\left\{  
\frac{(2+\sqrt{r/n})^2+1}{r/n}
\right\}.
\end{align}
Letting $t = \sqrt{n}/2$, if $r<n/4$, we have 
\begin{align}\label{eq:bound_d_r_plus_1}
\frac{r}{d_{r+1}^{2}} \le \frac{r}{s_{\min}^{2}} = 
O_{\pr}\left\{
\frac{r}{
	(\sqrt{n}/2 - \sqrt{r})^2
}
\right\}
= 
O_{\pr}\left\{
\frac{r/n}{
	(1/2 - \sqrt{r/n})^2
}
\right\}. 
\end{align}

Third, we study the asymptotic rate of $\varepsilon(r, \trainingset_n)$. 
From \eqref{eq:bound_estimation_linear} and \eqref{eq:bound_d_1}, 
\begin{align}\label{eq:lower_bound_estimation_linear}
\varepsilon(r, \trainingset_n)^{-1} = 
O_{\pr}\left( \tau^{-2} \frac{r}{\|\bm{\zeta}\|_2^2} \cdot \frac{d_1^{2}}{r} \right)
= 
O_{\pr} \left(
\frac{d_1^{2}}{r}
\right)
= O_{\pr}\left\{  
\frac{(2+\sqrt{r/n})^2+1}{r/n}
\right\}.
\end{align}
Thus, 
for any sequence $\{r_n\}$, 
a necessary condition for $\varepsilon(r_n,\trainingset_n)=o_{\pr}(1)$ is $r_n = o(n).$
If $r = o(n)$, from  \eqref{eq:bound_estimation_linear} and \eqref{eq:bound_d_r_plus_1}, 
\begin{align}\label{eq:upper_bound_estimation_linear}
\varepsilon(r, \trainingset_n) & = O_{\pr}\left(
\left[A(0) + \tau^2 \right]\frac{\|\bm{\zeta}\|_2^2}{r} \cdot \frac{r}{d_{r+1}^{2}}
\right)
= 
O_{\pr}\left(
\frac{r}{d_{r+1}^{2}}
\right)
= 
O_{\pr}\left\{
\frac{r/n}{
	(1/2 - \sqrt{r/n})^2
}
\right\}. 
\end{align}
When $r = o(n)$, 
\eqref{eq:lower_bound_estimation_linear} implies that 
$
\varepsilon(r, \trainingset_n)^{-1} = O_{\pr}\{(r/n)^{-1}\}, 
$
and 
\eqref{eq:upper_bound_estimation_linear} implies that 
$
\varepsilon(r, \trainingset_n) = O_{\pr}(r/n).  
$
Therefore, we have 
$
\varepsilon(r, \trainingset_n) \overset{p}{\asymp} r/n
$
when $r= o(n).$

\section{Proof of Theorem~\ref{th:cont}}
All cases are proved in two steps. First we show by contradiction that the error rate $A(r_n)+r^{\alpha}/n$
cannot be 
faster
than a specified rate $h(n)$ for any resolution sequence $r_n$. We then show by construction that there exists a sequence $R_n$ that reaches the rate $h(n)$. 
\begin{itemize}
	\item[(i)] (Hard Thresholding). 
	First, for any sequence $\{r_n\}$, 
	\begin{align*}
	\liminf_{n\rightarrow \infty}
	\frac{A(r_n) + r_n^{\alpha}/n}{1/n}>0.  
	\end{align*}
	If this inequality is false, then 
	there must exist a subsequence $\{r_{n_k}\}$ such that 
	\begin{align*}
	\lim_{k\rightarrow \infty}
	\frac{A(r_{n_k})}{1/{n_k}}=0 \ 
	\text{ and }  \ 
	\lim_{k\rightarrow \infty}
	\frac{ r_{n_k}^{\alpha}/{n_k}}{1/{n_k}}= \lim_{k\rightarrow \infty}
	r_{n_k}^{\alpha} =0,
	\end{align*}
	implying that 
	(i) $\liminf_{k\rightarrow \infty}r_{n_k}\geq r_0>0$ and (ii) $\lim_{k\rightarrow \infty}r_{n_k}=0$, a contradiction.

	Second, it is easy to verify that $A(r_n) + \varepsilon(r_n,n) \asymp 1/n$ if and only if 
	$r_n = O(1)$ and $\liminf_{n\rightarrow \infty} r_n \geq r_0$, which is the $R_n$ for case (i).

	\item[(ii)] (Exponential Decay).  
	First,  for any sequence $\{r_n=a_n \log(n)\}$, 
	\begin{align}\label{eq:proof_poly_opt_expo_goal}
	\liminf_{n\rightarrow \infty}
	\frac{e^{-\xi r_n} + r_n^{\alpha}/n}{\log^{\alpha}(n)/n}
	=
	\liminf_{n\rightarrow \infty} 
	\left\{
	\frac{n^{1-\xi a_n}}{\log^{\alpha}(n)} + a_n^{\alpha}
	\right\}
	>0.  
	\end{align} 
	If this inequality is false, then there must exist a subsequence $\{a_{n_k}\}$ such that 
	\begin{align*}
	\lim_{k\rightarrow \infty} 
	\frac{n_k^{1-\xi a_{n_k}}}{\log^{\alpha}(n_k)} = 0 \ 
	\text{ and } \ 
	\lim_{k\rightarrow \infty}  a_{n_k}^{\alpha}=0,
	\end{align*} 
	implying that (i) $\liminf_{k\rightarrow\infty} a_{n_k} \geq \xi^{-1}>0$ and (ii) $\lim_{k\rightarrow \infty}  a_{n_k}=0$, a contradiction.

	Second, by the equality in \eqref{eq:proof_poly_opt_expo_goal}, $A(r_n) + \varepsilon(r_n,n) \asymp \log^{\alpha}(n)/n$ if and only if $\{r_n = a_n\log(n)\}$ satisfies 
	\begin{align}\label{eq:cont_exp_proof}
	a_n = O(1), \quad \frac{n^{1-\xi a_n}}{\log^\alpha(n)} = O(1).
	\end{align}
	Note that the second condition in \eqref{eq:cont_exp_proof} implies that $a_n^{-1} = O(1)$. Thus, \eqref{eq:cont_exp_proof} is also equivalent to 
	\begin{align*}
		a_n \asymp 1, \quad \frac{n^{1-\xi a_n}}{\log^\alpha(n)} = O(1).
	\end{align*}

	\item[(iii)] (Polynomial Decay). 
	First, for any sequence $\{r_n = a_n n^{1/(\alpha+\xi)} \}$, 
	\begin{align}\label{eq:proof_poly_opt_poly}
	\liminf_{n\rightarrow \infty}
	\frac{r_n^{-\xi} + r_n^{\alpha}/n}{n^{-\xi/(\alpha+\xi)}}
	=
	\liminf_{n\rightarrow \infty} 
	\left(
	a_n^{-\xi} + a_n^{\alpha}
	\right)
	>0.  
	\end{align}
	If this inequality is false,  then there must exist a subsequence $\{a_{n_k}\}$ such that 
	$
	\lim_{k\rightarrow \infty} a_{n_k}^{-1} = 0
	$
	and 
	$
	\lim_{k\rightarrow \infty} a_{n_k} = 0,
	$ a clear contradiction. 
	
	Second, by the equality in \eqref{eq:proof_poly_opt_poly},  $A(r_n)+\varepsilon(r_n,n) \asymp n^{-\xi/(\alpha+\xi)}$ if and only if $a_n\asymp 1$.

	\item[(iv)] (Logarithmic Decay). 
	First,  for any sequence $\{r_n = a_n n^{1/\alpha}/\log^{\xi/\alpha}(n)\}$, 
	\begin{align}\label{eq:proof_poly_opt_log_goal_para}
	\liminf_{n\rightarrow \infty}\frac{\log^{-\xi}(r_n) + r_n^\alpha/n}{\log^{-\xi}(n)} & = 
	\liminf_{n\rightarrow \infty}
	\left\{
	\left(
	\frac{\log(a_n)}{\log(n)} + \alpha^{-1} - \frac{\xi}{\alpha}
	\frac{\log\log(n)}{\log(n)}
	\right)^{-\xi}
	+ a_n^{\alpha}
	\right\} >0. 
	\end{align}
	If this inequality is false,  then there must exist a subsequence $\{a_{n_k}\}$ such that
	\begin{align*}
	\lim_{k\rightarrow \infty} \left(
	\frac{\log(a_{n_k})}{\log(n_k)} + \alpha^{-1} - \frac{\xi}{\alpha}
	\frac{\log\log(n_k)}{\log(n_k)}
	\right)^{-\xi} = 0 \ 
	\text{ and } \ 
	\lim_{k\rightarrow \infty}a_{n_k}^{\alpha}=0, 
	\end{align*}
	which implies that (i) $\lim_{k\rightarrow\infty}a_{n_k} = \infty$  
	and (ii) $\lim_{k\rightarrow\infty}a_{n_k} = 0$, a clear contradiction.

	Second, by the equality in  \eqref{eq:proof_poly_opt_log_goal_para}, $A(r_n) + r_n^\alpha/n \asymp \log^{-\xi}(n)$ if and only if 
	\begin{align*}
	a_n = O(1), \quad 
    \liminf_{n\rightarrow \infty} \frac{\log(a_n)}{\log(n)} > -\alpha^{-1}.
	\end{align*} 
\end{itemize}

\section{Proof of Theorem \ref{th:contzero}}
From \eqref{eq:pred_loss_linear} and the fact that $\tau^2=0$, the prediction error has the following equivalent form at resolution $r$: 
\begin{align*}
\E_n \E\{ [Y - g(\vec{\bm{X}}_{r}, \hat{\bm{\theta}}_r)]^2\} & = A(r) + \varepsilon(r, n)
= A(r)
\frac{(n+1)(n-2)}{n(n-r-2)}.
\end{align*}
Below we consider the four cases depending on the decay rate of $A(r)$. 
\begin{itemize}
	\item[(i)] (Hard Thresholding). 
	First, for any sequence $\{r_n\}$, $A(r_n) + \varepsilon(r_n, n) \ge 0$. 
	
	Second, it is easy to verify that $A(r_n) + \varepsilon(r_n, n) = 0$ for sufficiently large $n$ if and only if 	$\liminf_{n\rightarrow \infty} r_n\geq r_0$ and $r_n \le n-3$.

	\item[(ii)] (Exponential Decay). 
	First, for any sequence $\{r_n = n - a_n\}$, 
	\begin{align}\label{eq:expo_0}
	\liminf_{n \rightarrow \infty} 
	\frac{
		e^{-\xi r_n}
		\frac{(n+1)(n-2)}{n(n-r-2)}
	}{
		n e^{-\xi n}
	}
	= 
	\liminf_{n \rightarrow \infty}  
	\frac{e^{\xi a_n}}{a_n-2} 
	\ge
	\liminf_{n \rightarrow \infty} \frac{\xi a_n + 1}{a_n-2}
	\ge \xi  > 0. 
	\end{align}
	Second, by the equality in \eqref{eq:expo_0}, 
	$A(r_n) + \varepsilon(r_n, n) \asymp n e^{-\xi n}$ if and only if $a_n = O(1)$ and $a_n\ge 3$; recall $a_n=n-r_n$ takes integer value only.  
	
	\item[(iii)] (Polynomial Decay). 
	First, by an inequality for the weighted arithmetic mean and weighted geometric mean, 
	\begin{align*}
	(r/\xi)^{\xi} (n-r-2)
	\le 
	\left( \frac{
		\xi \cdot r/\xi + n-r-2	
	}{
		\xi+1
	} \right)^{\xi+1}
	\le 
	\left( \frac{n}{\xi+1} \right)^{\xi+1},
	\end{align*}
	which immediately implies that 
	$
	r^{\xi} (n-r-2) \le \xi^\xi/(\xi+1)^{\xi+1} \cdot n^{\xi+1}. 
	$
	Therefore, for any sequence $\{r_n\}$, 
	\begin{align*}
	\liminf_{n \rightarrow \infty} 
	\frac{
		r_n^{-\xi}
		\frac{(n+1)(n-2)}{n(n-r-2)}
	}{
		n^{-\xi}
	}
	=
	\liminf_{n \rightarrow \infty} 
	\frac{
		n^{(\xi+1)}
	}{
		r_n^{\xi}
		(n-r-2)
	}
	\ge \frac{(\xi+1)^{\xi+1}}{\xi^\xi} > 0. 
	\end{align*}
	Second, 
	for any sequence $\{r_n = a_n n\}$, 
	\begin{align*}
	\frac{
		r_n^{-\xi}
		\frac{(n+1)(n-2)}{n(n-r-2)}
	}{
		n^{-\xi}
	}
	= \frac{(n+1)(n-2)}{n^2} \cdot a_n^{-\xi} \cdot \frac{1}{1 - a_n - 2n^{-1}} \asymp a_n^{-\xi} \cdot \frac{1}{1 - a_n - 2n^{-1}}. 
	\end{align*}
	Therefore, 
	$A(r_n) + \varepsilon(r_n, n) \asymp n^{-\xi}$ if and only if $0< \liminf a_n \le \limsup a_n < 1$.

	\item[(iv)] (Logarithmic Decay).
	First, for any sequence $\{r_n\}$, 
	\begin{align}\label{eq:logarithm_0}
	\liminf_{n \rightarrow \infty} 
	\frac{
		\log^{-\xi}(r_n)
		\frac{(n+1)(n-2)}{n(n-r-2)}
	}{
		\log^{-\xi}(n)
	}
	= 
	\liminf_{n \rightarrow \infty} 
	\left(\frac{\log n}{\log r_n}  \right)^{\xi} \cdot \frac{1}{1-2n^{-1} - r_n/n} \ge 1 > 0. 
	\end{align}
	Second, by the equality in \eqref{eq:logarithm_0}, 
	$A(r_n) + \varepsilon(r_n, n) \asymp \log^{-\xi}(n)$ if and only if $\limsup r_n/n < 1$, $\liminf \frac{\log r_n}{\log n} > 0$. 
\end{itemize}

\section{Technical Details for the Tree Models in Section \ref{sec:theory_tree}}\label{app:theory_tree}

\subsection{Expression for estimation error $\varepsilon(r,n)$}\label{sec:expr_estimation_binary}

For any $r\ge 1$ and any $\vec{\bm{x}}_{r} \in \{1, 2, \ldots, M\}^{r+1}$, 
let  $\mu(\bm{x}_r)=\E(Y\mid \vec{\bm{X}}_{r} = \bm{x}_r)$ and $\sigma^2(\vec{\bm{x}}_{r})=\Var(Y\mid \vec{\bm{X}}_{r} = \bm{x}_r)$. Also let $n(\vec{\bm{x}}_{r}) = \sum_{i=1}^{n} \I(\vec{\bm{x}}_{ir} = \vec{\bm{x}}_{r})$ be the number of units in the training set with covariate value $\vec{\bm{x}}_{r}$ up to resolution $r$.  
From 
\eqref{eq:quad}, we have  
\begin{align*}
\varepsilon(r, \trainingset_n) & = 
\E
\left\{
\left[
\hat{\bm{\theta}}_{r}(\vec{\bm{X}}_{r}) - \mu(\vec{\bm{X}}_{r})
\right]^2
\right\} = 
\sum_{\vec{\bm{x}}_{r}}
\pr(\vec{\bm{X}}_{r} = \vec{\bm{x}}_{r}) \cdot 	\left[
\hat{\bm{\theta}}_{r}(\vec{\bm{x}}_{r}) - \mu(\bm{x}_r)
\right]^2.
\end{align*}
By definition, we can simplify 
$\hat{\bm{\theta}}_{r}(\vec{\bm{x}}_{r})$
as 
\begin{align*}
\hat{\bm{\theta}}_{r}(\vec{\bm{x}}_{r}) 
& = 
\I( n(\vec{\bm{x}}_{r}) > 0 ) \hat{\bm{\theta}}_{r}(\vec{\bm{x}}_{r}) 
+ 
\sum_{k=0}^{r-1} \I(n(\vec{\bm{x}}_{k}) > 0, n(\vec{\bm{x}}_{k+1}) = 0) \hat{\bm{\theta}}_{r}(\vec{\bm{x}}_{k}). 
\end{align*}
The conditional expectation of $\left[
\hat{\bm{\theta}}_{r}(\vec{\bm{x}}_{r}) - \mu(\vec{\bm{x}}_r)
\right]^2$ given $n(\bm{a}_{\vec{r}})$ for all $\bm{a}_{\vec{r}} \in \{1, \ldots, M\}^{r+1}$ then has the following equivalent form:
\begin{align*}
& \quad \ 
\E_n\left\{\left[ 
\hat{\bm{\theta}}_{r}(\vec{\bm{x}}_{r}) - \mu(\vec{\bm{x}}_r)
\right]^2
\big|
n(\bm{a}_{\vec{r}}), \forall \bm{a}_{\vec{r}} \in \{1, \ldots, M\}^{r+1}
\right\}\\
& = 
\I( n(\vec{\bm{x}}_{r}) > 0 ) \cdot \E_n\left\{ \left[
\hat{\bm{\theta}}_{r}(\vec{\bm{x}}_{r}) - \mu(\vec{\bm{x}}_{r})
\right]^2
\big|
n(\bm{a}_{\vec{r}}), \forall \bm{a}_{\vec{r}} \in \{1, \ldots, M\}^{r+1}
\right\}
\\
& \quad \ 
+ 
\sum_{k=0}^{r-1} \I(n(\vec{\bm{x}}_{k}) > 0, n(\vec{\bm{x}}_{k+1}) = 0)
\cdot
\E_n\left\{\left[
\hat{\bm{\theta}}_{r}(\vec{\bm{x}}_{k}) - \mu(\vec{\bm{x}}_{r})
\right]^2
\big|
n(\bm{a}_{\vec{r}}), \forall \bm{a}_{\vec{r}} \in \{1, \ldots, M\}^{r+1}
\right\}\\
& = 
\I( n(\vec{\bm{x}}_{r}) > 0 ) \cdot \frac{\sigma^2(\vec{\bm{x}}_{r})}{n(\vec{\bm{x}}_{r})}
+ 
\sum_{k=0}^{r-1} \I(n(\vec{\bm{x}}_{k}) > 0, n(\vec{\bm{x}}_{k+1}) = 0)
\cdot
\left\{\left[ 
\mu(\vec{\bm{x}}_{k}) - \mu(\vec{\bm{x}}_{r})
\right]^2
+ 
\frac{\sigma^2(\vec{\bm{x}}_{k})}{n(\vec{\bm{x}}_{k})}
\right\}\\
& = 
\I( n(\vec{\bm{x}}_{r}) > 0 ) \cdot \frac{\sigma^2(\vec{\bm{x}}_{r})}{n(\vec{\bm{x}}_{r})}
+ 
\sum_{k=0}^{r-1} \I(n(\vec{\bm{x}}_{k}) > 0, n(\vec{\bm{x}}_{k+1}) = 0)
\cdot
\frac{\sigma^2(\vec{\bm{x}}_{k})}{n(\vec{\bm{x}}_{k})} 
\\
& \quad \ 
+ \sum_{k=0}^{r-1} \I(n(\vec{\bm{x}}_{k}) > 0, n(\vec{\bm{x}}_{k+1}) = 0)
\cdot
\left[ 
\mu(\vec{\bm{x}}_{k}) - \mu(\vec{\bm{x}}_{r})
\right]^2. 
\end{align*}
By the law of iterated expectation, this immediately implies that 
\begin{align*}
& \quad \ \E_n\left[
\hat{\bm{\theta}}_{r}(\vec{\bm{x}}_{r}) - \mu(\vec{\bm{x}}_r)
\right]^2
\\
& = 
\E_n\left\{
\E_n\left[ \left[
\hat{\bm{\theta}}_{r}(\vec{\bm{x}}_{r}) - \mu(\vec{\bm{x}}_r)
\right]^2
\bigg|
n(\bm{a}_{\vec{r}}), \forall \bm{a}_{\vec{r}} \in \{1, \ldots, M\}^{r+1}
\right]
\right\}
\\
& = 
\sigma^2(\vec{\bm{x}}_{r})  \cdot 
\E_n\left[ \frac{\I( n(\vec{\bm{x}}_{r}) > 0 )}{n(\vec{\bm{x}}_{r})}
\right]
+ 
\sum_{k=0}^{r-1} 
\sigma^2(\vec{\bm{x}}_{k})
\cdot
\E_n\left[
\frac{\I(n(\vec{\bm{x}}_{k}) > 0, n(\vec{\bm{x}}_{k+1}) = 0)}{n(\vec{\bm{x}}_{k})} \right]
\\
& \quad \ 
+ 
\sum_{k=0}^{r-1} 
\left[
\mu(\vec{\bm{x}}_{k}) - \mu(\vec{\bm{x}}_{r})
\right]^2
\cdot 
\E_n\left[ 
\I(n(\vec{\bm{x}}_{k}) > 0, n(\vec{\bm{x}}_{k+1}) = 0)
\right]\\
& = 
\sigma^2(\vec{\bm{x}}_{r})  \cdot 
\E_n\left[ \frac{\I( n(\vec{\bm{1}}_{r}) > 0 )}{n(\vec{\bm{1}}_{r})}
\right]
+ 
\sum_{k=0}^{r-1} 
\sigma^2(\vec{\bm{x}}_{k})
\cdot
\E_n\left[
\frac{\I(n(\vec{\bm{1}}_{k}) > 0, n(\vec{\bm{1}}_{k+1}) = 0)}{n(\vec{\bm{1}}_{k})} \right]
\\
& \quad \ 
+ 
\sum_{k=0}^{r-1}
\left[ 
\mu (\vec{\bm{x}}_{k}) - \mu(\vec{\bm{x}}_{r})
\right]^2
\cdot 
\E_n\left[ 
\I(n(\vec{\bm{1}}_{k}) > 0, n(\vec{\bm{1}}_{k+1}) = 0)
\right],
\end{align*}
where the last equality holds because of the symmetry rendered  by our assumption that the covariates are i.i.d. 
with the same probability taking values $1, 2, \ldots, M$. 

It follows then that
\begin{align*}
\varepsilon(r,n) = 
\E_n\left[\varepsilon(r, \trainingset_n) \right]
& = 
\E\left[\sigma^2(\vec{\bm{X}}_{r}) \right]  \cdot 
\E_n\left[ \frac{\I( n(\vec{\bm{1}}_{r}) > 0 )}{n(\vec{\bm{1}}_{r})}
\right]
\\
& \quad \ + 
\sum_{k=0}^{r-1} 
\E\left[\sigma^2(\vec{\bm{X}}_{k}) \right]
\cdot
\E_n\left[
\frac{\I(n(\vec{\bm{1}}_{k}) > 0, n(\vec{\bm{1}}_{k+1}) = 0)}{n(\vec{\bm{1}}_{k})} \right]
\\
& \quad \ 
+ 
\sum_{k=0}^{r-1} 
\E\left[ 
\mu (\vec{\bm{X}}_{k}) - \mu(\vec{\bm{X}}_{r})
\right]^2
\cdot 
\E_n\left[ 
\I(n(\vec{\bm{1}}_{k}) > 0, n(\vec{\bm{1}}_{k+1}) = 0)
\right]. 
\end{align*}
Noting that 
\begin{align*}
& \quad \ \E\left\{ \left[ 
\mu (\vec{\bm{X}}_{k}) - \mu(\vec{\bm{X}}_{r})
\right]^2
\right\} = 
\E\left\{
\Var\left[ 
\mu(\vec{\bm{X}}_{r})
\mid \vec{\bm{X}}_{k}
\right]
\right\}
= 
\E\left\{
\Var\left[ 
\E(Y\mid \vec{\bm{X}}_{r})
\mid \vec{\bm{X}}_{k}
\right]
\right\}\\
& = 
\E\left\{
\Var(Y\big| \vec{\bm{X}}_{k})
- 
\E\left[ 
\Var(Y\mid \vec{\bm{X}}_{r})
\big| \vec{\bm{X}}_{k}
\right]
\right\}
= 
\E\left[
\sigma^2(\vec{\bm{X}}_{k})
\right]
-
\E\left[
\sigma^2(\vec{\bm{X}}_{r})
\right],
\end{align*}
and that $\E\left[
\sigma^2(\vec{\bm{X}}_{k})
\right] =  A(k) + \tau^2$ for any $k \ge 0$,  
we  obtain  expression \eqref{eq:epsilon_r_n_binary}.

\subsection{Some technical lemmas}

\begin{lemma}\label{lemma:recip_binomial}
	For the $n(\vec{\bm{x}}_{k})$'s defined in Section \ref{sec:regression_tree}, we have, for any $k\ge 0$, 
	\begin{align*}
	\E_n\left[
	\frac{1}{n(\vec{\bm{1}}_{k})+1}
	\right] - \pr\left\{ n(\vec{\bm{1}}_{k}) = 0 \right\}
	\le 
	\E_n\left[ 
	\frac{\I( n(\vec{\bm{1}}_{k}) > 0 )}{n(\vec{\bm{1}}_{k})}
	\right]
	\le 
	2 \cdot
	\E_n\left[
	\frac{1}{n(\vec{\bm{1}}_{k})+1}
	\right],
	\end{align*}
	and 
	\begin{align*}
	\E_n\left[
	\frac{1}{n(\vec{\bm{1}}_{k})+1}
	\right]
	& = \frac{M^{k}}{n+1}
	\left\{
	1 - (1-M^{-k})^{n+1}
	\right\}.
	\end{align*}
\end{lemma}
\begin{proof}[Proof:] The first part (i.e., inequality) follows from the fact that for any integer $z\ge 0$, 
	\begin{align*}
	\frac{\I(z>0)}{z} & \ge \frac{\I(z>0)}{z+1} = \frac{1-\I(z=0)}{z+1} 
    = 
	\frac{1}{z+1} - \I(z=0),
	\end{align*}
	and 
	\begin{align*}
	\frac{\I(z>0)}{z} = \frac{(z+1)\I(z>0)}{z(z+1)} \le \frac{2z}{z(z+1)} = \frac{2}{z+1}. 
	\end{align*}
	The second part (i.e., equality) follows from the fact that 
	$n(\vec{\bm{1}}_{k}) \sim \text{Binomial}(n, M^{-k})$.  
\end{proof}

\begin{lemma}\label{lemma:one_over_nM}
	For the  $n(\vec{\bm{x}}_{k})$'s defined in Section \ref{sec:regression_tree},
	recall that 
	$$
	\mathcal{K} = \max\left\{
	k: 
	n(\vec{\bm{1}}_{k})  > 0, n(\vec{\bm{1}}_{k+1}) = 0
	\right\},
	$$
	and 
	$\mathcal{K}\wedge r = \min\{\mathcal{K}, r\}$ 
	for any $r\ge 0$. 
	Then 
	\begin{itemize}
		\item[(a)] for any $r\ge 0$, 
		\begin{align*}
		\E_n\left[ \frac{1}{n(\vec{\bm{1}}_{\mathcal{K} \wedge r})} \right]
		& = 
		\E_n\left[ \frac{\I( n(\vec{\bm{1}}_{r}) > 0 )}{n(\vec{\bm{1}}_{r})}
		\right]
		+ 
		\sum_{k=0}^{r-1} 
		\E_n\left[
		\frac{\I(n(\vec{\bm{1}}_{k}) > 0, n(\vec{\bm{1}}_{k+1}) = 0)}{n(\vec{\bm{1}}_{k})} \right]. 
		\end{align*}
		\item[(b)] 
		$
		\E_n\left[ \frac{1}{n(\vec{\bm{1}}_{\mathcal{K} \wedge r})} \right]
		$
		is a non-decreasing function of  $r$. 
	\end{itemize}
\end{lemma}
\begin{proof}[Proof:]
	Part (a) in Lemma \ref{lemma:one_over_nM} follows immediately from its definition. 
	Part (b) follows from the simple fact that 
	$(\mathcal{K} \wedge r)$ is non-decreasing
	in
	$r$, and hence 
	$n(\vec{\bm{1}}_{\mathcal{K} \wedge r})$ is non-increasing in $r$, which immediately implies (b). 
\end{proof}

\begin{lemma}\label{lemma:br}
	Let $b(r) = (1 - M^{-r})^{M^r}$ be a function of $r$. 
	Then $b(r)$ is increasing in $r$, and 
	$b(r) \le \lim_{k\rightarrow \infty} b(k) = e^{-1}$. 
\end{lemma}

\begin{proof}[Proof:]
	This result follows from the well-known fact that 
	the function $(1-x^{-1})^x$ monotonically  increases with $x \in (1, \infty)$, approaching a limit $e^{-1}$ as $x\rightarrow \infty$. 
\end{proof}

\begin{lemma}\label{lemma:np}
	For any $n \ge 1$ and $p \in [0,1]$, $n p (1-p)^n \le 1$. 
\end{lemma}
\begin{proof}[Proof:]
	Let $Z\!\sim\! \text{Binomial}(n+1, p)$. Then $np(1-p)^n\!\le\!\!(n+1) p (1-p)^n\! =\! \pr(Z=1)\! \le\! 1$.
\end{proof}

\begin{lemma}\label{lemma:nk_inverse_k_kplus1}
    For the $n(\vec{\bm{x}}_{k})$'s defined in Section \ref{sec:regression_tree},
    we have, for any $n\ge 1$ and $k\ge 0$, 
    \begin{align*}
    & \quad \ \E_n\left[
    \frac{\I(n(\vec{\bm{1}}_{k}) > 0, n(\vec{\bm{1}}_{k+1}) = 0)}{n(\vec{\bm{1}}_{k})} \right]
    \\
    & \ge 
    \frac{M^k}{(n+1)(1-M^{-1})} 
    \left[
    \left( 1 - M^{-k} \right)^{\frac{n+1}{M}} - (1-M^{-k})^{n+1}
    \right]
    - 
    (1-M^{-k})^n. 
\end{align*}
\end{lemma}

\begin{proof}[Proof:] For notational simplicity,
let $Z_k=n(\vec{\bm{1}}_{k})$, $p = M^{-k}$, and
$q=1-p.$ Then $Z_k\sim \text{Binomial}(n,p)$, and 
    $Z_{k+1}\mid Z_k \sim \text{Binomial}(Z_k, M^{-1})$. 
    By the law of iterated expectation, 
    \begin{align*}
    \E_n\left[
    \frac{\I(Z_{k} > 0, Z_{k+1} = 0)}{Z_k} \right]
    & = 
    \E_n\left[
    \frac{\I(Z_{k} > 0) \cdot \E_n\left\{
    \I( Z_{k+1} = 0) \mid Z_{k}
    \right\}
    }{Z_{k}}
    \right]\\
    & = 
    \E_n\left[
    \frac{\I(Z_{k} > 0) \cdot (1-M^{-1})^{Z_{k}}
    }{Z_{k}}
    \right]\\
    & \ge
    \E\left[\frac{(1-M^{-1})^{Z_k}}{Z_k+1} - \I(Z_k=0)\right]\\
    & = 
    \sum_{k=0}^n \binom{n}{k} p^k q^{n-k} \left( 1- M^{-1} \right)^k \frac{1}{k+1} 
    - 
    q^n\\
 \left[{\rm let}\ c_{n, M}=\frac{1}{(n+1)(1-M^{-1})} \right]\quad    & = 
    c_{n,M}p^{-1} \sum_{m=1}^{n+1} \binom{n+1}{m}  \left[p \left(1-M^{-1} \right) \right]^{m} q^{n+1-m} 
    - 
    q^n\\
    & =  c_{n,M}p^{-1} 
    \left[
    \left( 1 - \frac{p}{M} \right)^{n+1} - q^{n+1}
    \right]
    - 
    q^n.\\
 \left[{\rm note}\ \left(1-\frac{p}{M}\right)^M \ge 1-p=q\right]\quad  & \ge  
  c_{n,M}p^{-1} 
    \left( 
    q^{\frac{n+1}{M}} - q^{n+1}
    \right)
    - 
    q^n,
\end{align*}
which establishes the result because $p=M^{-k}$,
and $q=1-M^{-k}.$
\end{proof}

\begin{lemma}\label{lemma:lowerbound_large_r}
Let $\mathcal{S}$ be a set of infinitely many positive integers. Then for any positive integer sequence $\{r_{n},\  n\in \mathcal{S}\}$ satisfying  $M^{r_{n}}/n > c$ for all $n\in \mathcal{S}$ and some constant $c>0$, there must exist a countably infinite subset $S \subset \mathcal{S}$ 
    and a corresponding positive integer sequence $\{\tilde{r}_n: n \in S\}$
    such that 
    (i) $\tilde{r}_{n} \le r_{n}$ for all $n\in S$, 
    (ii)
    $M^{\tilde{r}_{n}}/n$ has a positive limit $\tilde c\in [c, Mc]$ as $n\in S$ goes infinity, 
    and 
    (iii) 
    $
    \liminf_{n \in S, n\rightarrow \infty} \varepsilon(r_{n}, n) / A(\tilde{r}_{n}) > 0. 
    $
\end{lemma}

\begin{proof}[Proof:]  Let $\tilde r_n=\min\{r_n, \lceil\log_M(nc)\rceil\}$, where $\lceil{x}\rceil$ is the smallest integer that is not below $x$. Clearly (i) follows immediately.  For (ii), we note that 
the fact $r_n>\log_M(cn)$ implies
$\tilde r_n \ge \log_M(cn)$. On the other hand,  $\tilde r_n\le \lceil\log_M(nc)\rceil
\le \log_M(nc) + 1$.  Therefore the sequence 
$M^{\tilde{r}_{n}}/{n} \in [c, Mc]$ for all $n\in \mathcal{S}$. 
By the Bolzano--Weierstrass theorem, 
there must exist a subsequence $\{\tilde{r}_{n},\ n\in S\subset \mathcal{S}\}$ such that $M^{\tilde{r}_{n}}/n$ converges to some constant $\tilde{c} \in [c, Mc]$ as $n\rightarrow \infty$ but with $n\in S$. 

For part (iii), we note that 
from \eqref{eq:epsilon_r_n_binary}, 
    for any $n \in S$, 
    we can bound $\varepsilon(r_{n},n)$ by 
    \begin{align*}
    \varepsilon(r_{n},n)
    & \ge 
    \sum_{k=0}^{r_{n}} A(k) \cdot \E_{n}\left[
    \frac{\I(n(\vec{\bm{1}}_{k}) > 0, n(\vec{\bm{1}}_{k+1}) = 0)}{n(\vec{\bm{1}}_{k})} \right]\\
    & \ge 
    A(\tilde{r}_{n}) \cdot \E_n\left[
    \frac{\I(n(\vec{\bm{1}}_{\tilde{r}_{n}}) > 0, n(\vec{\bm{1}}_{\tilde{r}_{n}+1}) = 0)}{n(\vec{\bm{1}}_{\tilde{r}_{n}})} \right],
    \end{align*}
    where the last inequality holds because $\tilde{r}_{n} \le r_{n}$ by construction.  Consequently,  by Lemma \ref{lemma:nk_inverse_k_kplus1},  
    for any $n \in S$, we have  
    \begin{align}\label{eq:bound_large_r}
        \frac{\varepsilon(r_{n},n)}{A(\tilde{r}_{n})}
        & \ge 
    \frac{M^{\tilde{r}_{n}}}{(n+1)(1-M^{-1})} 
    \left[ 
    \left( 1 - M^{-\tilde{r}_{n}} \right)^{(n+1)/M} - (1-M^{-\tilde{r}_{n}})^{n+1}
    \right]
    - 
    (1-M^{-\tilde{r}_{n}})^{n}
    \nonumber\\
        & =
        \frac{M^{\tilde{r}_{n}}}{(n+1)(1-M^{-1})} 
        \left[
       b(\tilde{r}_{n})^{(n+1)/( M^{\tilde{r}_{n}}\cdot M)} - b(\tilde{r}_{n})^{ (n+1)/M^{\tilde{r}_{n}}}
        \right]
        - 
        b(\tilde{r}_{n})^{n/M^{\tilde{r}_{n}}},
    \end{align}
    recalling that $b(r) = (1-M^{-r})^{M^r}$. 
    By construction, 
    as $n\in S$ goes to infinity, 
    $M^{\tilde{r}_{n}}/(n+a) \rightarrow \tilde{c}$ for
    any fixed $a\ge 0$
    and
    $b(\tilde{r}_{n}) \rightarrow e^{-1}$. 
    Consequently, 
    \begin{align*}
     \liminf_{n \in S, n \rightarrow \infty} 
        \frac{\varepsilon(r_{n},n)}{A(\tilde{r}_{n})}
        & 
        \ge\frac{1}{1-M^{-1}} \tilde{c} 
        \left(
        e^{-1/(M\tilde{c})} - e^{-1/\tilde{c}}
        \right)
        - e^{-1/\tilde{c}}\\
        & = 
        \frac{1}{1-M^{-1}} \tilde{c} e^{-1/\tilde{c}}
        \left[
        e^{(1-M^{-1})/\tilde{c}} - 1 - \frac{1-M^{-1}}{\tilde{c}}
        \right] > 0, 
    \end{align*}
which concludes the proof. \end{proof}

\subsection{Asymptotic rate of $\varepsilon(r, n)$ when $\tau^2 > 0$}

First, we show that $M^r/n = o(1)$ is necessary for  $\varepsilon(r,n)\rightarrow 0$  as $n\rightarrow \infty$. We prove this by contradiction. 
Suppose that there exists a sequence $r_n$ such that $\varepsilon(r_n, n) = o(1)$, but $M^{r_n}/n\not = o(1)$. 
Then there must exist a subsequence $\{r_{n_i}\}$ such that $M^{r_{n_i}}/{n_i} > c$ for some constant $c>0$. 
 This allows us to invoke the subsequence $\{\tilde{r}_{n}: n \in {S}\}$ as defined by Lemma~\ref{lemma:lowerbound_large_r}. Consequently, by 
\eqref{eq:epsilon_r_n_binary} and Lemma \ref{lemma:one_over_nM}, 
for any $n\in S$, we have
\begin{align*}
\varepsilon(r_{n},n) \ge 
\tau^2 \cdot 
\E_{n}\left[ 
\frac{1}{n(\vec{\bm{1}}_{\mathcal{K} \wedge r_{n}})}
\right]
\ge 
\tau^2 \cdot
\E_{n}\left[ 
\frac{1}{n(\vec{\bm{1}}_{\mathcal{K}\wedge \tilde{r}_{n}})}
\right]\ge 
\tau^2  \cdot 
\E_{n}\left[ \frac{\I( n(\vec{\bm{1}}_{\tilde{r}_{n}}) > 0 )}{ n(\vec{\bm{1}}_{\tilde{r}_{n}}) }
\right],
\end{align*}
where the second inequality comes from the fact that $\tilde r_n\le r_n$. By Lemma \ref{lemma:recip_binomial}, this further implies that for $n\in {S}$, 
\begin{align}\label{eq:lower_bound_epsilon}
\frac{\varepsilon(r_{n},n)}{\tau^2}
& \ge 
\E_{n}\left[ \frac{\I( n(\vec{\bm{1}}_{\tilde{r}_{n}}) > 0 )}{ n(\vec{\bm{1}}_{\tilde{r}_{n}}) }
\right]
\ge 
\E_{n}\left[
\frac{1}{n(\vec{\bm{1}}_{\tilde{r}_{n}})+1}
\right] - \pr\left\{  n(\vec{\bm{1}}_{\tilde{r}_{n}}) = 0 \right\}
\nonumber
\\
& = 
\frac{M^{\tilde{r}_{n}}}{n+1}
	\left\{
	1 - (1-M^{-\tilde{r}_{n}})^{n+1}
	\right\}
	- \left( 1- M^{-\tilde{r}_{n}} \right)^{n}.
\end{align}
Since the sequence $\{\tilde{r}_{n}: n \in {S}\}$ here has the same properties as the one in 
\eqref{eq:bound_large_r}, 
the same limit calculation there leads to 
\begin{align*}
\frac{1}{\tau^2} \liminf_{n \in \mathcal{S}, n \rightarrow \infty} \varepsilon(r_{n},n)
& \ge 
\tilde{c} 
\left(
1 - e^{-1/\tilde{c}}
\right)
- 
e^{-1/\tilde{c}}
= 
\tilde{c} e^{-1/\tilde{c}}
\left[
e^{1/\tilde{c}} - \left( 1 + \frac{1}{\tilde{c}} \right)
\right]
> 0. 
\end{align*}
However, this contradicts the fact that $\varepsilon(r_n,n) = o(1)$. 

We now prove $\varepsilon(r_n,n) = M^r/n$ when $M^r/n = o(1)$. 
Following the same reasoning as for \eqref{eq:lower_bound_epsilon}, we have 
\begin{align}
\frac{\varepsilon(r,n)}{\tau^2}
& \ge
\E_n\left[ \frac{\I( n(\vec{\bm{1}}_{r}) > 0 )}{n(\vec{\bm{1}}_{r})}
\right]
\ge 
\frac{M^{r}}{n+1}
\left[
1 - (1- M^{-r} )^{ M^{r} \cdot (n+1)/M^{r}}
\right]
- 
(1-M^{-r})^{M^{r} \cdot n/ M^{r}}. 
\end{align}
By Lemma \ref{lemma:br}, this further implies that 
\begin{align*}
\frac{\varepsilon(r,n)}{\tau^2}
& \ge
\frac{M^{r}}{n+1}
\left[
1 - e^{-(n+1)/M^{r}}
\right]
- 
e^{-n/ M^{r}} \\
& = 
\frac{M^r}{n} \cdot \frac{n}{n+1} 
\left[
1 - e^{-(n+1)/M^{r}} - \frac{n+1}{n}\frac{n}{M^r} e^{-n/M^r}
\right]. 
\end{align*}
Note that when $M^r/n = o(1)$, we must have $e^{-(n+1)/M^{r}} = o(1)$ and $n/M^r \cdot e^{-n/M^r} = o(1)$. 
Thus, for sufficiently large $n$, we deduce that
$$
\varepsilon(r,n) \ge \tau^2 \cdot \frac{1}{2} \cdot \frac{M^r}{n} = \frac{\tau^2}{2} \cdot \frac{M^r}{n}. 
$$

From \eqref{eq:epsilon_r_n_binary}, we also have 
\begin{align*}
\varepsilon(r,n)
& \le 
\left[ A(0) +  \tau^2 \right]  \cdot 
\E_n\left[ \frac{\I( n(\vec{\bm{1}}_{r}) > 0 )}{n(\vec{\bm{1}}_{r})}
\right]
\\
& \quad \ + 
\left[ A(0) +  \tau^2 \right]
\cdot \sum_{k=0}^{r-1} 
\E_n\left[
\I(n(\vec{\bm{1}}_{k}) > 0, n(\vec{\bm{1}}_{k+1}) = 0) \right]
\\
& \quad \ 
+ 
A(0) \cdot \sum_{k=0}^{r-1} 
\E_n\left[ 
\I(n(\vec{\bm{1}}_{k}) > 0, n(\vec{\bm{1}}_{k+1}) = 0)
\right]\\
& = 
\left[ A(0) +  \tau^2 \right]  \cdot 
\E_n\left[ \frac{\I( n(\vec{\bm{1}}_{r}) > 0 )}{n(\vec{\bm{1}}_{r})}
\right]
+ 
\left[ 2 A(0) +  \tau^2 \right] \cdot \pr\left[ n(\vec{\bm{1}}_{r}) = 0 \right],
\end{align*}
where the last equality holds because 
$\I(n(\vec{\bm{1}}_{k}) > 0, n(\vec{\bm{1}}_{k+1}) = 0) = \I(n(\vec{\bm{1}}_{k+1}) = 0) - \I(n(\vec{\bm{1}}_{k}) = 0)$ and $n(\vec{\bm{1}}_{0}) \equiv n > 0$. 
Following Lemmas \ref{lemma:recip_binomial} and \ref{lemma:br}, we deduce that
\begin{align*}
\varepsilon(r,n)
& \le 
2\left[ A(0) +  \tau^2 \right]  \cdot 
\frac{M^r}{n+1}
+ 
\left[2 A(0) +  \tau^2 \right]  \cdot 
(1-M^{-r})^{M^r \cdot n/M^r}\\
& \le 
2\left[ A(0) +  \tau^2 \right]  \cdot
\left(
\frac{M^r}{n} + e^{-n/M^r}
\right) = O\left( \frac{M^r}{n} \right),
\end{align*}
where the last equality holds because $M^r/n = o(1)$. This completes our proof.

\subsection{Asymptotic rate of $\varepsilon(r, n)$ when $\tau^2 = 0$}

Because it is more challenging to study lower bounds for the estimation error in this case, 
below we consider only upper bounds, which give sufficient conditions for the estimation error to achieve a certain rate. 

First, we give an upper bound for the estimation error. 
From \eqref{eq:epsilon_r_n_binary}  , we can bound the estimation error $\varepsilon(r, n)$ from above by 
\begin{align*}
\varepsilon(r,n)
& \le 
A(r) \cdot 
\E_n\left[ \frac{\I( n(\vec{\bm{1}}_{r}) > 0 )}{n(\vec{\bm{1}}_{r})}
\right]
+ 
\sum_{k=0}^{r-1} 
A(k)
\cdot
\E_n\left[
\I(n(\vec{\bm{1}}_{k}) > 0, n(\vec{\bm{1}}_{k+1}) = 0)
\right]
\\
& \quad \ 
+ 
\sum_{k=0}^{r-1} 
\left[
A(k) - A(r)
\right]
\cdot 
\E_n\left[ 
\I(n(\vec{\bm{1}}_{k}) > 0, n(\vec{\bm{1}}_{k+1}) = 0)
\right]\\
& \le 
A(r) \cdot 
\E_n\left[ \frac{\I( n(\vec{\bm{1}}_{r}) > 0 )}{n(\vec{\bm{1}}_{r})}
\right] 
+ 
2 \sum_{k=0}^{r-1} 
A(k)
\cdot 
\E_n\left[ 
\I(n(\vec{\bm{1}}_{k}) > 0, n(\vec{\bm{1}}_{k+1}) = 0)
\right]\\
& \le 
A(r) \cdot 
\E_n\left[ \frac{\I( n(\vec{\bm{1}}_{r}) > 0 )}{n(\vec{\bm{1}}_{r})}
\right] 
+ 
2 \sum_{k=0}^{r-1} 
A(k)
\cdot 
\pr \left[
n(\vec{\bm{1}}_{k+1}) = 0
\right]. 
\end{align*}
Note that from Lemmas \ref{lemma:recip_binomial} and \ref{lemma:np}, for any $k\ge 0$, 
\begin{align*}
\E_n\left[ 
\frac{\I( n(\vec{\bm{1}}_{k}) > 0 )}{n(\vec{\bm{1}}_{k})}
\right]
& \le 
2 \cdot
\E_n\left[
\frac{1}{n(\vec{\bm{1}}_{k})+1}
\right]
\le 
2
\frac{M^{k}}{n+1}
\le 
2
\frac{M^{k}}{n}, 
\end{align*}
and 
\begin{align*}
\pr \left[
n(\vec{\bm{1}}_{k+1}) = 0
\right] 
& = \left( 1 - M^{-(k+1)} \right)^n \le \frac{M^{k+1}}{n} = M \frac{M^k}{n}.
\end{align*}
Moreover, both terms $\E_n\left[ 
\I( n(\vec{\bm{1}}_{k}) > 0 )/n(\vec{\bm{1}}_{k})
\right]$
and 
$\pr \left[
n(\vec{\bm{1}}_{k+1}) = 0
\right]$ are less than or equal to 1. 
Thus, 
we can further bound the estimation error by 
\begin{align}\label{eq:upper_binary_zero}
\varepsilon(r,n)
& \le 
A(r) \cdot \min\left[ 2\cdot \frac{M^r}{n}, 1\right] + 
2 \sum_{k=0}^{r-1} A(k) \min \left\{ M\cdot \frac{M^{k}}{n}, 1 \right\}
\nonumber
\\
& \le 
2M \sum_{k=0}^r A(k) \cdot \min\left\{ \frac{M^k}{n}, 1 \right\} 
\le \frac{2M}{n} \sum_{k=0}^{r} M^k A(k). 
\end{align}

Second, we study the asymptotic rate of the upper bound \eqref{eq:upper_binary_zero} under varying decay rates for the resolution bias $A(r)$. 
We initially consider the exponential decay case, i.e., $A(r) \asymp e^{-\xi r}$. 
Note that  
\begin{align*}
\sum_{k=0}^{r}  M^k A(k)
& \asymp
\sum_{k=0}^{r}  M^k e^{-\xi k} 
= \sum_{k=0}^{r}  ( M e^{-\xi} )^k
= 
\begin{cases}
\frac{ 1 - ( M e^{-\xi} )^{r+1} }{
	1 - M e^{-\xi}
}, & \text{ if } M e^{-\xi} < 1,\\
r+1, & \text{ if } M e^{-\xi} = 1,\\
\frac{ ( M e^{-\xi} )^{r+1} - 1 }{
	M e^{-\xi} - 1
}, &  \text{ if } M e^{-\xi} > 1,
\end{cases}
\\
& \asymp
\begin{cases}
1, & \text{ if } M e^{-\xi} < 1,\\
r, & \text{ if } M e^{-\xi} = 1, \\
M^r A(r), & \text{ if } M e^{-\xi} > 1. 
\end{cases}
\end{align*}
From \eqref{eq:upper_binary_zero}, we know that 
\begin{align*}
\varepsilon(r,n)
& 
\le \frac{2M}{n} \sum_{k=0}^{r} M^k A(k)
\asymp 
\begin{cases}
n^{-1}, & \text{ if } M e^{-\xi} < 1,\\
r/n, & \text{ if } M e^{-\xi} = 1, \\
M^r A(r)/n, & \text{ if } M e^{-\xi} > 1. 
\end{cases}
\end{align*}
We now consider the polynomial decay case, i.e., $A(r) \asymp r^{-\xi}$. 
Note that as $k \rightarrow \infty$, 
$M^{k+1} (k+1)^{-\xi} / \{M^k k^{-\xi}\} \rightarrow M \ge 2 > 1$. 
Thus, 
there exists $K_0\ge 0$ such that when $k \ge K_0,$
$A(k)\le c k^{-\xi}$, for some constant $c$, and 
$
M^{k+1} (k+1)^{-\xi} / \{M^k k^{-\xi} \} \ge 3/2.
$
This implies that, for $r>K_0$,
\begin{align*}
M^rA(r)\le 
& \sum_{k=0}^{r} M^k A(k) \le 
\sum_{k=0}^{K_0-1} M^k A(k) + M^r r^{-\xi}c \sum_{k=K_0}^{r} \left( \frac{2}{3} \right)^{r-k}\\
& \le \sum_{k=0}^{K_0-1}M^k A(k) + M^r r^{-\xi} \cdot 3c = O\left(  M^r r^{-\xi}  \right) = O\left(  M^rA(r)   \right).
\end{align*}
Therefore, 
\begin{equation}\label{eq:epsilon-bound}
\varepsilon(r,n)
\le \frac{2M}{n} \sum_{k=0}^{r} M^k A(k)
\asymp 
\frac{1}{n} M^r A(r). 
\end{equation}
A similar argument shows that \eqref{eq:epsilon-bound} also holds when $A(r)$ has a logarithmic decay. 

Third, we consider the hard thresholding case, i.e., $A(r) = 0$ for $r \ge r_0$ and $A(r)> 0$ for $r< r_0$. 
When $r\ge r_0$, from \eqref{eq:epsilon_r_n_binary}, we have
\begin{align*}
\varepsilon(r,n)
& = 
\sum_{k=0}^{r_0-1} 
A(k) \cdot 
\left\{
\E_n\left[
\frac{\I(n(\vec{\bm{1}}_{k}) > 0, n(\vec{\bm{1}}_{k+1}) = 0)}{n(\vec{\bm{1}}_{k})} \right]
+ 
\E_n\left[ 
\I(n(\vec{\bm{1}}_{k}) > 0, n(\vec{\bm{1}}_{k+1}) = 0)
\right]
\right\}\\
& \le 
 2 \sum_{k=0}^{r_0-1} A(k) \cdot \E_n\left[ 
    \I(n(\vec{\bm{1}}_{k}) > 0, n(\vec{\bm{1}}_{k+1}) = 0)
    \right]\\
    & 
    \le 
    2 A(0) \sum_{k=0}^{r_0-1} \E_n\left[ 
    \I(n(\vec{\bm{1}}_{k}) > 0, n(\vec{\bm{1}}_{k+1}) = 0)
    \right] \\ 
   & = 2 A(0) \cdot 
    \pr\left\{ n(\vec{\bm{1}}_{r_0}) = 0 \right\} 
     = 
    2 A(0) \cdot
    \left( 1 - M^{-r_0} \right)^n, 
\end{align*}
where the last equality holds because $n(\vec{\bm{1}}_{r_0}) \sim \text{Binomial}(n, M^{-r_0})$. 
On the other hand,
\begin{align*}
    \varepsilon(r,n) 
    & \ge 
    \sum_{k=0}^{r_0-1} A(k) \cdot \E_n\left[ 
    \I(n(\vec{\bm{1}}_{k}) > 0, n(\vec{\bm{1}}_{k+1}) = 0)
    \right] \\
    & 
    \ge 
    A(r_0-1) \cdot
    \sum_{k=0}^{r_0-1} \E_n\left[ 
    \I(n(\vec{\bm{1}}_{k}) > 0, n(\vec{\bm{1}}_{k+1}) = 0)
    \right] \\
    & = 
    A(r_0-1) \cdot \left( 1 - M^{-r_0} \right)^n. 
\end{align*}
Consequently, when $\liminf_{n\rightarrow \infty} r_n \ge r_0$, 
\begin{align}\label{eq:asymp_rate_binary_zero_tau2_hard_threshold}
    \varepsilon(r_n, n) \asymp \left( 1 - M^{-r_0} \right)^n. 
\end{align}

\section{Proof of Theorem~\ref{th:disc}}
The proof strategy is the same as that for proving Theorem~\ref{th:cont}, except the estimation error instead satisfies $\varepsilon(r,n) \asymp \alpha^r/n$ for some $\alpha>1$. 
\begin{itemize}
	\item[(i)] (Hard Thresholding). 
	First, for any sequence  $\{r_n\}$, 
	\begin{align*}
	\liminf_{n\rightarrow \infty}
	\frac{A(r_n) + \alpha^{r_n}/n}{1/n}>0.  
	\end{align*}
	If this inequality does not hold, then there exists a subsequence $\{r_{n_k}\}$ such that
	\begin{align}\label{eq:proof_exp_opt_hard}
	\lim_{k\rightarrow \infty}
	\frac{A(r_{n_k})}{1/n_k}=0 \ 
	\text{ and }  \ 
	\lim_{k\rightarrow \infty}
	\frac{ \alpha^{r_{n_k}}/n_k}{1/n_k}= \lim_{k\rightarrow \infty}
	\alpha^{r_{n_k}} =0. 
	\end{align}
	In \eqref{eq:proof_exp_opt_hard}, the left-hand formula implies that $\liminf_{k\rightarrow \infty}r_{n_k}\geq r_0,$ 
	which contradicts the right-hand formula.
	Second, we can verify that 
	that $A(r_n) + \varepsilon(r_n,n) \asymp 1/n$ if and only if 
	$r_n = O(1)$ and $\liminf_{n\rightarrow \infty} r_n \geq r_0$. 
	
	\item[(ii)] (Exponential Decay).  
	First,  for any sequence $\left\{r_n=\frac{\log(n)+\log(a_n)}{\log(\alpha)+\xi}\right\}$,
	\begin{align}\label{eq:proof_exp_opt_expo_goal}
	\liminf_{n\rightarrow \infty}
	\frac{e^{-\xi r_n} + \alpha^{r_n}/n}{n^{-\xi/(\log(\alpha)+\xi)}}
	=
	\liminf_{n\rightarrow \infty} 
	\left\{
	a_n^{-\xi/\{\log(\alpha)+\xi\}} + a_n^{\log(\alpha)/\{\log(\alpha)+\xi\}}
	\right\}
	>0.
	\end{align} 
	If this inequality is false,  then there must exist a subsequence $\{a_{n_k}\}$ such that
	$
	\lim_{k\rightarrow \infty} a_{n_k}^{-1} = 0
	$
	and 
	$
	\lim_{k\rightarrow \infty} a_{n_k} = 0,
	$
	a clear contradiction.

	Second, by the equality in  \eqref{eq:proof_exp_opt_expo_goal}, $A(r_n) + \varepsilon(r_n,n) \asymp n^{-\xi/(\log(\alpha)+\xi)}$ if and only if $\{r_n=\{\log(n)+\log(a_n)\}/\{\log(\alpha)+\xi\}\}$ satisfies 
	$a_n\asymp 1$.

	\item[(iii)] (Polynomial Decay). 
	First, for any sequence $\{r_n = a_n \log(n) \}$, 
	\begin{align}\label{eq:proof_exp_opt_poly_goal}
	\liminf_{n\rightarrow \infty}
	\frac{r_n^{-\xi} + \alpha^{r_n}/n}{\log^{-\xi}(n)}
	=
	\liminf_{n\rightarrow \infty} 
	\left(
	a_n^{-\xi} + \frac{\log^\xi(n)}{n^{1-a_n\log(\alpha)}}
	\right)
	>0.  
	\end{align}
	If \eqref{eq:proof_exp_opt_poly_goal} does not hold, then there must exist a subsequence $\{a_{n_k}\}$ such that
	\begin{align*}
	\lim_{k\rightarrow\infty}a_{n_k}^{-1} = 0 \ 
	\text{ and } \ 
	\lim_{k\rightarrow \infty} 
	\frac{\log^\xi(n_k)}{n_k^{1-a_{n_k}\log(\alpha)}} = 0,
	\end{align*}
	implying that (i) $\lim_{k\rightarrow\infty}a_{n_k} = \infty$ and (ii) $\limsup_{k\rightarrow\infty}a_{n_k} \leq 1/\log(\alpha)$, a contradiction.

	Second, by the equality in \eqref{eq:proof_exp_opt_poly_goal},  $A(r_n)+\varepsilon(r_n,n) \asymp \log^{-\xi}(n)$ if and only if 
	\begin{align}\label{eq:dist_poly_proof}
		a_n^{-1} = O(1), \quad \frac{\log^\xi (n)}{ n^{1-a_n \log(\alpha)}}=O(1).
	\end{align}
	The second condition in \eqref{eq:dist_poly_proof} implies that $a_n = O(1)$. 
	Thus, \eqref{eq:dist_poly_proof} is the same as  
	\begin{align*}
	a_n \asymp 1, \quad \frac{\log^\xi (n)}{ n^{1-a_n \log(\alpha)}}=O(1).
	\end{align*}
	
	\item[(iv)] (Logarithmic Decay). 
	First,  for any sequence $\{r_n = a_n \log(n)\}$, 
	\begin{align}\label{eq:proof_poly_opt_log_goal}
	\liminf_{n\rightarrow \infty}\frac{\log^{-\xi}(r_n) + \alpha^{r_n}/n}{[\log\log(n)]^{-\xi}} & = 
	\liminf_{n\rightarrow \infty}
	\left\{
	\left[
	\frac{\log(a_n)}{\log\log(n)} + 1
	\right]^{-\xi}
	+ \frac{[\log\log(n)]^\xi}{n^{1-a_n\log(\alpha)}}
	\right\} >0. 
	\end{align}
	If \eqref{eq:proof_poly_opt_log_goal} does not hold, then there must exist a subsequence $\{a_{n_k}\}$ such that
	\begin{align*}
	\lim_{k\rightarrow \infty} \left[
	\frac{\log(a_{n_k})}{\log\log(n_k)} + 1
	\right]^{-1} = 0 \ 
	\text{ and } \ 
	\lim_{k\rightarrow \infty}\frac{[\log\log(n_k)]^\xi}{n_k^{1-a_{n_k}\log(\alpha)}}=0, 
	\end{align*}
	which imply respectively that (i) $\lim_{k\rightarrow\infty}a_{n_k} = \infty$ and (ii) $\limsup_{k\rightarrow\infty}a_{n_k} \leq 1/\log(\alpha)$, a contradiction.

	Second, by the equality in 
	\eqref{eq:proof_poly_opt_log_goal}, $A(r_n) + \varepsilon(r_n,n) \asymp [\log\log(n)]^{-\xi}$ if and only if 
	\begin{align*}
	\liminf_{n\rightarrow \infty} \frac{\log(a_n)}{\log\log(n)} >-1, \quad \frac{[\log\log(n)]^\xi}{n^{1-a_n \log(\alpha)}} =  O( 1 ).
	\end{align*}
\end{itemize}

\section{Proof of Theorem \ref{thm:binary_zero_tau}}
\begin{itemize}
	\item[(i)] (Hard Thresholding). 
	First, for any sequence  $\{r_n\}$, 
	\begin{align}\label{eq:proof_exp_opt_hard_goal_binary_tau2_zero}
	\liminf_{n\rightarrow \infty}
	\frac{A(r_n) + \varepsilon(r_n,n)}{(1-M^{-r_0})^n} = 
	\liminf_{n\rightarrow \infty}
	\left\{
	\frac{A(r_n)}{(1-M^{-r_0})^n} + \frac{\varepsilon(r_n,n)}{(1-M^{-r_0})^n}
	\right\}
	>0.  
	\end{align}
	If \eqref{eq:proof_exp_opt_hard_goal_binary_tau2_zero} does not hold, then there exists a subsequence $\{r_{n_k}\}$ such that as $k \rightarrow \infty$, 
	\begin{align}\label{eq:proof_binary_zero_tau2_hard_threshold}
	    \frac{A(r_{n_k})}{(1-M^{-r_0})^{n_k}} \rightarrow 0,
	    \ \ \ 
	    \text{and}
	    \ \ \ 
	    \frac{\varepsilon(r_{n_k}, n_k)}{(1-M^{-r_0})^{n_k}} \rightarrow 0.
	\end{align}
	Note that the left-hand condition in \eqref{eq:proof_binary_zero_tau2_hard_threshold} implies that $\liminf_{k\rightarrow \infty} r_{n_k} \ge r_0$. 
	From \eqref{eq:asymp_rate_binary_zero_tau2_hard_threshold}, this further implies that $\varepsilon(r_{n_k}, n_k) \asymp (1-M^{-r_0})^{n_k}$, which contradicts the right-hand condition in \eqref{eq:proof_binary_zero_tau2_hard_threshold}.

	Second, from \eqref{eq:proof_exp_opt_hard_goal_binary_tau2_zero}, we can verify that 
	that $A(r_n) + \varepsilon(r_n,n) \asymp (1-M^{-r_0})^n$ if and only if 
	$\liminf_{n\rightarrow \infty} r_n \ge r_0$.

	\item[(ii)] (Exponential Decay).
	\begin{itemize}
		\item[(1)] When $e^{-\xi} < 1/M$,  for any sequence $\{r_n\}$, we have 
		\begin{align*}
		\liminf_{n\rightarrow \infty}
		\frac{A(r_n) + n^{-1}}{n^{-1}} = 
		\liminf_{n\rightarrow \infty}
		\left\{
		nA(r_n)+ 1
		\right\}
		>0.  
		\end{align*}
		We can then verify that $A(r_n) + \overline{\varepsilon}(r_n,n) \asymp n^{-1}$ if and only if 
		$n e^{-\xi r_n} = O(1)$. 
		
		\item[(2)] When $e^{-\xi} = 1/M$, $A(r) + \overline{\varepsilon}(r,n) \asymp M^{-r} + r/n$. 
		By the same logic as for proving Theorem \ref{th:cont}(ii) with $\xi  = \log(M)$ and $\alpha = 1$, 
		we can show that the optimal rate for $\overline{L}_n$ is $\overline{L}_n \asymp n^{-1} \log(n)$, and it is achieved if and only if 
		$
		r_n = a_n \log(n)
		$
		with $a_n$ satisfying $a_n \asymp 1$ and $n^{1-a_n\log(M)}/\log(n) = O(1)$. 
		
		\item[(3)] When $e^{-\xi} > 1/M$, for any sequence $\{r_n = a_n \log(n) \}$, 
		\begin{align}\label{eq:binary_zero_exp_ratio}
			\liminf_{n\rightarrow \infty}  
			\frac{e^{-\xi r_n} + e^{-\xi r_n} M^{r_n}/n }{n^{-\xi/\log(M)}} 
			& = 
			\liminf_{n\rightarrow \infty}  
			\left\{
			n^{-\xi (a_n - \log^{-1}(M))} + n^{(\log(M) - \xi) (a_n - \log^{-1}(M) )}
			\right\} 
			\nonumber
			\\
			& > 0. 
		\end{align}
		If \eqref{eq:binary_zero_exp_ratio} does not hold, then there exist a subsequence $\{a_{n_k}\}$ such that \\
		$
		n^{-( a_{n_k} - \log^{-1}(M))} = o(1)
		$
		and 
		$
		n^{a_{n_k} - \log^{-1}(M)} = o(1),
		$
		a contradiction. 
		
		Second, by the equality in \eqref{eq:binary_zero_exp_ratio}, 
		$A(r_n)+ \overline{\varepsilon}(r_n,n) \asymp n^{-\xi/\log(M)}$ if and only if 
		$
		n^{a_n - \log^{-1}(M)}  \asymp 1,
		$
		which is further equivalent to $
		n^{a_n\log(M) - 1}  \asymp 1.
		$
	\end{itemize}

	\item[(iii)] (Polynomial Decay). 
	First, for any sequence $\{r_n = a_n \log(n) \}$,
	\begin{align}\label{eq:binary_zero_poly_ratio}
	\liminf_{n\rightarrow \infty}
	\frac{r_n^{-\xi} + r_n^{-\xi} M^{r_n}/n}{\log^{-\xi}(n)}
	& =
	\liminf_{n\rightarrow \infty}
	\left\{
	a_n^{-\xi} + a_n^{-\xi} n^{a_n \log(M) - 1}
	\right\}
	>0.  
	\end{align}
	If \eqref{eq:binary_zero_poly_ratio} does not hold, then there exists a subsequence $\{a_{n_k}\}$ such that
	\begin{align*}
	\lim_{k\rightarrow\infty}a_{n_k}^{-1} = 0 \ 
	\text{ and } \ 
	\lim_{k\rightarrow \infty} 
	a_{n_k}^{-\xi} n_k^{a_{n_k} \log(M) - 1} = 0
	\end{align*}
	which implies that (i) $\lim_{k\rightarrow\infty}a_{n_k} = \infty$ and (ii) $a_{n_k} = O(1)$, a contradiction.   
	
	Second, by the equality in \eqref{eq:binary_zero_poly_ratio},  $A(r_n)+ \overline{\varepsilon}(r_n,n) \asymp \log^{-\xi}(n)$ if and only if 
	\begin{align}\label{eq:binary_zero_poly_cond}
		a_n^{-1} = O(1), \quad a_n^{-\xi} n^{a_n \log(M) - 1}=O(1).
	\end{align}
	Note that the right-hand condition in \eqref{eq:binary_zero_poly_cond} implies that $a_n = O(1)$. Thus, \eqref{eq:binary_zero_poly_cond} is also equivalent to 
	\begin{align*}
	a_n \asymp 1, \quad n^{a_n \log(M) - 1}=O(1).
	\end{align*}
	
	\item[(iv)] (Logarithmic Decay). 
	First,  for any sequence $\{r_n = a_n \log(n)\}$, 
	\begin{align}\label{eq:proof_poly_opt_log_goal_no_intrin}
	& \quad \ \liminf_{n\rightarrow \infty}\frac{\log^{-\xi}(r_n) + \log^{-\xi}(r_n) M^{r_n}/n }{[\log\log(n)]^{-\xi}} 
	\nonumber
	\\
	& = 
	\liminf_{n\rightarrow \infty}
	\left\{
	\left(
	\frac{\log(a_n)}{\log\log(n)} + 1
	\right)^{-\xi}
	+ 
	\left(
	\frac{\log(a_n)}{\log\log(n)} + 1
	\right)^{-\xi}
	n^{a_n \log(M) - 1}
	\right\} >0.
	\end{align}
	If \eqref{eq:proof_poly_opt_log_goal_no_intrin} does not hold, then there exists a subsequence $\{a_{n_k}\}$ such that
	\begin{align*}
	\lim_{k\rightarrow \infty} \left(
	\frac{\log(a_{n_k})}{\log\log(n_k)} + 1
	\right)^{-1} = 0 \ 
	\text{ and } \ 
	\lim_{k\rightarrow \infty}
	\frac{
	n_k^{a_{n_k} \log(M) - 1}
	}
	{
	\left(
	\frac{\log(a_{n_k})}{\log\log(n_k)} + 1
	\right)^{\xi}
	}	
	=0, 
	\end{align*}
	which implies that (i) $\lim_{k\rightarrow\infty}a_{n_k} = \infty$ and (ii) $a_{n_k} = O(1)$, a contradiction.   
	
	Second, by the equality in 
	\eqref{eq:proof_poly_opt_log_goal_no_intrin}, $A(r_n) + \overline{\varepsilon}(r_n,n) \asymp [\log\log(n)]^{-\xi}$ if and only if 
	\begin{align}\label{eq:binary_zero_log_cond}
	\liminf_{n\rightarrow \infty} \frac{\log(a_n)}{\log\log(n)} >-1, 
	\ \ 
	\text{and}
	\ \ 
	\left(
	\frac{\log(a_n)}{\log\log(n)} + 1
	\right)^{-\xi}
	n^{a_n \log(M) - 1}
	=  O( 1 ).
	\end{align}
	The right-hand condition in \eqref{eq:binary_zero_log_cond} implies that $a_n = O(1)$, which in turn implies that $\limsup_{n\rightarrow \infty}\frac{\log(a_n)}{\log\log(n)} \le 0$. 
	Thus, \eqref{eq:binary_zero_log_cond} is also equivalent to 
	\begin{align*}
	\liminf_{n\rightarrow \infty} \frac{\log(a_n)}{\log\log(n)} >-1, 
	\ \ 
	\text{and}
	\ \ 
	n^{a_n \log(M) - 1}
	=  O( 1 ).
	\end{align*}
\end{itemize}

\section{Proof of Theorem \ref{thm:lower_bound_binary_tau2_zero}}

First, for the exponential decay case, we have for any sequence $\{r_n\}$, 
    \begin{align*}
			\liminf_{n\rightarrow \infty}  
			\frac{e^{-\xi r_n} +\varepsilon(r_n,n) }{n^{-\xi/\log(M)}} 
			& = 
			\liminf_{n\rightarrow \infty}  
			\left\{
			\left( \frac{M^{r_n}}{n} \right)^{-\xi/\log(M)} 
			+ 
			\frac{\varepsilon(r_n, n)}{n^{-\xi/\log(M)}}
			\right\} 
			> 0. 
	\end{align*}
If this inequality is false, 
then there exists a subsequence $\{r_{n_k}\}$ such that 
as $k \rightarrow \infty$, 
$M^{r_{n_k}}/n_k \rightarrow \infty$ and $\varepsilon(r_{n_k}, n_k)/n_k^{-\xi/\log(M)} \rightarrow 0$. 
From Lemma \ref{lemma:lowerbound_large_r}, 
there exists a countably infinte set $S$ and a subsequence $\{\tilde{r}_{n}: n \in S\}$ such that $\lim_{n \in S, n \rightarrow \infty} M^{\tilde{r}_{n}}/n = \tilde{c}$ for some constant $\tilde{c}>0$ and  
$\liminf_{n \in S, n \rightarrow \infty} \varepsilon(r_{n}, n) / A(\tilde{r}_{n}) > 0$. 
These imply that as $n \in {S}$ goes to infinity, we have
$$
\frac{A(\tilde{r}_{n})}{n^{-\xi/\log(M)}} = \frac{A(\tilde{r}_{n})}{\varepsilon(r_{n}, n)} \cdot \frac{\varepsilon(r_{n}, n)}{n^{-\xi/\log(M)}} \rightarrow 0,
$$
but contradictorily also
\begin{align*}
    \frac{ A(\tilde{r}_{n}) }{ n^{-\xi/\log(M)} }
    & \asymp
    \frac{e^{-\xi \tilde{r}_{n}}}{ n^{-\xi/\log(M)} }
    = 
    \frac{M^{-\xi \tilde{r}_{n}/\log(M)}}{ n^{-\xi/\log(M)} }
    = 
    \left( \frac{M^{\tilde{r}_{n}}}{n} \right)^{-\xi/\log(M)}
    \rightarrow 
    \tilde{c}^{-\xi/\log(M)} > 0.
\end{align*}

For the polynomial decay case, we have for any sequence $\{r_n\}$, 
\begin{align*}
	\liminf_{n\rightarrow \infty}
	\frac{r_n^{-\xi} + \varepsilon(r_n, n)}{\log^{-\xi}(n)}
	& =
	\liminf_{n\rightarrow \infty}
	\left\{
	\left( \frac{r_n}{\log(n)} \right)^{-\xi}
	+ 
	\frac{\varepsilon(r_n, n)}{\log^{-\xi}(n)}
	\right\}
	>0.  
\end{align*}
If this inequality is false, then there exists a subsequence $\{r_{n_k}\}$ such that as $k \rightarrow \infty$, 
$r_{n_k}/\log(n_k) \rightarrow \infty$, 
and 
$\varepsilon(r_{n_k}, n_k)/\log^{-\xi}(n_k) \rightarrow 0$. 
This implies that as $k \rightarrow \infty$, 
\begin{align*}
    \frac{M^{r_{n_k}} }{n_k}
    & = \frac{\exp\{[r_{n_k}/\log(n_k)] \cdot \log(n_k) \cdot \log(M)\}}{n_k}
    =
    n_k^{[r_{n_k}/\log(n_k)]\log(M) - 1} \rightarrow \infty.
\end{align*}
From Lemma \ref{lemma:lowerbound_large_r}, 
there exist a countably infinite set $S$ and a subsequence $\{\tilde{r}_{n}: n \in S\}$ 
such that $\lim_{n \in S, n \rightarrow \infty} M^{\tilde{r}_{n}}/n = \tilde{c}$ for some constant $\tilde{c}>0$ and  
$\liminf_{n \in S, n \rightarrow \infty} \varepsilon(r_{n}, n) / A(\tilde{r}_{n}) > 0$. 
These imply that as $n \in S$ goes to infinity, we have 
$$
 \frac{A(\tilde{r}_{n})}{\log^{-\xi}(n)}
    = 
    \frac{A(\tilde{r}_{n})}{\varepsilon(r_{n}, n)} \cdot \frac{\varepsilon(r_{n}, n)}{\log^{-\xi}(n)}
    \rightarrow 0,
$$
but contradictorily also
\begin{align*}
    \frac{A(\tilde{r}_{n})}{\log^{-\xi}(n)}
    & \asymp
    \left( \frac{\tilde{r}_{n}}{\log(n)} \right)^{-\xi}=
    \left\{
    \left[\frac{\log(M^{\tilde{r}_n}/n)}{\log (n)} + 1 \right] 
    \frac{1}{\log (M)}
    \right\}^{-\xi}
    \rightarrow 
    \log^{\xi} (M),
\end{align*}
where the last convergence holds because $\lim_{n \in S, n \rightarrow \infty} M^{\tilde{r}_{n}}/n = \tilde{c}$.

For the logarithmic decay case, 
we have for any sequence $\{r_n\}$,
	\begin{align*}
	\liminf_{n\rightarrow \infty}\frac{\log^{-\xi}(r_n) + \varepsilon(r_n, n) }{[\log\log(n)]^{-\xi}} 
	= 
	\liminf_{n\rightarrow \infty}
	\left\{
	\left( \frac{\log(r_n)}{\log\log(n)} \right)^{-\xi}
	+ \frac{\varepsilon(r_n, n)}{[\log\log(n)]^{-\xi}}
	\right\} > 0. 
	\end{align*}
	If this inequality is false, then there exists a subsequence $\{r_{n_k}\}$ such that as $k \rightarrow \infty$, 
	$\log(r_{n_k})/\log\log(n_k) \rightarrow \infty$ and 
	$\varepsilon(r_{n_k}, n_k)/\{\log\log(n_k)\}^{-\xi} \rightarrow 0$. 
	These further imply that as $k\rightarrow \infty$, 
	\begin{align*}
	    \frac{M^{r_{n_k}}}{n_k}
	    & = 
	    n_k^{[r_{n_k}/\log(n_k)]\log(M) - 1} 
	    = 
	    n_k^{\log(M) \cdot \exp \left[ \log\log(n_k) \cdot 
	\left\{
	\frac{\log(r_{n_k})}{\log\log(n_k)} -  1
	\right\} \right] - 1}
	    \rightarrow \infty. 
	\end{align*}
    From Lemma \ref{lemma:lowerbound_large_r}, 
    there exist a countably infinite set $S$
    and
    a subsequence $\{\tilde{r}_{n}\}$ such that $\lim_{n \in S, n \rightarrow \infty} M^{\tilde{r}_{n}}/n = \tilde{c}$ for some finite constant $\tilde{c}>0$ and  
    $\liminf_{n \in S, n \rightarrow \infty} \varepsilon(r_{n}, n) / A(\tilde{r}_{n}) > 0$. 
    These imply that as $n\in S$ goes to infinity,  we have 
    \begin{align*}
        \frac{A(\tilde{r}_{n})}{[\log\log(n)]^{-\xi}}
        = 
        \frac{A(\tilde{r}_{n})}{\varepsilon(r_{n}, n)} \frac{\varepsilon(r_{n}, n)}{[\log\log(n)]^{-\xi}}
        \rightarrow 0;
    \end{align*}
    but also contradictorily 
    \begin{align*}
        \frac{A(\tilde{r}_{n})}{[\log\log(n)]^{-\xi}}
        & \asymp
        \left[ \frac{\log(\tilde{r}_{n})}{\log\log(n)} \right]^{-\xi}
        \\
        & = 
        \left\{ \log\left[ \left( \frac{\log(M^{\tilde{r}_n}/n)}{\log(n)} + 1 \right) \frac{1}{\log (M)} \right]/\log\log(n) + 1 \right\}^{-\xi}
        \rightarrow 1,
    \end{align*}
    where the last convergence holds because $\lim_{n \in S, n \rightarrow \infty} M^{\tilde{r}_{n}}/n = \tilde{c}$.

\section{Explore the Practicality of the MR Methods}\label{sec:practical}

\subsection{Cross-validation}\label{sec:cross}
We want to choose the resolution level $R$ such that the estimated prediction function, $g(\vec{\bm{x}}_{R};\hat{\bm{\theta}}_{R})$, has the smallest prediction error \eqref{eq:decomposition}, or the smallest prediction error \eqref{eq:decomposition_average_training} averaging over the training set. 
A usual strategy is to first estimate the prediction error at each resolution $r$, 
and then estimate the optimal resolution by the $r$-value that minimizes the estimated prediction error. 
To avoid over-fitting, we typically split the training set randomly into two parts: 
one for estimating the prediction function and the other for estimating the prediction error. 
Here we adopt the leave-one-out cross-validation approach;  see, e.g., \citet{Stone1978} for a review. 
Specifically, 
at each resolution $r$, for $1\leq j\leq n$, we use the $n-1$ samples $\{y_i, \vec{\bm{x}}_{ir} \}_{i\neq j}$ to obtain an estimator $\hat{\bm{\theta}}_{r,j}$ of the parameter $\bm{\theta}_r^*$. 
The corresponding prediction function is then  
$
g(\vec{\bm{x}}_{r};\hat{\bm{\theta}}_{r,j}),
$
and the prediction of the $j$th sample's response is thus $\hat{y}_j \equiv g(\vec{\bm{x}}_{jr}; \hat{\bm{\theta}}_{r,j})$. 
Importantly, the prediction for the $j$th sample depends only on the other samples excluding itself. 
We can then estimate the prediction error at resolution $r$ by the average prediction error for all $n$ samples from the cross-validation: 
\begin{align}\label{eq:CV}
\CV_n(r) = \frac{1}{n} \sum_{j=1}^{n} \loss_{\odot} \left( y_j, \hat{y}_j \right),
\end{align}
where the subscript $n$ emphasizes the dependence on the sample size. 
Let $\PE_n(r) \equiv \E_n \E[\loss_{\odot}( Y, g(\vec{\bm{X}}_{R}; \hat{\bm{\theta}}_R)]$ denote 
the prediction error \eqref{eq:decomposition_average_training} averaging over the training set of size $n$. 
We can show that the cross-validation error $\CV_n(r)$ in \eqref{eq:CV} at resolution $r$ with sample size $n$ is an unbiased estimator for the prediction error \eqref{eq:decomposition_average_training} at resolution $r$ but with sample size $n-1$, i.e., $\E_n[\CV_n(r)] = \PE_{n-1}(r)$. 
When $r$ is not too close to $n$, we can expect the prediction errors \eqref{eq:decomposition_average_training} to have similar values at sample sizes $n$ and $n-1$, i.e., $\PE_{n}(r) \approx \PE_{n-1}(r)$;  
hence $\CV_n(r)$ can serve as a good estimator for the prediction error at resolution $r$. 

To estimate the optimal resolution, 
we can use the resolution level $\hat{R}_n$ that minimizes the cross-validation error $\CV_n(r)$ over a reasonable range of $r$, say from $0$ to some $\overline{r}_n>0$. 
In general, when we believe the intrinsic variance is not zero,  we do not want $\overline{r}_n$ to be too large compared to the sample size $n$. 
First, when the resolution $r$ is too large, the estimator for $\bm{\theta}_r^*$ from 
minimizing the empirical risk
may not be unique.
That is, there may be multiple minimizers for the empirical loss 
$
n^{-1}
\sum_{i=1}^{n}
\loss_{\odot}\left(y_i, g(\vec{\bm{x}}_{ir};\bm{\theta}_r) \right), 
$
and thus the estimated prediction function itself has some variability. 
Second, as we discussed before, $\CV_n(r)$ is essentially estimating the prediction performance at sample size $n-1$, which can be quite different from that at sample size $n$ when the resolution level $r$ becomes close to $n$.  
Third, and more importantly, 
as discussed in Sections \ref{sec:theory_linear} and \ref{sec:theory_tree}, when there is positive intrinsic variance, 
a necessary condition for 
$\varepsilon(r, n) = o(1)$ is often that $\dim(\bm{\theta}_r)/n = o(1)$, which generally requires that $r = o(n)$. 

However, 
in practice with a finite sample size $n$, these rate results often provide only a rough idea of 
the choice of $\overline{r}_n$. 
Besides, as we demonstrated in Sections~\ref{sec:linear_zero_tau} and \ref{sec:determin_categorial}, when there is no intrinsic error, it is possible that the optimal resolution can be of the same size of $n$ and even close to $n$ up to certain constant. 
Therefore, in the following simulation, we choose $\overline{r}_n$ to be almost the largest $r$-value where the empirical risk minimizer exists.
It turns out cross-validation is fairly robust even when we search over a large range of possible resolutions. 
Nevertheless,
obtaining general finite-sample theoretical properties of the estimated optimal resolution $\hat{R}_n$ as well as the corresponding prediction performance is very challenging, especially when $\overline{r}_n$ is relatively or even moderately large compared to $n$.  
Therefore, below we report a simulation study based on a special case of linear models as in Section \ref{sec:linear_model}, as a first step to understand the finite-sample
properties of the estimated optimal resolution $\hat{R}_n$. 
All theoretical derivations are collected at the end.

\subsection{Normal linear model with infinitely many continuous covariates}\label{sec:simu_example}
We consider again 
model \eqref{eq:linear_model}, 
where $Y$ given $\vec{\bm{X}}_{\infty}$ follows a linear model. 
At each resolution $r$, 
we regress the responses $y_i$'s on the covariates $\vec{\bm{x}}_{ir}$'s in the training set to ascertain the least squares coefficient $\hat{\bm{\theta}}_r$, and then use $g(\vec{\bm{x}}_{r}, \hat{\bm{\theta}}_r)$ as our prediction. 
Because the prediction function is linear in covariates, 
the prediction performance at all resolution levels is invariant under a Gram--Schmidt orthogonalization of the original covariates. 
Therefore, without loss of generality, we further assume the covariates $X_1, X_2, \ldots, $ are i.i.d. realizations of $\mathcal{N}(0,1)$. 
The ultimate risk is still $\tau^2$, 
the resolution bias reduces to $
\sum_{k=r+1}^\infty \beta_k^2,
$
and the estimation error simplifies to 
$
\|\hat{\bm{\theta}}_r-\bm{\theta}^*_r\|_2^2
$
with 
$
\bm{\theta}^*_r = (\beta_0, \beta_1, \ldots, \beta_r)^\top. 
$
Moreover, 
from Appendices \ref{sec:simu_example} and \ref{sec:proof_linear}, for the linear model with i.i.d.\ standard normal covariates, the prediction error 
$\PE_n(r)$ as in \eqref{eq:decomposition_average_training} and \eqref{eq:pred_loss_linear} averaged over the training set of size $n$ simplifies to
\begin{align}\label{eq:pred_error_linear}
	\PE_n(r)
	= \tau^2 + A(r) + 
	\frac{A(r)+\tau^2}{n-r-2}\left(
	\frac{n-2}{n}+r
	\right)
	= \left[
	\tau^2 + A(r)
	\right]
	\frac{(n+1)(n-2)}{n(n-r-2)}. 
\end{align}
Below we consider three strategies to estimate the prediction error at each resolution. 

We will use $\tilde{\bm{Y}} = (y_1, y_2, \ldots, y_n)^\top \in \mathbb{R}^n$
and 
$\tilde{\bm{X}}_{r} = (\bm{x}_{1\vec{r}}, \bm{x}_{2\vec{r}}, \ldots, \bm{x}_{n\vec{r}} )^\top \in \mathbb{R}^{n\times (r+1)}$
to denote the response vector and covariate matrix for the training set, 
and 
$\bm{H}_r = \tilde{\bm{X}}_{r} (
\tilde{\bm{X}}_{r}^\top \tilde{\bm{X}}_{r}
)^{-1} \tilde{\bm{X}}_{r}^\top \in \mathbb{R}^{n\times n}$ to denote the projection matrix onto the column space of $\tilde{\bm{X}}_{r}$. 
Then at each resolution $r$, the least squares coefficient is 
$
\hat{\bm{\theta}}_{r} =
( \tilde{\bm{X}}_{r}^\top \tilde{\bm{X}}_{r})^{-1} \tilde{\bm{X}}_{r}^\top \tilde{\bm{Y}}, 
$
and the mean squared error is 
$
\hat{\sigma}^2_r = n^{-1}\| \tilde{\bm{Y}} - \tilde{\bm{X}}_{r}\hat{\bm{\theta}}_{r} \|_2^2. 
$

First, we consider the leave-one-out cross validation in Section \ref{sec:cross}. For our linear model,
$\CV_n(r)$ has the following equivalent form that is much easier to compute: 
\begin{align}\label{eq:CV_linear}
\CV_n(r) & 
= \frac{1}{n} \sum_{i=1}^{n}
\left(
\frac{y_i - [\bm{H}_r \tilde{\bm{Y}}]_i } {1 - [\bm{H}_r]_{ii}}
\right)^2, 
\end{align}
where 
$[\bm{H}_r \tilde{\bm{Y}}]_i$ is the $i$th coordinate of $\bm{H}_r \tilde{\bm{Y}}$, 
and 
$[\bm{H}_r]_{ii}$ is the $i$th diagonal element of $\bm{H}_r$. 
Second, we consider a robust version of the model selection criterion in the style of AIC for M-estimation \citep{Tharmaratnam2013}. 
In this linear model case,  
the information-criterion type estimator for the prediction error at resolution $r$ is 
\begin{align}\label{eq:IC}
\IC_n(r) 
= 
\hat{\sigma}_r^2 \cdot \frac{n+2(r+1)}{n}
= \hat{\sigma}^2_r + 2 \hat{\sigma}^2_r \frac{r+1}{n}, 
\end{align} 
where the second term serves as a penalty for the resolution level or equivalently the number of unknown parameters. 
Third, we use an unbiased estimator for the prediction error \eqref{eq:pred_error_linear}, which has 
the following form: 
\begin{align}\label{eq:UE}
\UE_n(r)  
= \hat{\sigma}^2_r \cdot
\frac{(n-2)(n+1)}{(n-r-2)(n-r-1)} = 
\hat{\sigma}^2_r + 
\hat{\sigma}^2_r \left[
\frac{(n-2)(n+1)}{(n-r-2)(n-r-1)} - 1
\right].
\end{align}
Again, we can view the second term in $\UE_n(r)$ as a penalty for the resolution level $r$. 
However, when $r$ is large and even becomes close to $n$, compared to $\IC_n(r)$ in \eqref{eq:IC}, the penalty in $\UE_n(r)  $ in \eqref{eq:UE} is much larger. Indeed, as we will show in  Section~\ref{sec:simulation_tech}, whereas the bias in $\CV_n(r)$ for estimating the prediction error $\PE_n(r)$ is rather controllable, $\IC_n(r)$ can dramatically underestimate $\PE_n(r)$ when $r$ is close to $n$.

For the following simulations, we generate the training set as i.i.d.\  samples from
\begin{align*}
& Y = \beta_0 + \sum_{j=1} \beta_j X_j + \varepsilon, 
\quad
X_1, X_2, \ldots \overset{i.i.d.}{\sim} \mathcal{N}(0,1), \quad 
\varepsilon \sim 
\mathcal{N}(0, \tau^2), 
\quad 
\vec{\bm{X}}_{\infty} \ind \varepsilon. 
\end{align*}
We consider three choices for the $\beta_j$'s, which correspond to exponential, polynomial and logarithmic decay resolution biases $A(\cdot)$'s:
$\beta_0 = 0$, and for $j\ge 1$,
\begin{align*}
\beta_j = 
\begin{cases}
\sqrt{e^{-(j-1)} - e^{-j}}, & \text{for exponential decay } A(\cdot),\\
\sqrt{1/j - 1/(j+1)}, & \text{for polynomial decay } A(\cdot),\\
\sqrt{\log(2)}\sqrt{1/\log(j+1) - 1/\log(j+2)}, & \text{for logarithmic decay } A(\cdot).
\end{cases}
\end{align*}

\subsection{Empirical Findings}
We begin by setting $\tau^2 = 0.5$ and $n=50$.  Figure \ref{fig:estimate_pred_loss} displays the comparisons, where each sub-figure plots the logarithm of the average value of the estimated prediction error over 500 simulated training sets, as well as the logarithm of the true average prediction error, against the resolution level $r$. The top row corresponds to $\tau^2>0$, which is the focus of this sub-section; the bottom row is for the setting where $\tau^2=0$, to be discussed 
shortly. 
When $r$ is small, all three estimators are approximately unbiased, but with $\IC_n(\cdot)$ deteriorating very quickly when $r$ becomes moderate.
When $r$ approaches $n$, $\CV_n(\cdot)$ overestimates the prediction error. All these numerical results are consistent with the theoretical calculations in Section~\ref{sec:simulation_tech}.

\begin{table}[h]
    \centering
	\caption{Estimated resolution and prediction error when $n = 50$, with $\tau^2=\frac{1}{2}$.}\label{tab:tau2_half_n50}
	\begin{tabular}{cccccc}
		\toprule
		Type / $r_{\opt}$ / $\PE_n(r_{\opt})$ & Method & $\hat{R}$ & $95\%$ QR & std. $\PE_n(\hat{R})$ &  $95\%$ QR  \\
		\midrule
		Exponential & Oracle & -- & -- & $1.00$ & $[0.92, \  1.18]$ \\
		$r_{\opt} = 4$ & CV & $6$ & $[2, \  20]$ & $1.59 $ & $[0.93, \  1.89]$ \\
		$\PE_n(r_{\opt}) = 0.5767$ & UE & $7$ & $[2, \ 47]$ &  $2.33$ & $[0.92, \ 10.86]$ \\
		&  IC & $47$ & $[46, \ 47]$  & $37.89$ & $[4.59 , \ 172.44]$\\
		\midrule
		Polynomial & Oracle & -- & -- & $1.00$ & $[0.87, \  1.24]$ \\
		$r_{\opt} = 7$ & CV & $10$ & $[2, \  44]$ & $2.38$ & $[0.90, \  7.27]$ \\
		$\PE_n(r_{\opt}) = 0.7463$ & UE & $11$ & $[2, \ 47]$ &  $2.81$ & $[0.91, \ 10.54]$ \\
		& IC & $47$ & $[45, \ 47]$  & $31.31$ & $[4.04, \ 176.27]$\\
		\midrule
		Logarithmic & Oracle & -- & -- & $1.00$ & $[0.89, \  1.21]$ \\
		$r_{\opt} = 6$ & CV & $9$ & $[2, \  41]$ & $1.42$ & $[0.91, \  4.47]$ \\
		$\PE_n(r_{\opt}) = 0.9714$ & UE & $10$ & $[2, \ 46]$ &  $3.61$ & $[0.91, \ 12.26]$ \\
		& IC & $47$ & $[46, \ 47]$  & $27.77$ & $[4.09, \ 117.36]$\\
		\bottomrule
	\end{tabular}
\end{table}

\begin{table}[htbp]
    \centering
	\caption{Estimated resolution and prediction error when $n = 200$, with $\tau^2=\frac{1}{2}$}\label{tab:tau2_half_n200}
	\begin{tabular}{cccccc}
		\toprule
		Type / $r_{\opt}$ / $\PE_n(r_{\opt})$ & Method & $\hat{R}$ & $95\%$ QR & std. $\PE_n(\hat{R})$ &  $95\%$ QR  \\
		\midrule
		Exponential & Oracle & -- & -- & $1.00$ & $[0.97, \  1.04]$ \\
		$r_{\opt} = 6$ & CV & $7$ & $[4, \  16]$ & $1.02$ & $[0.98, \  1.13]$ \\
		$\PE_n(r_{\opt}) = 0.5208$ & UE & $7$ & $[4, \ 16]$ &  $1.02$ & $[0.98, \ 1.12]$ \\
		&  IC & $197$ & $[196, \ 197]$  & $174.69$ & $[19.84, \ 691.16]$\\
		\midrule
		Polynomial & Oracle & -- & -- & $1.00$ & $[0.95, \  1.07]$ \\
		$r_{\opt} = 17$ & CV & $18$ & $[8, \  34]$ & $1.22$ & $[0.96, \  1.13]$ \\
		$\PE_n(r_{\opt}) = 0.6108$ & UE & $19$ & $[8, \ 35]$ &  $1.35$ & $[0.96, \ 1.13]$ \\
		& IC & $197$ & $[196, \ 197]$  & $117.58$ & $[16.98, \ 485.41]$\\
		\midrule
		Logarithmic & Oracle & -- & -- & $1.00$ & $[0.95, \  1.07]$ \\
		$r_{\opt} = 17$ & CV & $19$ & $[7, \  39]$ & $1.03$ & $[0.96, \  1.16]$ \\
		$\PE_n(r_{\opt}) = 0.8085$ & UE & $20$ & $[7, \ 42]$ &  $1.76$ & $[0.96, \ 1.20]$ \\
		& IC & $197$ & $[196, \ 197]$  & $152.57$ & $[17.61, \ 799.91]$
		\\
		\bottomrule
	\end{tabular}
\end{table}

Tables \ref{tab:tau2_half_n50} and \ref{tab:tau2_half_n200} continue the comparisons with the training set sizes $50$ and $200$, respectively. 
In both tables, the search range of the resolution level $[0,\overline{r}_n]$, is set to be $[0, n-3]$. 
In both tables, the first column describes the decay rate of the resolution bias, the theoretical optimal resolution $r_{\opt}$ minimizing the prediction error in \eqref{eq:pred_error_linear}, and the corresponding minimum prediction error. The second column describes the method, where ``Oracle'' means using the theoretical optimal $r_{\opt}$ for prediction. 
The third and fourth columns show the average and $95\%$ quantile range (QR) of the estimated optimal resolution using different methods over 500 simulated data sets. 
The fifth and sixth columns show the average and $95\%$ quantile range of the average prediction error, standardized (std.)  by the oracle error $\PE_n(r_\opt)$, 
using the prediction function $g(\vec{\bm{x}}_{r}, \hat{\bm{\theta}}_r)$ 
at the estimated resolution $r$ over 500 simulated data sets. 

Tables \ref{tab:tau2_half_n50} and \ref{tab:tau2_half_n200} further confirm our theoretical calculations in Section~\ref{sec:simulation_tech} that both cross validation and unbiased estimation perform much better than the information criterion, which overestimates the optimal resolution dramatically.
Both tables also indicate that the optimal resolutions under polynomial decay and logarithmic decay are very similar, which some may consider as contradicting Theorem \ref{th:cont}, which  suggests that the optimal resolutions under exponential, polynomial and Logarithmic decay resolution biases are, respectively, 
in the increasing orders of $\log(n)$, $n^{1/2}$ and $n/\log(n)$. 
However,  we must keep in mind that these asymptotic rates are \textit{asymptotic}.  To illustrate this point, we plot in 
Figure \ref{fig:exact_optimal_rate_linear}(a) the optimal resolution against the sample size for resolution biases of different decay rates. We see when $n$ is larger than $250$, the optimal resolution under logarithmic decay is larger than that under polynomial decay, and they follow two increasingly distinctive curves as $n$ increases. 
Furthermore, the optimal resolutions are approximately linear functions of their corresponding rate-optimal ones, and some simple linear regression fitting can help us understand the coefficients before these rates. We use the exponential decay case as an example. 
Figure \ref{fig:exact_optimal_rate_linear}(b) plots the optimal resolution against the rate $\log(n)$ suggested by Theorem \ref{th:cont}, with a fitted regression line using ordinary least squares. 
From Figure \ref{fig:exact_optimal_rate_linear}(b), the optimal resolution is roughly linear in $\log(n)$, with the regression line $r \approx 1.7 + \log(n)$. These results also further help confirming Theorem \ref{th:cont}. 

\begin{figure}[ht]
	\centering
	\begin{subfigure}{.39\textwidth}
		\centering
		\includegraphics[width=1\linewidth]{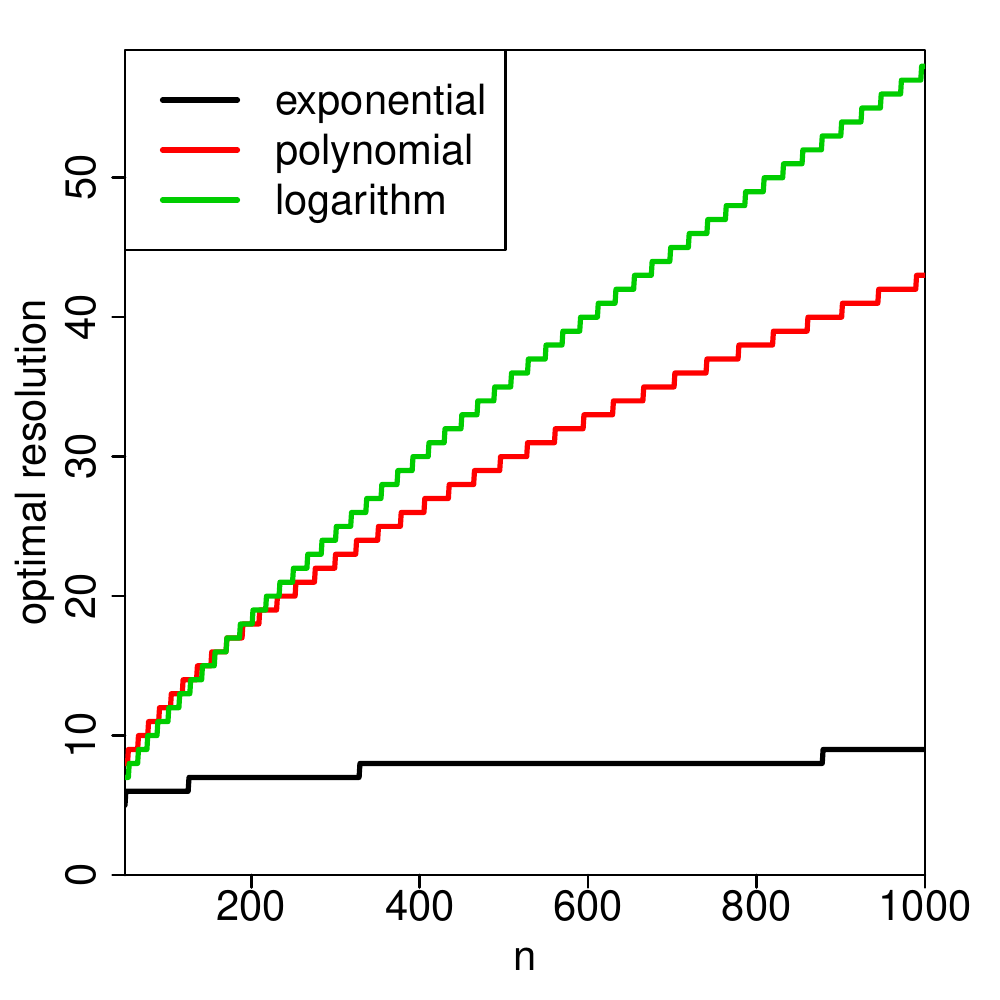}
		\caption{\centering}
	\end{subfigure}%
	\begin{subfigure}{.39\textwidth}
		\centering
		\includegraphics[width=1\linewidth]{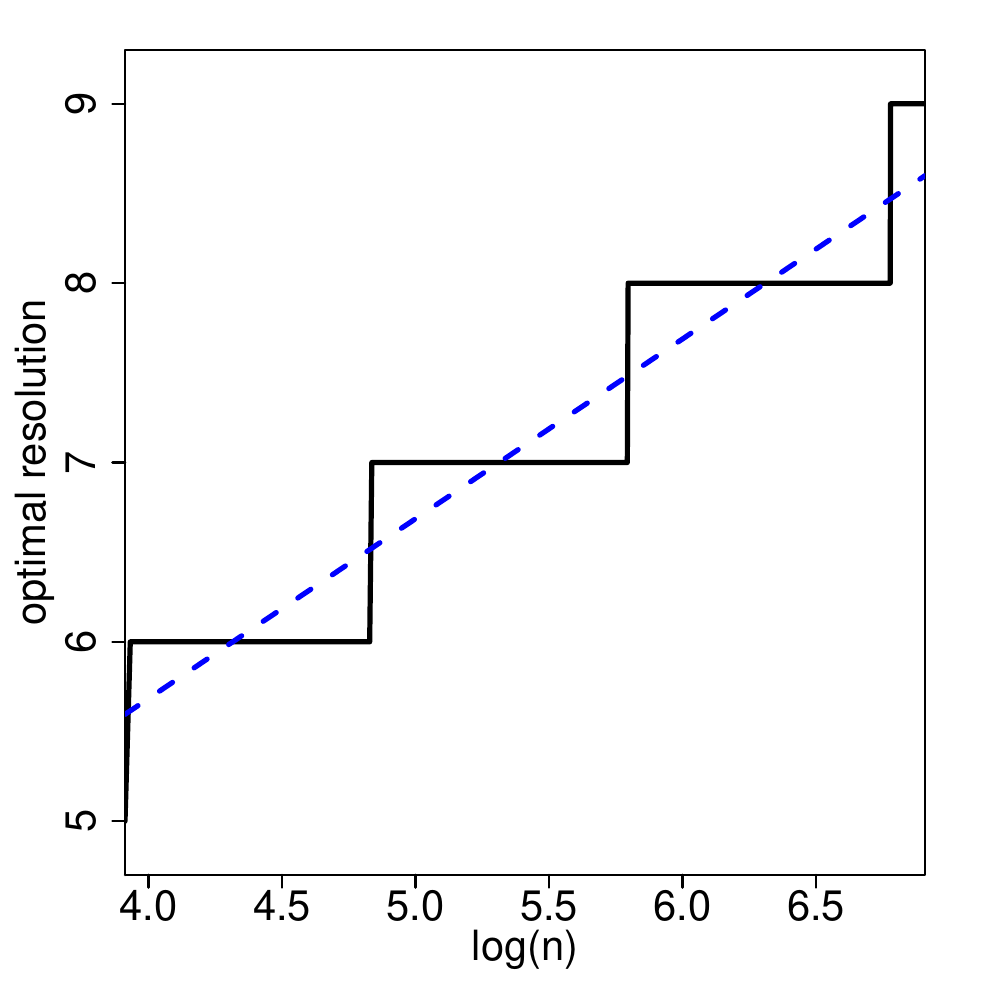}
		\caption{}
	\end{subfigure}
	\caption{
	Optimal resolution as sample size increases, with varying decay rates for the resolution bias, under the linear model with $\tau^2 = 0.5$.
	(a) plots the optimal resolution against the sample size for exponential, polynomial and logarithmic decay rates of the resolution bias. 
	(b) plots the optimal resolution against the rate suggested from Theorem \ref{th:cont} (i.e., $\log(n)$) for the exponential decay resolution bias, with a fitted linear regression line. 
	}\label{fig:exact_optimal_rate_linear}
\end{figure}   

In terms of prediction error, Tables \ref{tab:tau2_half_n50} and \ref{tab:tau2_half_n200} demonstrate that cross validation and the unbiased estimation lead to much smaller prediction error than the information criterion, and cross validation seems to slightly outperform unbiased estimation. 
We observe that under cross validation or unbiased estimation, occasionally 
the average of the standardized prediction error (SPE) can be larger than its $97.5\%$ quantile. 
This is because the estimated resolution can be close to $n$, although with a very small probability, making SPE heavy-tail.
Specifically, Figure \ref{fig:scatterhist_std_pe} 
shows the scatter plot of the SPE using UE versus that using CV, as well as their histograms, from the 500 simulated training sets. 
It shows that the SPE can take very large values but only occasionally, and CV seems to be more robust than UE. 
Under exponential decay, all the estimated resolutions from either CV or UE are below 27; 
under polynomial decay, most estimated resolutions are below 53, while CV has one exception with value 197 and UE has two exceptions both with value 197; 
under logarithmic decay, most estimated resolutions are below 64, while UE has five exceptions with values 196 and 197.
In practice, we can avoid 
the extreme values of SPE
by restricting $\overline{r}_n$, the upper limit of our search for optimal $R_n$, to a smaller number compared to $n$. 

More importantly, 
it is worthwhile to also pay attention to the entire estimated prediction error curve from CV or UE, instead of only focusing on the resolution minimizing it. 
For example, 
Figure \ref{fig:est_pe_one_iteration} shows the logarithm of the estimated prediction error at each resolution using CV and UE from one simulated training set of size 50, where in (a) the data are generated from polynomial decay with intrinsic variance $1/2$ and in (b) the data are generated from exponential decay with zero intrinsic variance. 
The latter case will be discussed shortly. 
From Table \ref{tab:tau2_half_n50} and \ref{tab:tau_0_n50}, 
the corresponding optimal resolutions are respectively 7 and 47. 
Although in both Figure \ref{fig:est_pe_one_iteration}(a) and (b) the estimated resolutions have the same value 47 (i.e., the maximum resolution under search), 
the patterns of the estimated prediction error as a function of the resolution $r$ are very different. 
In particular, the estimated prediction error in (b) monotonically decreases in $r$, while that in (a) shows an approximately U shape excluding the last point with a big drop.

\begin{figure}[ht]
	\centering
	\begin{subfigure}{.33\textwidth}
		\centering 
		\includegraphics[width=1\linewidth]{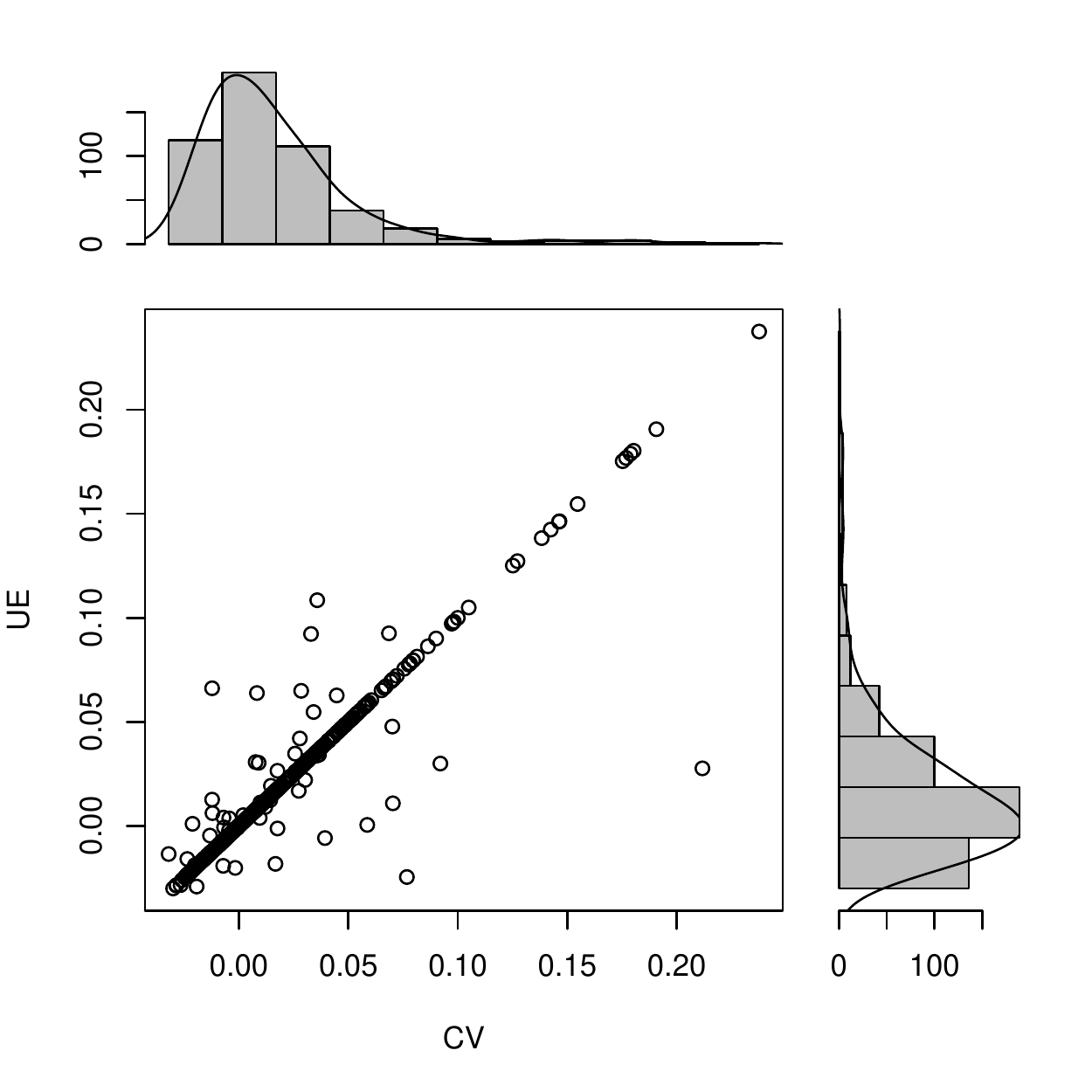}
		\caption{Exponential, $\tau^2=\frac{1}{2}$}
	\end{subfigure}%
	\begin{subfigure}{.33\textwidth}
		\centering
		\includegraphics[width=1\linewidth]{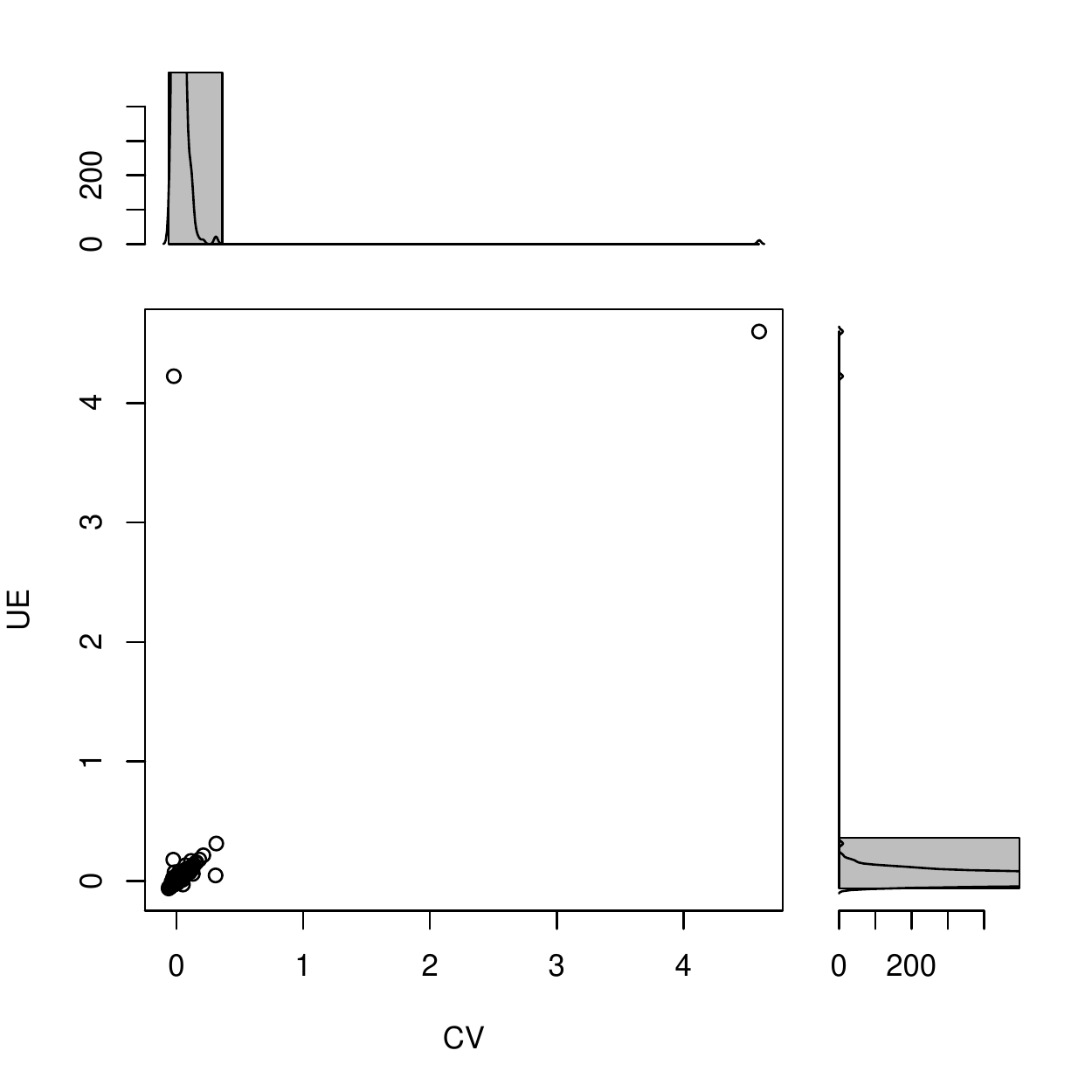}
		\caption{Polynomial, $\tau^2=\frac{1}{2}$}
	\end{subfigure}%
	\begin{subfigure}{.33\textwidth}
		\centering
		\includegraphics[width=1\linewidth]{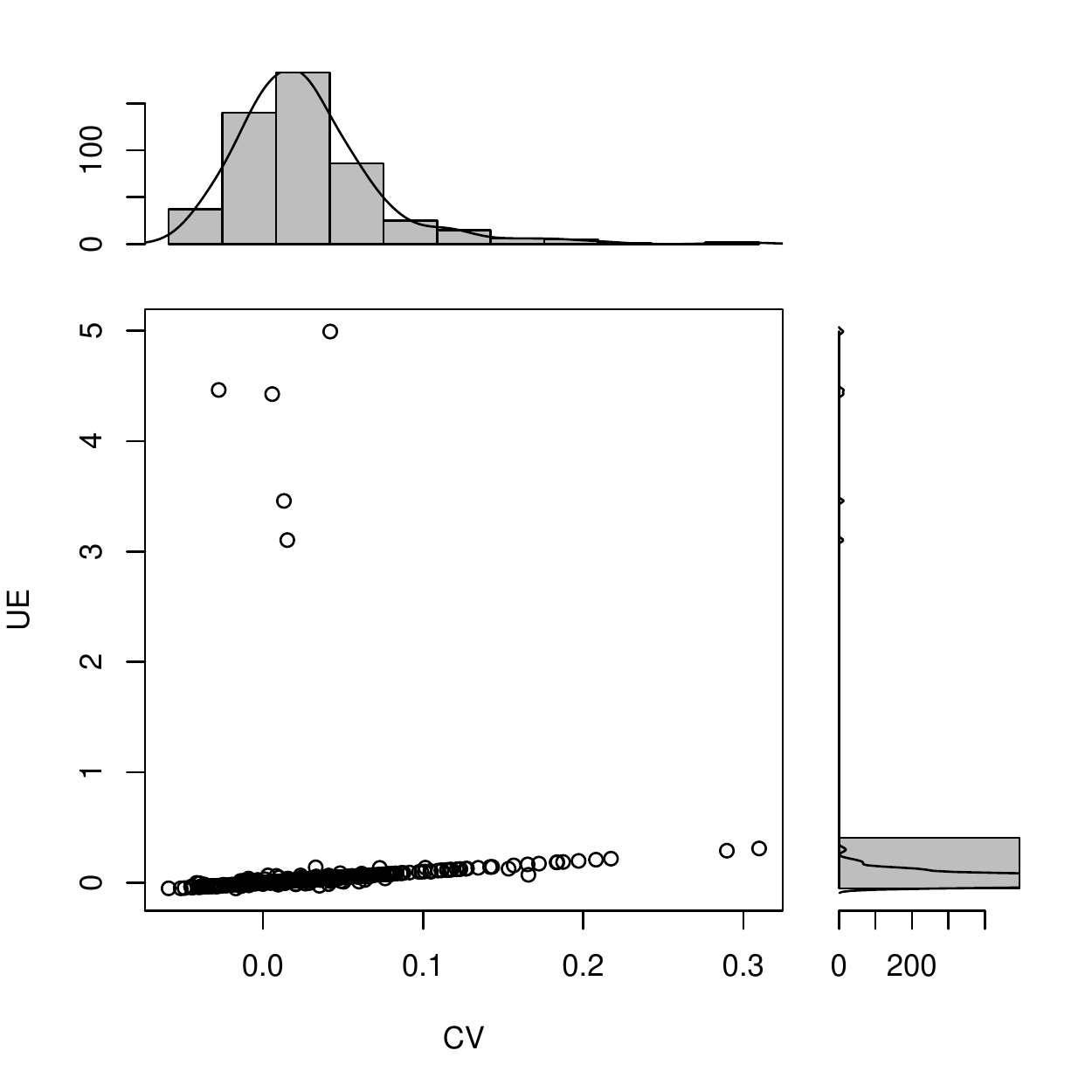}
		\caption{Logarithmic,$\tau^2=\frac{1}{2}$}
	\end{subfigure}
	\caption{Scatter plots of the logarithm of estimated prediction errors from CV and UE, as well as their histograms, from 500 simulated data sets under different decay rates of the resolution bias and $1/2$ intrinsic variance $\tau^2$.}\label{fig:scatterhist_std_pe}
\end{figure}

\begin{figure}[ht]
	\centering
	\begin{subfigure}{.39\textwidth}
		\centering
		\includegraphics[width=1\linewidth]{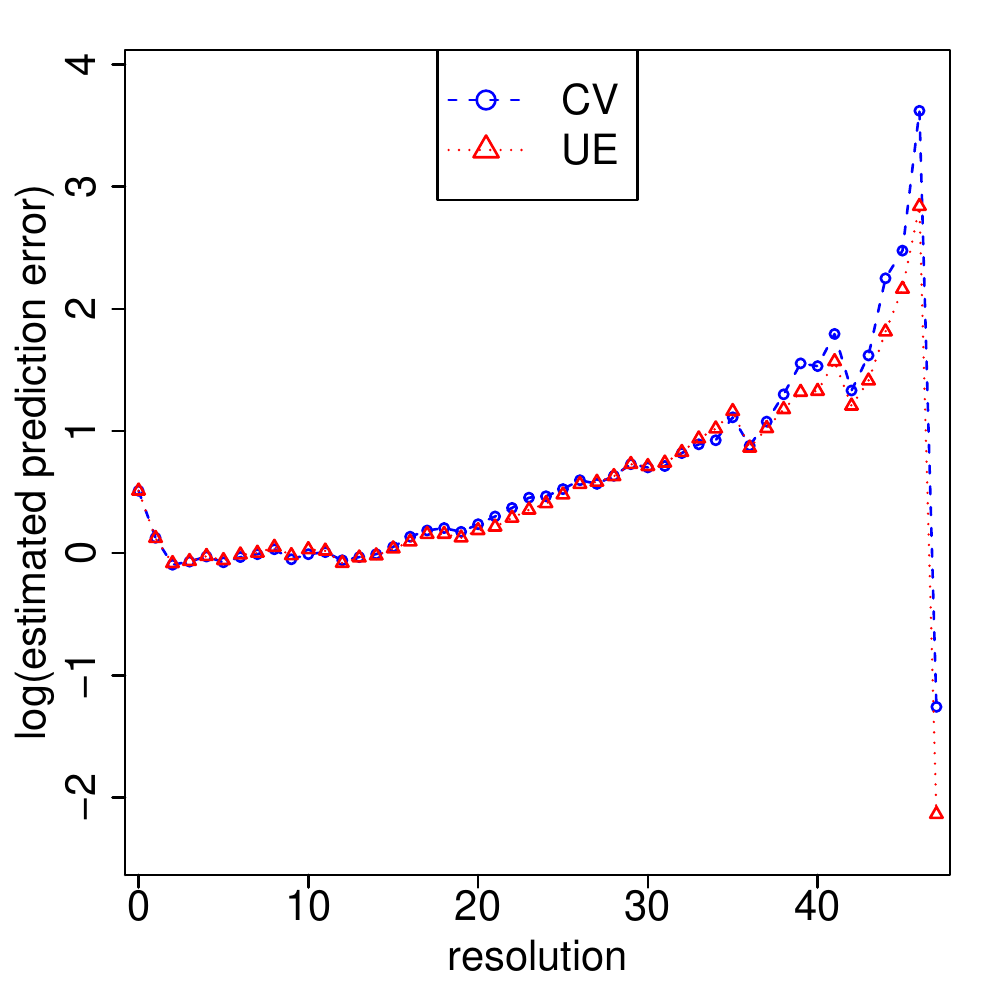}
		\caption{\centering Polynomial, $\tau^2=\frac{1}{2}$}
	\end{subfigure}%
	\begin{subfigure}{.39\textwidth}
		\centering
		\includegraphics[width=1\linewidth]{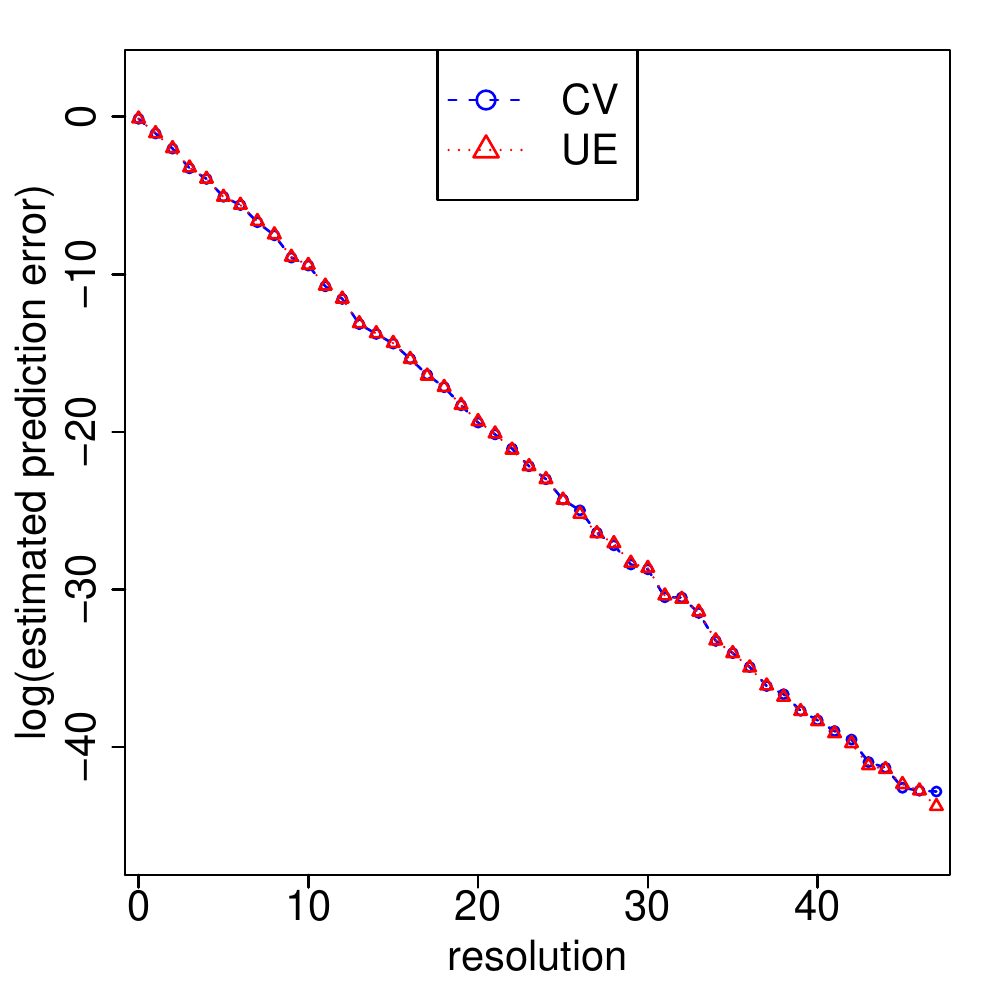}
		\caption{Exponential, $\tau^2=0$}
	\end{subfigure}
	\caption{Logarithms of the estimated prediction errors from CV and UE from one simulated data set, under polynomial decay with $1/2$ intrinsic variance and exponential decay with zero intrinsic variance. 
	}\label{fig:est_pe_one_iteration}
\end{figure}  

\begin{table}[htbp]
    \centering
	\caption{Estimated resolution and prediction error when $n = 50$, with $\tau^2=0$.}\label{tab:tau_0_n50}
	\begin{tabular}{cccccc}
		\toprule
		Type / $r_{\opt}$ / $\PE_n(r_{\opt})$ & Method & $\hat{R}$ & $95\%$ QR & std. $\PE_n(\hat{R})$ &  $95\%$ QR  \\
		\midrule
		Exponential & Oracle & -- & -- & $1.08$ & $[0.12, \  4.38]$ \\
		$r_{\opt} = 47$ & CV & $46$ & $[44, \  47]$ & $1.91$ & $[0.14, \  9.52]$ \\
		$\PE_n(r_{\opt}) = 1.90 \times 10^{-19}$ & UE & $46$ & $[45, \ 47]$ &  $1.21$ & $[0.13, \ 5.08]$ \\
		&  IC & $47$ & $[47, \ 47]$  & $1.09$ & $[ 0.12 , \ 4.38]$\\
		\midrule
		Polynomial & Oracle & -- & -- & $0.99$ & $[0.71, \  1.48]$ \\
		$r_{\opt} = 23$ & CV & $27$ & $[13, \  47]$ & $1.57$ & $[0.75, \  6.20]$ \\
		$\PE_n(r_{\opt}) = 0.0816$ & UE & $28$ & $[13, \ 47]$ &  $1.86$ & $[0.76, \ 8.07]$ \\
		& IC & $47$ & $[46, \ 47]$  & $10.78$ & $[1.46, \ 54.04]$\\
		\midrule
		Logarithmic & Oracle & -- & -- & $1.00$ & $[0.83, \  1.34]$ \\
		$r_{\opt} = 12$ & CV & $15$ & $[4, \  44]$ & $1.51$ & $[0.87, \  6.47]$ \\
		$\PE_n(r_{\opt}) = 0.357$ & UE & $16$ & $[5, \ 47]$ &  $2.20$ & $[0.87, \ 13.60]$ \\
		& IC & $47$ & $[46, \ 47]$  & $21.78$ & $[2.79, \ 109.50]$\\
		\bottomrule
	\end{tabular}
\end{table} 
We remark here that as reviewed in Section~\ref{sec:sieve}, mathematically, the construction of MR is essentially the same as constructing a sieve.
It is therefore no surprise that our findings here echo some findings in the literature of sieve methods. For example, as pointed out by many researchers, 
compared to AIC, 
the cross-validation is often strikingly more effective for sieve methods and for other smoothing problems \citep{Stone1978, Scott1981, Utreras1979, wahba1981, Geman1982}, yet its theoretical properties are not entirely clear.  

The bottom row of Figure \ref{fig:estimate_pred_loss} shows the estimated prediction errors when $\tau^2=0$; everything else is the same as for the top row. Similarly, 
Table \ref{tab:tau_0_n50} is the counterpart of Table~\ref{tab:tau2_half_n50} but with $\tau^2=0$. Compared to the results with $\tau^2=0.5$, we see the general patterns are similar except in the case with exponential decay resolution bias, where the prediction error decreases monotonically with $r$. Hence the larger the resolution the better the prediction, as suggested by Theorem \ref{th:contzero}.  It also (accidentally) makes the IC procedure acceptable in this case because of its preference for the largest possible $r$.

\subsection{Technical details for prediction error estimate in Appendix \ref{sec:simu_example}}\label{sec:simulation_tech} 

To derive an unbiased estimator for $\PE_n(r)$,  we first calculate the expectation of the mean squared error at each resolution $r$. 
By definition, the fitted residual at resolution $r$ is 
$
	\tilde{\bm{Y}} - \tilde{\bm{X}}_{r}\hat{\bm{\theta}}_{r} = 
	\tilde{\bm{Y}} - \bm{H}_r \tilde{\bm{Y}}
	= 
	\left( \bm{I}_r - \bm{H}_r \right)
	(\tilde{\bm{Y}} - \tilde{\bm{X}} \bm{\theta}_{r}^*  ), 
$
where the last equality follows from a property of the projection matrix $\bm{H}_r$. 
By definition,  conditional on $\tilde{\bm{X}}$, 
$\tilde{\bm{Y}} - \tilde{\bm{X}} \bm{\theta}_{r}^* $ follows a multivariate normal distribution with mean zero and covariance matrix $\{\tau^2 + A(r)\} \cdot \bm{I}_n$. 
This implies that $\hat{\sigma}^2$ has the following conditional expectation: 
\begin{align}
	\E_n \left( \hat{\sigma}^2 \big| \tilde{\bm{X}} \right)
	& = \frac{1}{n} 
	\E_n \left(
	\left\| \tilde{\bm{Y}} - \tilde{\bm{X}}\hat{\bm{\theta}}_{r} \right\|_2^2
	\big| \tilde{\bm{X}}
	\right)
	= 
	\frac{1}{n} 
	\E_n \left[
	\left\|
	\left( \bm{I}_r - \bm{H}_r \right)
	(\tilde{\bm{Y}} - \tilde{\bm{X}} \bm{\theta}_{r}^*  )
	\right\|^2
	\big| \tilde{\bm{X}}
	\right]\nonumber
	\\
	& = 
	\frac{1}{n}
	\E_n \left[
	\text{tr} 
	\left\{
	\left( \bm{I}_r - \bm{H}_r \right)
	(\tilde{\bm{Y}} - \tilde{\bm{X}} \bm{\theta}_{r}^*  )
	(\tilde{\bm{Y}} - \tilde{\bm{X}} \bm{\theta}_{r}^*  )^\top 
	\left( \bm{I}_r - \bm{H}_r \right)
	\right\}
	\big| \tilde{\bm{X}}
	\right] \nonumber\\
	& = 
	\frac{1}{n}
	\text{tr} 
	\left\{
	\left( \bm{I}_r - \bm{H}_r \right)
	\cdot
	\E_n \left[
	(\tilde{\bm{Y}} - \tilde{\bm{X}} \bm{\theta}_{r}^*  )
	(\tilde{\bm{Y}} - \tilde{\bm{X}} \bm{\theta}_{r}^*  )^\top 
	\bigg| \tilde{\bm{X}}
	\right]
	\cdot
	\left( \bm{I}_r - \bm{H}_r \right)
	\right\}\nonumber\\
	& = 
	\frac{1}{n}\left[\tau^2+A(r)\right]
	\text{tr} 
	\left\{
	\left( \bm{I}_r - \bm{H}_r \right)
	\right\} = 
	\frac{n-r-1}{n}\left[\tau^2+A(r)\right], \label{eq:mean_mse_linear}
\end{align}
because  $\text{tr}(\bm{H}_r)=r+1$; recall $\dim(\bm{\theta}_r)=r+1.$ Simple algebra then shows $\UE_n(r)$ of \eqref{eq:UE} is unbiased for $\PE_n(r)$ of \eqref{eq:pred_error_linear}.

To study the biases in $\CV_n(r)$ and $\IC_n(r)$ for estimating  $\PE_n(r)$, we first note that from the discussion in Section \ref{sec:cross}, 
\begin{align}\label{eq:mean_cv_linear}
	\E_n\left[\CV_n(r)\right] & = \PE_{n-1}(r) = 
	\left[
	\tau^2 + A(r)
	\right]
	\frac{n(n-3)}{(n-1)(n-r-3)}.
\end{align}
From \eqref{eq:IC} and \eqref{eq:mean_mse_linear}, we have
\begin{align}\label{eq:mean_ic_linear}
    \E_n \left[\IC_n(r) \right]
    = \E_n [\hat{\sigma}_r^2] \cdot \frac{n+2(r+1)}{n}
    =
    \left[\tau^2+A(r)\right]
    \frac{(n-r-1)(n+2r+2)}{n^2}.
\end{align}
Consequently, 
from \eqref{eq:mean_cv_linear}, we have 
\begin{align*}
     \E_n\left[\CV_n(r)\right] - \PE_n(r) 
	= 
	\left[
	\tau^2 + A(r)
	\right]
	\left[
	\frac{1 - 3/n}{(1-1/n)\{1-(r+3)/n\}}
	- 
	\frac{(1+1/n)(1-2/n)}{1-(r+2)/n}
	\right].
\end{align*}
Using the Taylor expansion $(1-x)^{-1}=1+x+O(x^2)$ when $x=o(1)$, we can easily verify that the above expression is $[\tau^2 + A(r)]^2 O(n^{-2})$
when $r=O(1)$. It follows that
\begin{align*}
    \frac{\E_n\left[\CV_n(r)\right]}{\PE_n(r)}
    & = 
    1 + 
    \frac{O(n^{-2})}{\frac{(n+1)(n-2)}{n(n-r-2)}} 
    = 1 + O(n^{-2}).
\end{align*}
Furthermore, when $0 \le r\le n-4$, we have
\begin{align*}
    \frac{\E_n\left[\CV_n(r)\right]}{\PE_n(r)}
    = 
    \frac{n^2(n-3)}{(n^2-1)(n-2)}
    \cdot
    \frac{n-r-2}{n-r-3}
    = 1 + \frac{1}{n-r-3} + o(1).
\end{align*}
Similarly, from \eqref{eq:mean_ic_linear}, when $r=O(1)$, 
\begin{align*}
    & \quad \ \E_n \left[\IC_n(r) \right] - \PE_{n}(r)
     = 
    \left[\tau^2+A(r)\right]
    \left[
    \left(1 - \frac{r+1}{n} \right)
    \left( 1 + \frac{2r+2}{n} \right)
    -
    \frac{(1+1/n)(1-2/n)}{1-(r+2)/n}
    \right]\\
    & = 
    \left[\tau^2+A(r)\right]
    \left[
    \left( 1 - \frac{r+1}{n} + \frac{2r+2}{n} \right)
    - 
    \left( 1 + \frac{1}{n} - \frac{2}{n} + \frac{r+2}{n} \right)
    +
    O(n^{-2})
    \right]\\
    & = 
    \left[\tau^2+A(r)\right] \cdot
    O(n^{-2}),
\end{align*}
which implies that 
\begin{align*}
    \frac{\E_n \left[\IC_n(r) \right]}{\PE_{n}(r)}
    & = 
    1 + 
    \frac{O(n^{-2})}{\frac{(n+1)(n-2)}{n(n-r-2)}} 
    = 1 + O(n^{-2}).
\end{align*}
However, when $r = n - o(n)$, 
\begin{align*}
    \frac{ \E_n \left[\IC_n(r) \right]}{\PE_{n}(r)}
    & = 
    \frac{(n-r-1)(n+2r+2)}{n^2} \cdot \frac{n(n-r-2)}{(n+1)(n-2)}
    \\
    & = 
    \frac{o(n) \cdot O(n)}{n^2} \cdot \frac{n \cdot o(n)}{(n+1)(n-2)}
    = o(1), 
\end{align*}
which means $\IC_n(r)$ tends to grossly underestimate $\PE_n(r)$ when $r$ approaches $n$.

\end{document}